\documentclass[11pt]{amsart} 
\usepackage{amsmath, amssymb, amsthm,
  mathrsfs, verbatim, epsf, graphicx, multirow, color}
\usepackage[all]{xy} \xyoption{poly} \xyoption{arc}

\newcommand\rt{{\bullet\hskip-4pt=}}
\newcommand\rst{\sqcup}
\newcommand{\D}{\partial}
\newcommand{\Q}{\mathbb{Q}}

\newcommand{\C}{\mathbb{C}}
\newcommand{\R}{\mathbb{R}}

\newcommand{\A}{\mathcal{A}}
\newcommand{\B}{\mathcal{B}}
\newcommand{\Li} {{\mathbb L}{\rm i}}
\newcommand{\id}{\mathbb{I}}
\newcommand{\sha}{{\,\amalg\hskip -3.6pt\amalg\,}}
\newcommand{\sgn}{\textrm{sign}}
\def\ba #1\ea{\begin{align} #1 \end{align}}
\def\bas #1\eas{\begin{align*} #1 \end{align*}}
\newcommand{\mcp}{\mathcal{P}_\bullet^{(\star)}(R)}
\newcommand{\mcT}{\mathcal{T}}
\newcommand{\T}{\mathcal{T}^\bullet(R)}
\newcommand{\Tsub}[1]{\mathfrak{T}_{#1}}

\newcommand{\ofv}{\sigma fv}

\newcommand{\bP}{\sigma P}
\newcommand{\st}{\textrm{ s.t. }}
\newcommand{\Hom}{\textrm{Hom}}

\def \fourgon#1#2#3#4{{
\xy
\POS(10,4) \ar@{=} +(-10,0)_#4
\ar@{-} +(0,-10)^#3
\POS(10,-6) \ar@{-} +(-10,0)^#2
\POS(0,4) \ar@{-} +(0,-10)_#1
\POS(0,4) *+{\bullet}
\endxy
}}

\def \fourgondotted#1#2#3#4{{
\xy
\POS(10,4) \ar@{=} +(-10,0)_#4
\ar@{-} +(0,-10)^#3
\POS(10,-6) \ar@{-} +(-10,0)^#2
\POS(0,4) \ar@{.} +(0,-10)_#1
\endxy
}}

\def \fourgona#1#2#3#4{{
\xy
\POS(10,4) \ar@{=} +(-10,0)_#4
\ar@{-} +(0,-10)^#3
\POS(10,-6) \ar@{-} +(-10,0)^#2
\POS(0,4) \ar@{-} +(0,-10)_#1
\POS(0,4) *+{\bullet}
\POS(10,4) \ar@{->} +(-9,-4)
\endxy
}}

\def \fourgonb#1#2#3#4{{
\xy
\POS(10,4) \ar@{=} +(-10,0)_#4
\ar@{-} +(0,-10)^#3
\POS(10,-6) \ar@{-} +(-10,0)^#2
\POS(0,4) \ar@{-} +(0,-10)_#1
\POS(0,4) *+{\bullet}
\POS(10,4) \ar@{->} +(-4,-9)
\endxy
}}

\def \fourgonab#1#2#3#4{{
\xy
\POS(10,4) \ar@{=} +(-10,0)_#4
\ar@{-} +(0,-10)^#3
\POS(10,-6) \ar@{-} +(-10,0)^#2
\POS(0,4) \ar@{-} +(0,-10)_#1
\POS(0,4) *+{\bullet}
\POS(10,4) \ar@{->} +(-9,-4)
\POS(10,4) \ar@{->} +(-4,-9)
\endxy
}}

\def \threegon#1#2#3{{
\xy
\POS(10,4) \ar@{=} +(-10,0)_#3
\ar@{-} +(-5,-8)^#2
\POS(0,4) \ar@{-} +(5,-8)_#1
\POS(0,4) *+{\bullet}
\endxy
}}

\def \threegonarrowc#1#2#3{{
\xy
\POS(10,4) \ar@{=} +(-10,0)_#3
\ar@{-} +(-5,-8)^#2
\POS(0,4) \ar@{-} +(5,-8)_#1
\POS(5,4) \ar@{<-} +(0,-8)
\POS(0,4) *+{\bullet}
\endxy
}}

\def \threegonarrowb#1#2#3{{
\xy
\POS(10,4) \ar@{=} +(-10,0)_#3
\ar@{-} +(-5,-8)^#2
\POS(0,4) \ar@{-} +(5,-8)_#1
\POS(7.5,0) \ar@{<-} +(-7.5,4)
\POS(0,4) *+{\bullet}
\endxy
}}

\def \threegonarrowa#1#2#3{{
\xy
\POS(10,4) \ar@{=} +(-10,0)_#3
\ar@{-} +(-5,-8)^#2
\POS(0,4) \ar@{-} +(5,-8)_#1
\POS(2.5,0) \ar@{<-} +(7.5,4)
\POS(0,4) *+{\bullet}
\endxy
}}

\def \twogon#1#2{\hskip 5pt\xy
\POS(10,0) \ar@{=} +(-10,0)_#2
\POS(0,0)*{};
\POS(0,0) *+{\bullet}
\POS(10,0)*{};
**\crv{(5,-5)};
\POS(5,-5) *+{\scriptstyle #1}
\endxy \hskip 5pt}

\def \twogonright#1#2{\hskip 5pt\xy
\POS(10,0) \ar@{=} +(-10,0)_#2
\POS(0,0)*{};
\POS(10,0) *+{\bullet}
\POS(10,0)*{};
**\crv{(5,-5)};
\POS(5,-5) *+{\scriptstyle #1}
\endxy \hskip 5pt}

\newtheorem{thm}{Theorem}
\newtheorem{lem}[thm]{Lemma}
\newtheorem{cor}[thm]{Corollary}
\newtheorem{prop}[thm]{Proposition}

\theoremstyle{remark}
\newtheorem{eg}[thm]{Example}

\theoremstyle{definition}
\newtheorem{dfn}[thm]{Definition}
\newtheorem{rem}[thm]{Remark}

\begin{document}
\author{Susama Agarwala}
\address{University of Hamburg \\
Mathematics Department \\
Bundestra{\ss}e, 55 \\
20146, Hamburg \\
Germany }

\title{Dihedral symmetries of multiple polylogarithms}
\maketitle 

\begin{abstract} 
This paper finds relationships between multiple polylogarithms with a
dihedral group action on the arguments. I generalize the combinatorics
developed in Gangl, Goncharov and Levin's $R$-deco polygon
representation of multiple polylogarithms to find these
relations. These relations between $R$-deco polygons, and between
$R$-deco polygons and iterated integrals, can only be defined up to a
primitive co-ideal.

\end{abstract}


\tableofcontents

This paper studies the relationships between multiple polylogarithms
that differ by a dihedral permutation on their arguments. Goncharov
has shown \cite{Gonch98} dihedral relations between multiple zeta values,
and has conjectured such a relation on the level of multiple
polylogarithms. To study the dihedral symmetries of multiple
polylogarithms, I use a technique developed by Gangl, Goncharov and
Levin \cite{GGL05}, that represents multiple polylogarithms as
decorated rooted oriented polygons, $R$-deco polygons. Goncharov
\cite{Gonch01} has shown a Hopf algebra structure underlying multiple
polylogarithms. The authors of \cite{GGL05} give a Hopf algebra
structure to the $R$-deco polygons and construct a coalgebra
homomorphism between their Hopf algebra of polygons and the Hopf
algebra of multiple polylogarithms. The $R$-deco polygon structure of
multiple polygons has recently become interesting objects in
physics. These polygons are used to define symbols of multiple
polylogarithms \cite{DGR11}. These symbols and their Hopf algebraic
structure have important applications in understanding amplitudes in super Yang-Mills theories in 4 dimensions \cite{Duhr12,
  GSVV10}.

Multiple polylogarithms are interesting number theoretical objects in
their own right. As a nested sum they are written \bas \Li_{n_1,
  \ldots ,n_r} (x_1, \ldots x_r) = \sum_{0 < k_1 < \ldots < k_r}
\frac{x_1^{k_1}\ldots x_r^{k_r}}{k_1^{n_1}\ldots k_r^{n_r}} \;.\eas
The multiple polylogarithm, evaluated at $x_i = 1$ gives the multiple
zeta value $\zeta(n_1, \ldots n_r)$. The weight of the multiple
polylogarithm is $w= \sum_{i=1}^rn_i$ and its depth is $r$. In
\cite{Gonch01}, this multiple sum is written in terms of Chen's
iterated integrals. Thus multiple polylogarithms inherit the bar
complex defined on iterated integrals \cite{chen}. There is a motivic
generalization of Chen's iterated integrals. Bloch and Kriz
\cite{BlochKriz} define a Hopf algebra of algebraic cycles,
$\chi_{Mot}$, over a field $F$ formed by taking the $0^{th}$
cohomology of a bar complex based on a differential graded algebra
(DGA) associated to the cycles. In \cite{GGL05}, the authors determine
that there are elements of $\chi_{Mot}$ that correspond to multiple
polylogarithms.

Iterated integrals also make their appearance in physics. Goncharov
\cite{Gonch05} shows a relationship between the Hopf algebra of
multiple polylogarithms, via iterated integrals, and the Hopf algebra
of rooted trees developed by Connes and Kreimer for renormalizing
Feynman integrals \cite{CK98}. Brown \cite{Brown08} develops a method
for evaluating Feynman integrals, under Schwinger parametrization, in
terms of iterated integrals, though the arguments for these iterated
integrals are more complicated than those for multiple polylogs. In
short, there is a lot of activity suggesting a close tie between
multiple polylogarithms and the calculations that appear in various
quantum field theories.

This paper studies multiple polylogarithms from a combinatorial point
of view, primarily on the level of $R$-deco polygons. Let $D_{2r}$ be
the dihedral group on $r$ elements, \bas D_{2r} = \langle \sigma, \tau
| \tau^2 = \sigma^r = 1, \sigma \tau = \tau \sigma^{-1} \rangle
\;.\eas In this paper, I study the relationship between the multiple
polylogarithms \bas \Li_{n_1, \ldots ,n_r}(x_1, \ldots, x_r) \quad
\text{and} \quad \Li_{g(n_1, \ldots ,n_r)}(g (x_1, \ldots, x_r)) \eas
for any $g \in D_{2r}$. Section one of this paper generalizes the
algorithm defined in \cite{GGL05} for assigning decorated trees to
multiple polylogarithms. There the authors associate to each multiple
polylogarithm an $R$-deco polygon, and a family of dissections to each
polygon. The differential structure of the iterated integral
associated to each multiple polylogarithm defines a dual tree
structure to the dissections, and a Hopf algebra structure to the $R$-
deco polygons. Each polygon is mapped to an element of the bar complex
on the algebra of $R$-deco polygons, $\mcp$. I introduce a
generalization of a rooted tree, called a multi-rooted tree. I show
that if the dual trees associated to dissections generate a Hopf
algebra, the trees dual to single dissections define a differential on
the $\mcp$. Combinatorial relationships between these different tree
structures allow me to compare multiple polylogarithms of weight $r$
under the action of the dihedral group $D_{2r}$. Section 2 of this
paper calculates the action of $\tau$ and $\sigma$ in $D_{2r}$ on
$\mcp$. Contrary to a conjecture by Gangl and Levin, I find that the
difference between the bar element associated to an $R$-deco polygon
and its image under the action of the dihedral group cannot completely
be expressed as a sum of products of bar elements associated to
subpolygons. On the level of $R$-deco polygons, this result holds up
to a primitive coideal.


\section*{Acknowledgements}

The author would like to thank Herbert Gangl for significant guidance
and inspiration during the writing of this paper.

\section{The various Hopf algebras of $R$-deco polygons \label{Hopfalgssect}}

This paper studies multiple polylogarithms by studying the iterated
integral associated to them. Let $w = \sum_i^n k_i$ be the weight of
the iterated integral \begin{multline*} I_{k_1, \ldots,
    k_n}(\gamma)(0; x_1 \ldots x_n; x_{n+1}) =
  \\ \int_\gamma \frac{dt_1}{t_1-x_1} \wedge
  \bigwedge_{i=1}^{k_1-1}\frac{dt_{1_i}}{t_{1_i}} \ldots \wedge
  \frac{dt_n}{t_n-x_n}\bigwedge_{i=1}^{k_n-1}\frac{dt_{n_i}}{t_{n_i}}
  \end{multline*} with $\gamma$ a path in $\C$ such that $\gamma(0) = 0$,
  $\gamma(1) = x_{n+1}$. The value of this integral depends on the
homotopy class of $\gamma$ \cite{chen}. If $\gamma$ is a straight path
in the real line, $\gamma(t) \in \R$, from $x_i$ to $x_{i+1} \in \R$ I
drop the notation $\gamma$.  These iterated integrals can be related
to multiple logarithms \bas (-1)^nI_{\underbrace{1, \ldots, 1}_{n
    \textrm{ times}}}(0; x_1, \ldots ,x_n; x_{n+1}) =
\Li_{\underbrace{1, \ldots ,1}_{n \textrm{ times}}} (\frac{x_2}{x_1},
\ldots ,\frac{x_n}{x_{n-1}} , \frac{x_{n+1}}{x_n})\;. \eas These
iterated integrals live in the $0^{th}$ cohomology of the associated
bar complex defined by Chen. The general class of iterated integrals,
$I_\bullet(R)$ have a Hopf algebra structure, as shown by Goncharov in
\cite{Gonch98}. The space of iterated integrals can be endowed with an
algebra structure, under path composition. In \cite{Gonch05}, the
author further show that these iterated integrals have a motivic
counterpart, $I^\mathcal{M}(0, x_1, \ldots , x_{n+1})$ with $x_i \in
F$, for a field $F$, that is an element of the fundamental motivic
Hopf algebra over $F$. The natural differential on the iterated
integrals defines a bar complex on algebra of iterated
integrals. Those iterated integrals corresponding to multiple
polylogarithms live in the $0^\textrm{th}$ cocycle of this bar
complex.

In \cite{GGL05}, Gangl, Goncharov and Levin associate to each multiple
polylogarithm an $R$-deco polygon. For instance, the integral \bas
I_{\underbrace{1, \ldots, 1}_{n \textrm{ times}}}(0; x_1, \ldots ,x_n;
x_{n+1}) \eas is associated to an oriented $n+1$-gon with sides
labeled from $x_1$ to $x_{n+1}$. Note that in this case, $x_i \neq 0$
by construction. They define a differential on the algebra of $R$-deco
polygons, that mimics the differential on iterated integrals. This
defines a bar complex on $R$-deco polygons. The authors of
loc. cit. associate a bar element to each polygon. They define a
family of dual trees to the $R$-deco polygons which induce a Hopf
algebra structure on the algebra of $R$-deco polygons. This translates
to a Hopf algebra sitting in the the $0^{\textrm{th}}$ cocycle of the
bar complex of $R$-deco polygons.  In this paper, I call this
$B_{\phi_2}$. This structure on the $R$-deco polygons is compatible
with the parallel structure on iterated integrals in that there is a a
coalgebra homomorphism from this to $I(R)$, \ba \Phi:
\Lambda(B_{\phi_2}) & \rightarrow I(R) \; \label{GGLmap}.\ea
Specifically, if $0 \not \in R$, and the polygon $P$ has sides labeled
$\{x_1, \ldots, x_{n+1}\}$ then \bas \Phi(\Tsub{\phi_2}(P)) =
I_{\underbrace{1, \ldots, 1}_{n \textrm{ times}}}(0; x_1, \ldots x_n ;
x_{n+1}) \;.\eas Relating dihedral symmetries of multiple
polylogarithms can be simplified to a combinatorial problem on the
dihedral symmetries of decorated polygons.

This section generalizes the above construction. I
define a class of Hopf algebras associated to these decorated polygons
that are useful in solving the combinatorics of how polylogarithms
vary as their order of the arguments are changed. Subsection
\ref{randbar} gives a definition of $R$-deco polygons and their
dissections, the vector space they generate, $V_\bullet(R)$, its
exterior product algebra $\mcp$, and the associated bar
complexes. Subsection \ref{treesection} defines the Hopf algebra of
multi-rooted trees, $\T$ , and the linearization map $\Lambda$. I show
that $\Lambda$ is a bialgebra homomorphism from $\T$ to the
commutative Hopf algebra of words on $R$-deco polygons. Subsection
\ref{Hopftrees} introduces a family of dissection compatible Hopf
algebras of $\T$ generated by sets associated to dissections of
$R$-deco polygons. Subsection \ref{hopftodif}, shows that these
generating sets define differentials on $\mcp$. It also introduces a
family of almost compatible algebra. Finally, subsection \ref{props}
defines a relationship between the linearizations of dissection
compatible Hopf algebras and almost compatible algebras that is useful
for the calculations in section \ref{permutesection}. Furthermore, I
show that the linearization of the latter also gives rise a Hopf
algebra.

\subsection{Bar constructions on $R$-deco polygons\label{randbar}}
Let $R$ be a set. Usually, $R$ is the set underlying a field.

\begin{dfn}
Let $P_n$ be the convex oriented polygon with $n+1\geq 2 $ sides, with
sides labeled by elements in $R$. One of those sides is a
distinguished side, called a root side. One of the endpoints of the
root side is marked as the first vertex. Orient $P_n$ by starting at
the first vertex and ending at the root side. The polygon $P_n$ is an
$R$-deco polygon, as defined in \cite{GGL05}.
\end{dfn}

In this paper, I draw polygons to be oriented
counterclockwise. I sometimes specify a polygon in terms of its
labels, proceeding counterclockwise and ending with the root
side. Therefore,

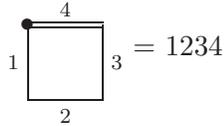
\xy
\hskip 150pt \hskip 10pt
\POS(0,4) *+{\bullet}
\POS(10,4) \ar@{=} +(-10,0)_4
\ar@{-} +(0,-10)^3
\POS(10,-6) \ar@{-} +(-10,0)^2
\POS(0,4) \ar@{-} +(0,-10)_1
\hskip 40pt
= $1234$
\endxy

The $R$-deco polygons generate a vector space.

\begin{dfn}
  Let $V_{\bullet}(R)$ be the graded vector space over $\Q$ generated
  by $R$-deco polygons. Let $V_n(R)$ be the vector space over $\Q$
  generated by $R$ deco $n+1$-gons, with $n>0$ and $V_0(R)$ identified with
  $\Q$  \bas V_\bullet(R) =
  \Q<\{1, P | P \textrm{ is a } R\textrm{-deco polygon} \}> =
  \oplus_{n=0}^\infty V_n(R)\quad ; \quad V_0(R) = \Q .\eas
\end{dfn}

The \emph{weight} of an element in $V_n(R)$ is $n$.

An $R$-deco polygon can be equipped with arrows, as in
\cite{GGL05}. An arrow of a polygon is drawn from a vertex of a
polygon to a side of a polygon. It divides the interior of the polygon
into two regions.  A trivial arrow of a polygon $P$ goes from a vertex
to an adjacent side. A non-trivial arrow of $P$ is an arrow that does
not end on a side adjacent to its starting vertex. Two arrows are said
to be non-intersecting if they share no points in common other than
possibly the starting vertex. 

\begin{dfn}
Arrows of $P_n$ are defined by their starting vertex and ending
edge. Write ${_i}\alpha_j$
for an dissecting arrow of the polygons $P$ that starts at the
$i^{th}$ vertex and ends on the $j^{th}$ edge. For non-trivial arrows, $j\neq i, i-1 \mod (n+1)$. Call
$_i\alpha_j$ a backwards arrow if $j <i$. Otherwise it is a forwards
arrow. \label{arrowdef}
\end{dfn}

\begin{eg}
The arrow $_1\alpha_4$ is a trivial arrow in the first polygon
below. In the second polygon, $_2\alpha_4$ and $_4\alpha_2$ are
non-intersecting, non-trivial arrows.

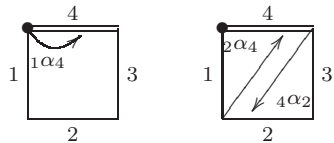
\begin{center}
\xy
\hskip 120pt 
\POS(0,4) *+{\bullet}
\POS(12,4) \ar@{=} +(-12,0)_4
\ar@{-} +(0,-12)^3
\POS(12,-8) \ar@{-} +(-12,0)^2
\POS(0,4) \ar@{-} +(0,-12)_1
{\ar_{_1\alpha_4}@{->}@/_/ (0,4); (7,3)}
\hskip 70pt

\POS(0,4) *+{\bullet}
\POS(12,4) \ar@{=} +(-12,0)_4
\ar@{-} +(0,-12)^3
\POS(12,-8) \ar@{-} +(-12,0)^2
\POS(0,4) \ar@{-} +(0,-12)_1
\POS(0,-8) \ar^>>>>{{_2\alpha_4}}@{->} +(8,11)
\POS(12,4) \ar^>>>>{_4\alpha_2}@{->} +(-8,-11)

\endxy

\end{center}

\end{eg}

Regions associated to dissection arrows can be viewed as polygons in
their own right. If $\alpha$, is an arrow of $P$, contracting $\alpha$
to a point results in a set of two polygons $\{P_\alpha, Q_\alpha\}$
associated to the two regions of $P$ as
follows. \begin{enumerate} \item The labels of the sides and the
  orientations of $P_\alpha$ and $Q_\alpha$ are inherited from
  $P$.  \item If $\alpha$ lands on a non-root side of $P$, then the
  subpolygon corresponding to the region that contains the root side
  of the original polygon inherits the root, and the side that
  $\alpha$ lands on becomes the new root for the other
  subpolygon. \item If $\alpha$ ends on the root side of $P$, then
  both subpolygons inherit the original root side as their root. See
  example \ref{dissecteg}. \end{enumerate} Under this construction,
dissection preserves the polygon weight. If $\alpha$ is a trivial
arrow, then one of the resulting subpolygons is the trivial
subpolygon, $1 \in V_0(R) = \Q$ while the other subpolygon is the
original polygon. For this reason, for most of the discussion in this
paper, we ignore the presence of trivial arrows.

\begin{dfn}
A dissection $d$ of $P$ is a set of non-intersecting
arrows of $P$. Denote $D(P)$ as the set of dissections of the polygon
$P$, including the trivial dissection (no arrows).  The cardinality of a
dissection, $|d|$ is the number of non-trivial arrows in
$d$. \end{dfn}

\begin{eg}\label{dissecteg}
For example, consider $P = 123456$, and dissection $d = \{\alpha,
\beta, \gamma \}$ as drawn. Contracting along the arrows in $d$ gives the a set
of $4$ polygons.

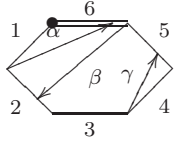
\bas \xy
\POS(0,4) *+{\bullet}
\POS(10,4) \ar@{=} +(-10,0)_6
\ar@{-} +(6,-6)^5
\POS(16,-2)\ar@{-} +(-6,-6)^4
\POS(0,-8)\ar@{-} +(10,0)_3
\ar@{-} +(-6,6)^2
\POS(0,4)\ar@{-} +(-6,-6)_1
\POS(-6,-2) \ar^{\alpha}@{->} +(14,6)
\POS(10,4) \ar^{\beta}@{->} +(-12,-10)
\POS(10,-8) \ar^{\gamma}@{->} +(4,8)
\endxy \eas

The arrow $\beta$ is an arrow not ending on the root side of
$P$. Contracting along $\beta$ gives the pair of polygons

\bas \{ P_\beta  =\threegon{1}{2}{6},  Q_\beta = \fourgon{3}{4}{5}{2} \} \;.\eas 

One can think of $\alpha$ now as an arrow in $P_\beta = 126$. It ends
on the root side of $P$. Contracting along both arrow $\alpha$ and
$\beta$ gives the three polygons

\bas 
\{\twogon{1}{6}, \twogon{2}{6}, \fourgon{3}{4}{5}{2} \} \;. 
\eas

Finally, consider $\gamma$ as an arrow in $Q_\beta = 3452$. It does
not end on the root side of either $Q_\beta$ or $P$. Contracting all arrows in $d$ gives 

\bas \{\twogon{1}{6}, \twogon{2}{6}, \threegon{3}{5}{2}, \twogon{4}{5} \} \;. 
\eas

The order of contracting arrows in $d$ does not affect the
set of polygons associated to it.
\end{eg}

The polygons $P_\alpha$ and $Q_\alpha$ above are called the polygons
associated to the dissecting arrow $\alpha$. If $d \in D(P)$ is a
dissection with $i$ arrows, there is a set of $i+1$ subpolygons,
$\{P_0, \ldots P_i\}$ associated to the dissection $d$, formed by
contracting the arrows in $d$. Since dissection preserves weight, if each
$P_j \in V_{n_j}(R)$, and $P \in V_n(R)$, then $\sum_{j=0}^i n_j =
n$. Two polygons $P_i$ and $P_j$ associated to a dissection are
\emph{adjacent} if regions they correspond to share a dissecting arrow
as a boundary.

When discussing polygons associated to dissections, it is useful to
label the regions associated to dissection consisting of a single
arrow. For $\alpha = d \in D(P)$, the subpolygons associated to $\alpha$
are sometimes referred to as the \emph{root polygon}, $P_\alpha^\rt$,
which is the subpolygon that contains the root side and first vertex
of $P$, and the \emph{cut off polygon}, $P_\alpha^\rst$, which is the
other subpolygon. At other times, it is convenient to consider whether
the subpolygon lies to the left or the right of the arrow, as
determined by the orientation of the arrow. In this
case, the \emph{left polygon} is indicated $P_\alpha^l$ and the
\emph{right polygon} is indicated $P_\alpha^r$. Notice that if
$\alpha$ is a forwards arrow, $P_\alpha^l = P_\alpha^\rt$. If it is a
backwards arrow, $P_\alpha^l = P_\alpha^\rst$. In example \ref{dissecteg},
since $\beta$ is a backwards arrow, \bas P_\beta^l = P_\beta^\rst=
 \fourgon{3}{4}{5}{2} \; , \quad P_\beta^r =
P_\beta^\rt = \threegon{1}{2}{6} \;. \eas

\begin{dfn}
Let $\mcp$ be the exterior product algebra of
$V_\bullet(R)$. It is bigraded, the subscript $\bullet$ corresponds to
the weight, or Adams grading, of the vector space $V_\bullet(R)$, and
the superscript $(\star)$ corresponds to the exterior product grading,
also referred to as the degree.
\end{dfn}

The algebra $\mcp$ can be endowed with a degree 1
differential operator to form a DGA $(\mcp,
\D)$. There are several such operators on this algebra, which I
discuss in section \ref{hopftodif}.  I consider the bar constructions
associated to each DGAs, $B_\D(\mcp)$.

\begin{dfn}
Let $(\A, \D)$ be a DGA with $\A$ a connected graded exterior product
algebra, and $\D$ a degree 1 differential operator. The bar
construction $B_\D(\A)$ associated to $(\A_\bullet, \D)$ is the
reduced tensor algebra $\bar{T}(\A_\bullet) = \oplus_{i=0}^\infty
\A_{\geq 1}^{|i}$, commutative under the shuffle product, $\sha$,
  with tensor symbol denoted by $|$. The bicomplex structure of
  $B_\D(\A)$ is given by the differential operators $D_1$ and
  $D_2$. \end{dfn}

The coproduct on $B_\D(\A)$ is induced from the deconcatenation
coproduct on $\bar{T}(\A)$ \ba \Delta [a_1| \ldots | a_n]
= \sum_{i = 0}^n [a_1| \ldots | a_i] \otimes [a_{i+1}|\ldots |a_n] \;
.\label{barcoprod} \ea It is compatible with the shuffle product on
$T(\A)$.

In this paper, I consider $\A_\bullet = \mcp$. Given a
differential operator $\D$, the bar construction
$B_\D(\mcp)$ is generated by terms of the form
$ [a_1| \ldots | a_n] $ where each $a_i \in
\mathcal{P}_\bullet^{(k_i)}(R)$ is homogeneous in the exterior product
grading of degree $k_i$.

\begin{enumerate}
\item Define $D_1 : \; \mcp^{|n} \rightarrow \mcp^{|n-1}$ be the operator
  defined \bas D_1([a_1| \ldots | a_n]) = \sum_{i=1}^{n-1}
 - (-1)^{\sum_{j\leq i}(\textrm{deg }a_j -1)} [a_1| \ldots| a_i \wedge a_{i+1}|
    \ldots | a_n] \; .\eas

\item Define $D_2: \; \mcp^{|n} \rightarrow \mcp^{|n}$ be the operator defined
  \bas D_2([a_1| \ldots | a_{n}]) = \sum_{j=1}^n
  (-1)^{\sum_{k < j}(\textrm{deg }a_k - 1)} [a_1| \ldots| \D a_j| \ldots
    a_n] \; . \eas
\end{enumerate}

Since $D_1$ does not involve the differential defining the DGA, this
differential is the same for all $B_\D(\mcp)$. If $\D$ and $\D'$ are
different differential operators on $\mcp$, the
differential $D_2$ is different on $B_{\D}(\mcp)$ and
$B_{\D'}(\mcp)$.

\begin{rem} 
The bar construction defined in this paper is different that the one
defined in \cite{GGL05}, specifically they differ by the overall sign
of $D_1$. The objects in this paper have different weights than those considered in \cite{LodayVallette}, Chapter 2, section 2.2, otherwise, the construction in this paper and Loday and Valette agrees. It is worth noting that this bar
construction also differs from that of Bloch and Kriz
\cite{BlochKriz}, where the shuffle product and coproduct have a very
different sign convention.

\end{rem}

\subsection{Multi-rooted trees\label{treesection}}
In this paper, I define several Hopf algebras associated to
the vector space of $R$ deco polygons $V_\bullet(R)$. These are
defined by introducing dual tree structures to polygons and their
dissections. In this subsection, I define these trees.

\begin{dfn}
A tree is a finite contractible graph with oriented edges. Vertexes
with all edges flowing away from them are called roots. Vertexes with
all edges flowing into it are called leaves. A tree may have many
roots, in which case is called a multi-rooted tree. If a tree has a
single vertex, that vertex is both a root and a leaf.
\end{dfn}

Unlike for single rooted trees, leaves on multi-rooted trees can have
multiple edges coming into them.

\begin{rem}
In this paper, root vertexes are marked by a circle. I do not
explicitly indicate the orientation of the edges, and
leave it to be assumed from the pictures. Generally, root vertexes are
drawn at the top of the tree, while the edges flow down.
\end{rem}

Let $\T$ be the augmented bialgebra over $\Q$ of multi-rooted
non-planar trees with vertexes decorated by $R$-deco polygons.  As
with trees, a multi-rooted tree $T \in \T$ induces a \emph{partial
  order} on its vertexes. A \emph{path} in $T$ from the vertex $v_1$
in $T$ to $v_2$ in $T$, is a linear subtree with $v_1$ as a root and
$v_2$ as leaf vertex, with orientation inherited from $T$. If $v_1$
and $v_2$ are two vertexes of a tree $T$, \bas v_1 \prec v_2 \textrm{
  in }T \iff \exists \textrm{ a path in }T \textrm{ from } v_1
\textrm{ to }v_2\;.\eas A \emph{linear order} of $T$ is a total
ordering of the vertexes of $T$ that respects the partial order.

The algebra structure of $\T$ is given as follows. It is graded by
number of vertexes in the tree \bas \T = \bigoplus_{n=0}^{\infty}
\mcT^n(R) = \Q\langle T | T \text{ has $n$ vertexes } \rangle \quad ;
\quad \mcT^0(R) = \Q \; . \eas The unit is the empty tree, \bas
\id_{\T} = T_\emptyset \;.\eas 

The sum of two trees $T_1$ and $T_2$ is formal.

The algebra $\T$ is a commutative algebra with the product
of trees being the disjoint union of trees, or a forest.

\begin{dfn}
For a tree $T \in \T$, let $c$ be a non-empty subset of edges of $T$, and $
\{ t_1, \ldots t_k\}$ be the set of trees formed by removing the edges in
$c$. The subset $c$ is a \emph{proper admissible cut} of $T$ if, for any
individual $t_i$, the edges of $c$ that have endpoints in $t_i$ either all
flow into $t_i$ or all flow from $t_i$. 
\end{dfn}

\begin{eg}
For example for the tree $T= 
\xy 
\POS(0,4) *+{\bullet} *\cir{}
\POS(3,4) *+{A} 
\POS(0,4) \ar@{-}_\alpha + (-4, -8) 
\POS(-4, -4) *+{\bullet} 
\POS(-7, -4) *+{B} 
\POS(0,4) \ar@{-}_\beta + (4, -8)
\POS(4, -4) *+{\bullet} 
\POS(8,4) \ar@{-}^\gamma + (-4, -8) 
\POS(7, -4) *+{C}
\POS(8, 4) *+{\bullet} *\cir{} 
\POS(11,4) *+{D} \endxy$ the set
$\{\alpha,\gamma\}$ is not a proper admissible cut, but the set $\{
\alpha,\beta\}$ is.
\end{eg}

Let $c$ be a proper admissible cut of $\T$, and $\{t_1,\ldots, t_k\}$ the
set of subtrees of $T$ formed by removing the edges in $c$ from
$T$. The definition of a proper admissible cut partitions the trees
$\{t_1,\ldots, t_k\}$ by whether they are connected to the edges in $c$
by terminal vertexes, or initial vertexes. This partitions the set of
subtrees in two, the set $\{t_{l_1}, \ldots, t_{l_n}\}$ of subtrees
that elements of $c$ have at most a terminal vertex in $t_{l_i}$, and
the set $\{t_{r_1}, \ldots, t_{r_m}\}$ of subtrees such that elements of $c$
have at most an initial vertex in $t_{r_i}$.

\begin{dfn}
\begin{itemize} 
\item The leaf forest of a proper admissible cut is \bas L(c)= \prod_{i=
  1}^n t_{l_i} \;. \eas
\item The root forest is \bas R(c)= \prod_{i= 1}^m t_{r_i}
  \;. \eas \end{itemize}
\end{dfn}

In the above example, for $c = \{\alpha, \beta\}$, the pruned forest
is \bas L(c) = \xy \POS(-4, -4) *+{\bullet} *\cir{} \POS(-8, -4) *+{B}
\POS(4, -4) *+{\bullet} \POS(4,-4) \ar@{-}_\gamma + (0, 8) \POS(7, -4)
*+{C} \POS(4, 4) *+{\bullet} *\cir{} \POS(7,4) *+{D} \endxy \eas and
the root forest is \bas R(c) = \xy \POS(0,4) *+{\bullet} *\cir{}
\POS(3,4) *+{A} \endxy \eas

In addition to proper admissible cuts, one considers two other
cuts. The empty cut is defined such that $L_{empty}(T) = 1$ and
$R_{empty}(T) = T$. The full cut is defined such that $L_{full}(T) =
T$ and $R_{full}(T) = 1$. The set of admissible cuts consists of
proper admissible cuts, the empty cut and the full cut.

\begin{dfn}
The coproduct on $\T$ is defined \ba \Delta (T) = \sum_{c \textrm{
    admis.}}R(c) \otimes L(c) \;.\label{treecoprod} \ea I denote the
contribution of the admissible cut $c$ to the coproduct as \bas
\Delta_c(T) = R(c) \otimes L(c)\;. \eas In this notation $\Delta (T) =
\sum_{c \textrm{ admis.}} \Delta_c(T)$.
\end{dfn}

Recall that in a coassociative bialgebra $\T$, for every $T \in \T$,
\ba (\Delta \otimes \id) \Delta (T) = (\id \otimes \Delta) \Delta
(T). \label{coassoc} \ea

\begin{lem}
The algebra $\T$ is a coassociative Hopf
algebra.\label{coassoclem}\end{lem}

\begin{proof}
Since $\T$ is connected and graded, if it is a bialgebra, it is a Hopf
algebra.

First I show that $\T$ is a bialgebra. The coproduct defined in
\eqref{treecoprod} is compatible with multiplication on $\T$: \bas
\Delta (TS) = \Delta(T) \Delta(S) \eas for $S, T \in \T$. Let $L_T$,
$L_S$ be the pruned forests of $T$ and $S$, and $R_T$ and $R_S$ the
root forests of $T$ and $S$. Then \bas \Delta (T) \Delta(S) = \sum_{d
  \textrm{ admis. of }T}\sum_{c \textrm{ admis. of }S}R_S(c)
R_T(d)\otimes L_S(c)L_T(d) \;.\eas Since the product of trees is the
disjoint union, an admissible cut of $TS$ is an element of the form
$d \cup c$, where $d$ is an admissible cut of $T$, and $c$ is an admissible
cut of $S$. Therefore, \bas \Delta (TS) = \sum_{d \cup c \textrm{
    admis. of }TS}R_S(c) R_T(d)\otimes L_S(c)L_T(d) = \Delta(T)
\Delta(S) \;.\eas

To see coassociativity, consider $c$, an admissible cut of $T$. Write
\bas \Delta_c(T) = R(c) \otimes L(c) \; .\eas Let $c_r$ be an
admissible cut of the forest $R(c)$. Then \ba (\Delta_{c_r} \otimes
\id)\Delta_c(T) = R_{c_r}(R_c(T)) \otimes L_{c_r}(R_c(T)) \otimes L_c(T)
\;.\label{coprodR}\ea Since the trees in the forest $R(c)$ are
subtrees of $T$, $c_r$ is also an admissible cut of $T$. The edges in
$c$ are an admissible cut of the forest formed by the product
$R_{c_r}(T) \cdot L_{c_r}(T)$. Write $c = c_1 \cup c_2$, with $c_1$ an
admissible cut of $R_{c_r}(T)$ and $c_2$ and admissible cut of
$L_{c_r}(T)$. Then $c' = c_r \cup c_1$ is an admissible cut of
$T$. The components of $R_{c'}(T) = R_{c_1}(R_{c_r}(T))$ are attached
to the source vertexes of the edges in $c'$ while the edges in
$L_{c_1}(R_{c_r}(T))$ are attached only to the terminal vertexes of
the edges in $c_1$. Furthermore, notice that by construction, \ba
R_{c'}(T) = R_{c_1}(R_{c_r}(T)) = R_{c_r}(R_c(T))
\;.  \label{coprod1}\ea Since $L_{c_1}(R_{c_r}(T))$ is part of the
forest $L_{c'}(T)$, \ba L_c(T) = L_{c_1}(R_{c_r}(T)) \cdot L_{c_2}(T)
= L_{c_2}(L_{c'}(T)) \;. \label{coprod2}\ea Since $c_2$ is an
admissible cut of $L_{c_r}(T)$, \ba L_{c_r}(R_c(T)) =
R_{c_2}(L_{c'}(T)) \;. \label{coprod3}\ea Combining equations
\eqref{coprodR} \eqref{coprod1} \eqref{coprod2} and \eqref{coprod3}
gives \bas R_{c'}(T) \otimes R_{c_2}(L_{c'}(T)) \otimes
L_{c_2}(L_{c'}(T)) =(\id \otimes \Delta_{c_2})\Delta_{c'}(T) \;.\eas
\end{proof}

\begin{rem}
Notice that if $T \in \T$ is a single rooted tree, the coproduct
defined above matches the coproduct and definition of admissible cut
in \cite{CK98}. For a single rooted tree, $R(c)$ is always a
tree. If $T$ is multi-rooted, $R(c)$ may be a forest.
\end{rem}

\begin{dfn}
Let $W(R)$ be the algebra of non-commutative words on $R$-deco polygons. 
\end{dfn}

The algebras $W(R)$ and $\bar T(V_\bullet(R))$ are isomorphic as
commutative Hopf algebras. There is a commutative product given by the
shuffle product and a coproduct given by deconcatenation. This is the
same as given in the bar construction in \eqref{barcoprod}. If $w =
w_1\ldots w_n \in W(R)$, with the $w_i$ non-trivial $R$-deco polygons,
\bas \Delta w_1 \ \ldots w_n = \sum_{i=0}^n (w_1 \otimes \ldots
\otimes w_i)\otimes (w_{i+1} \otimes \ldots \otimes w_n) .\eas There
is a natural identification \bas (W(R), \sha, \Delta) \simeq
(\bar{T}(V_\bullet(R)), \sha, \Delta) \simeq
(\bar{T}(\mathcal{P}^{(1)}_\bullet(R)), \sha, \Delta) \;. \eas



There is an algebra homomorphism from the algebra of trees, $\T$, to
the algebra of words, $W(R)$, which identifies the partial order
represented by $T$ with a sum of words in $W(R)$. I first need to define
linearizations of trees.

\begin{dfn}For $T \in \mathcal{T}^n(R)$, a partial order
preserving a linearization of $T$ is a word on $R$-deco polygons \bas
\lambda(T) = \lambda_1\otimes \lambda_2\otimes \ldots \otimes \lambda_n
\in W(R) \eas where each $\lambda_i$ is an $R$ deco polygon labeling a
vertex of $T$. If $\lambda_i \prec \lambda_j$ as vertexes in $T$, then
$i < j$. \end{dfn}

Let $\text{Lin}(T)$ be the set of partial order preserving
linearizations of trees. For any $\lambda \in \text{Lin}(T)$, the
polygon $\lambda_1$ is always the label of a root of $T$ and
$\lambda_n$ is always the label of a leaf of $T$. A forest in $\T$
also represents a partial order on its vertexes. The linearization of
trees extends naturally to forests.

In this paper, the partial order of $T$ is viewed as the sum of its
partial order preserving linearizations. I define a map from trees to
words by mapping each tree to the sum of its linearizations: \bas
\Lambda: \T &\rightarrow W(R) \\ T &\mapsto \sum_{\lambda \in
  \textrm{Lin}(T)}\lambda(T) \\ T' \cdot T &\mapsto \sum_{\lambda' \in
  \textrm{Lin} (T')}\lambda'(T') \sha \sum_{\lambda \in
  \textrm{Lin}(T)}\lambda(T) = \Lambda(T) \sha \Lambda(T') ,\eas where
$\sha$ is the shuffle product on $W(R)$.

\begin{eg}
Let $T$ be the tree, \bas T = 
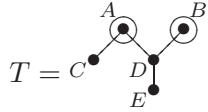
\xy \POS(2,6) *+{\bullet} *\cir{}
\POS(2,6) \ar@{-} +(-4,-4) \ar@{-} +(4,-4) \POS(0,8) *{^A} \POS(-2,2)
*+{\bullet} \POS(-4,0)*{^C} \POS(10,6) *+{\bullet} *\cir{} \POS(10,6)
\ar@{-} +(-4,-4) \POS(12,8) *{^B} \POS(6,2) *+{\bullet} \POS(4,0)*{^D}
\POS(6,2) \ar@{-} +(0,-4) \POS(6,-2) *+{\bullet} \POS(4,-4)*{^E}
\endxy \eas It has two root vertexes $A$ and $B$.  Then \bas \lambda(T) =
A\otimes B\otimes C\otimes D\otimes E \quad \rm{and} \quad \lambda'(T) =
B\otimes A\otimes D\otimes E\otimes C \eas are two partial order
preserving linearizations of $T$. The sum of all partial ordered
preserving linearizations is \bas \Lambda(T) = (A\sha B)\otimes
(C\sha(D\otimes E))+ A\otimes C\otimes B\otimes D\otimes E \;.\eas
\end{eg}

\begin{thm}
The map $\Lambda: \T \rightarrow W(R)$ is a bialgebra homomorphism.
\label{bialghomo}\end{thm}

\begin{proof}
The algebra homomorphism comes from the construction of the map
$\Lambda$. The coalgebra homomorphism is shown here.

For $T\in \mathcal{T}^n(R)$, the coproduct on $T$ is 
\bas \Delta(T) = \sum_{c \textrm{ admis.}} R(c) \otimes L(c) \eas and
the coproduct on the image, $\Lambda(T)$, is \bas \Delta
\Lambda(T)= \sum_{i=0}^{n}\sum_{\lambda\in \textrm{Lin}(T)} \left(\lambda_1\otimes
\ldots\otimes \lambda_i\right) \otimes\left(\lambda_{i+1}\otimes
\ldots \otimes \lambda_{n}\right) \;. \eas 

Any decomposition of a partial order preserving linearization $\lambda(T)$,
$[\lambda_1\otimes \ldots\otimes\lambda_i]$ and
$[\lambda_{i+1}\otimes\ldots \otimes \lambda_{n}]$, can be written as
a tensor product of partial order preserving linearization of forests of the form $\rho(R)$ and
$\eta(L)$ with $R$ and $L$ sub-forests of $T$ defined by the vertex
sets $\{\lambda_1 \ldots \lambda_i\}$ and $\{\lambda_{i+1} \ldots
\lambda_n\}$ respectively. The set of edges of $T$ that connect the
vertexes $\lambda_j$ to $\lambda_k$ for $j \leq i$ and $k > i$ define
an admissible cut of $T$.

For each admissible cut $c$, the trees in the forests $L(c)$ and
$R(c)$ are sub-trees of $T$. Let $\eta_c \in \text{Lin}(L(c))$ and
$\rho_c\in \text{Lin}(R(c))$ be partial order preserving
linearizations. Then \bas (\Lambda \otimes \Lambda) \circ \Delta(T) =
\sum_{c \textrm{ admis.}}  \mathop{\sum_{(\eta_c, \rho_c) \in}}_{
  \text{Lin}(R(c)) \times \text{Lin}(L(c))}\rho_c \otimes \eta_c \;,
\eas where the interior sum is taken over all partial order preserving
linearizations of $R(c)$ and $L(c)$. By definition of admissible cut,
each pair of partial order preserving linearizations $\rho_c \otimes
\eta_c$, corresponds to a decomposition of a partial order preserving
linearizations $\lambda$ of $T$, $[\lambda_1\otimes
  \ldots\otimes\lambda_i] \otimes[\lambda_{i+1}\otimes\ldots
  \otimes\lambda_{n}]$ where the vertexes of $R(c)$ precede the
vertexes of $L(c)$.
\end{proof}

\begin{rem}
Words in $W(R)$ represent a partial order on its letters. For $u$ and
$v$ letters in $W = \sum_i W_i \in W(R)$, $u\prec v$ if $u= x_{i_j}$
and $v = x_{i_k}$ in $W_i$ with $i_j \leq i_k$. Similarly, $v\prec u$
if $u= x_{i_j}$ and $v = x_{i_k}$ in $W_i$ with $i_j \geq i_k$. Under
this definition, the map $\Lambda$ is an order preserving Hopf algebra
homomorphism.
\end{rem}

To complete the analysis in this paper, I need to introduce a method
of inserting letters into words.

\begin{dfn} Define two insertion products on $W(R)$
\bas u\star_{\prec v} :  W(R)
& \rightarrow W(R) \\  x_1\otimes \ldots \otimes x_n &\rightarrow 
\begin{cases} \sum_{i < k}x_1\otimes \ldots u \otimes x_i\ldots  \otimes x_n & \textrm{if } v= x_k \\ 0 & v \not \in \{x_1,\ldots, x_k\} \end{cases} 
 \eas and \bas u\star_{\succ v} :  W(R)
& \rightarrow W(R) \\  x_1\otimes \ldots \otimes x_n &\rightarrow 
\begin{cases} \sum_{i >k}x_1\otimes \ldots x_i \otimes u\ldots  \otimes x_n & \textrm{if } v= x_k \\ 0 & v \not \in \{x_1,\ldots, x_k\}\end{cases}.
 \eas If $w= 1$ then \bas u\star_{\prec 1} 1 = u\star_{\succ 1} 1 = u\;.\eas 
 \label{grafting}\end{dfn}

To see this as a product, generalize the insertion of a letter to the intertwining of a word with another.

\begin{dfn}
Let $w$ and $w'$ be two words. Define a set of words $\{W_i | i \in
I\}$ such that $w \sha w' = \sum_{i\in I} W_i$, \bas \star_{u\prec v} :
W(R) \otimes W(R) & \rightarrow W(R) \\ (w,w') &\rightarrow
\begin{cases} \sum_{\begin{subarray}{c}i\in I \st \\ u \prec v\end{subarray}} W_i &
\textrm{if $u$ and $v$ letters of $w$ and $w'$ resp.} \\ 0 &
\textrm{if $u$ or $v$ not letters of $w$ and $w'$ resp.} \end{cases}
\eas Similarly, \bas \star_{u\succ v} : W(R) \otimes W(R) &
\rightarrow W(R) \\ (w,w') &\rightarrow
\begin{cases} \sum_{\begin{subarray}{c}i\in I \st \\ u \succ v\end{subarray}} W_i &
\textrm{if $u$ and $v$ letters of $w$ and $w'$ resp.} \\ 0 &
\textrm{if $u$ or $v$ not letters of $w$ and $w'$ resp.} \end{cases}
\eas
\end{dfn}

In the product $w \star_{u\prec v} w'$, the letter $u \prec v$, while
$v \succ u$ in $w \star_{u\succ v} w'$. In this notation, $u
\star_{\prec v}w := u \star_{u\prec v}w$ and $u \star_{\succ v}w := u
\star_{u\succ v}w$.

For further analysis, I extend this product for shuffles of words. For
$v$ a letter of the word $w$ and $v'$ a letter of the word $w'$,
define \bas (u \star_{\prec \{v,v'\}}w\sha w') := (u \star_{\prec v}w)
\star_{u\prec v'} w' \;, \eas and \bas (u \star_{\succ \{v,v'\}}w\sha
w') := (u \star_{\succ v}w) \star_{u\succ v'} w' \;. \eas

These operators can be lifted to grafting operators on trees. If $w =
\Lambda(T)$, and $w' = \Lambda(T')$, for $T$, and $T'$ in $\T$, \bas
(u \star_{\prec \{v,v'\}}w\sha w') = \Lambda(S) \eas where $S$ is the
multi-rooted tree formed by connecting the vertex labeled $v$ in $T$
and the vertex labeled $v'$ in $T'$ to a new root vertex labeled
$u$. The other insertion operator corresponds to connecting the two
marked vertexes to a new leaf, with the label $u$.

The coproduct on the images of these insertion operators behaves as
follows.

\begin{lem}
Write $w= x_1\otimes \ldots \otimes x_n$. The coproduct \bas \Delta (
u \star_{\prec x_k} w) =   \sum_{a=0}^{k-1} (x_1\otimes \ldots
  \otimes x_a) \otimes u \star_{\prec x_k} (x_{a+1}\otimes \ldots
  \otimes x_n) \\ + (x_1\otimes \ldots \otimes x_{k-1}) \sha u \otimes
  (x_k\otimes \ldots \otimes x_n) \\ + \sum_{a=k}^{n} u
  \star_{\prec x_k}(x_1\otimes \ldots \otimes x_a) \otimes
  (x_{a+1}\otimes \ldots \otimes x_n) \;.\eas Similarly, the coproduct
  The coproduct \bas \Delta ( u \star_{\succ x_k} w) =
     \sum_{a=0}^{k-1} (x_1\otimes \ldots \otimes x_a) \otimes u
    \star_{\succ x_k} (x_{a+1}\otimes \ldots \otimes x_n) \\ +
    \sum_{a=k}^{n} u \star_{\succ x_k}(x_1\otimes \ldots \otimes x_a)
    \otimes (x_{a+1}\otimes \ldots \otimes x_n)\\ + (x_1\otimes \ldots
    \otimes x_k) \otimes u \sha (x_{k+1}\otimes \ldots \otimes
    x_n)\;.\eas
 \label{graftingcoprod}
\end{lem}

\begin{proof}
The proof is straight forward from the definition of coproduct on all
words in the sum $u \sha w$ such that the letter $u$ appears to the left of
$x_k$, in the case of $\star_{\prec x_k}$, or to the right of
$x_k$, in the case of $\star_{\succ x_k}$
\end{proof}

\subsection{Hopf algebras of trees associated to dissected 
polygons\label{Hopftrees}}

To continue to generalizing the construction in \cite{GGL05}, I associate
a family of (multi-rooted) tree structures to each dissection of a
polygon. In subsection \ref{hopftodif} I define a family of bar complexes
on $\mcp$ and associate to each polygon a bar element in each bar
complex. 

The rest of this paper is concerned with subalgebras
$\Tsub{\phi}\subset \T$ generated by sets corresponding to polygons
and their dissections.

\begin{dfn}
A dual tree algebra $\Tsub{\phi}\subset \T $ is generated by a dual
tree generating set, which assigns to each polygon dissection pair
$(P,d)$ an element of $\T$, $T_{\phi,d}(P)$, \bas \phi =\{T_{\phi,d}(P) |P\;
R-\textrm{deco polygon} \quad ; \quad d \in D(P) \}\;.\eas
\end{dfn}

In this paper, the generators, $T_{\phi,d}(P) \in \T$ are trees with
an overall sign.

\begin{dfn}
I write the overall sign associated to a tree
$\sgn(T_{\phi,d}(P))$. For $d \in D(P)$, sometimes I write this
$\sgn_\phi(P,d)$. When the polygon is clear, I write $\sgn_\phi(d)$.
\end{dfn}

I am particularly interested in the cases when $\Tsub{\phi}$ has a Hopf
algebraic structure.

\begin{dfn}
The dual tree algebra $\Tsub{\phi}$ is a dissection compatible Hopf
algebra if
\begin{enumerate}
\item The dual tree algebra $\Tsub{\phi}$, generated by the set $\phi$, is a
  sub-Hopf algebra of $\T$. 
\item The edges of each tree $T_{\phi,d}(P)$ correspond to
  non-trivial arrows in $d$.
\item Let $d'$ be a sub-dissection of $d$, $d' \subset d \in D(P)$,
  corresponding to the subtree $T$ of $T_{\phi,d}(P)$. There exists an
  $R$-deco polygon $Q$ such that $T$ and $T_{\phi,d'}(Q)$, with $d'
  \in D(Q)$, agree up to a sign. In fact \bas \sgn(T_{\phi,d'}(Q)) T =
  T_{\phi,d'}(Q) \;.\eas \label{compat1}
\item Consider $d' \subset d \in D(P)$ as above. The generator
  $T_{\phi, d\setminus d'}(P)$ is formed by replacing the subtree $T$
  in $T_{\phi,d}(P)$ with a single vertex labeled $Q$.\label{compat2}
\item For each subdissection $d' \subset d$, with corresponding
  generator $T_{\phi,d'}(Q)$, \bas \sgn(T_{\phi,d}(P)) =
  \sgn(T_{\phi,d'}(Q)) \sgn (T_{\phi,d\setminus d'}(P)) \;.\eas \label{compat3}
\end{enumerate}
\label{discomp}\end{dfn}

The third condition of definition \ref{discomp} ensures that a
dissection compatible Hopf algebra is coassociative. Notice that there
is no requirement that the labels of the vertexes of $T_{\phi,d}(P)$
correspond to the subpolygons associated to the dissection. However
conditions \ref{compat1} and \ref{compat2} impose strong conditions on
the vertex labels of the generators of dissection compatible Hopf
algebras.  In most of the examples I consider in this paper,
the vertexes of the generators are labeled by the the subpolygons
associated to the relevant dissection. Below I give some examples of
some generators of dissection compatible Hopf algebras.
\begin{eg}
Consider the following pair of polygon and dissection: \bas P =
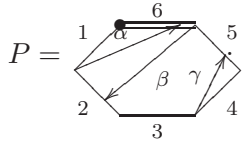
\xy
\POS(0,4) *+{\bullet}
\POS(10,4) \ar@{=} +(-10,0)_6
\ar@{-} +(6,-6)^5
\POS(16,-2)\ar@{-} +(-6,-6)^4
\POS(0,-8)\ar@{-} +(10,0)_3
\ar@{-} +(-6,6)^2
\POS(0,4)\ar@{-} +(-6,-6)_1
\POS(-6,-2) \ar^{\alpha}@{->} +(14,6)
\POS(10,4) \ar^{\beta}@{->} +(-12,-10)
\POS(10,-8) \ar^{\gamma}@{->} +(4,8)
\hskip 40pt.
\endxy \eas Four possible elements of dual tree generating sets are
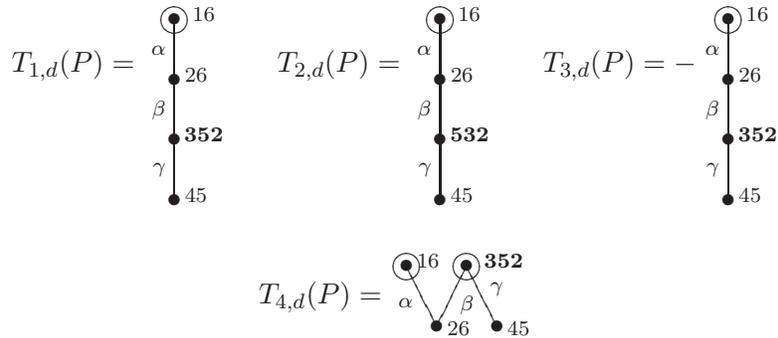
\bas T_{1,d}(P) =
 \xy
\POS(2,6) *+{\bullet} *\cir{}
\POS(2,6) \ar_{\alpha}@{-} +(0,-8)
\POS(6,6) *{^{16}}
\POS(2,-2) *+{\bullet}
\POS(2,-2) \ar_\beta@{-} +(0,-8)
\POS(5,-2) *{^{26}}
\POS(2,-10)  *+{\bullet}
\POS(2,-10) \ar_\gamma@{-} +(0,-8)
\POS(6,-10) *{^{\bf{352}}}
\POS(2,-18)  *+{\bullet}
\POS(5,-18) *{^{45}}
\endxy \hskip 20pt
T_{2,d}(P) =
 \xy
\POS(2,6) *+{\bullet} *\cir{}
\POS(2,6) \ar_{\alpha}@{-} +(0,-8)
\POS(6,6) *{^{16}}
\POS(2,-2) *+{\bullet}
\POS(2,-2) \ar_\beta@{-} +(0,-8)
\POS(5,-2) *{^{26}}
\POS(2,-10)  *+{\bullet}
\POS(2,-10) \ar_\gamma@{-} +(0,-8)
\POS(6,-10) *{^{\bf{532}}}
\POS(2,-18)  *+{\bullet}
\POS(5,-18) *{^{45}}
\endxy \hskip 20pt
T_{3, d}(P) =-
 \xy
\POS(2,6) *+{\bullet} *\cir{}
\POS(2,6) \ar_{\alpha}@{-} +(0,-8)
\POS(6,6) *{^{16}}
\POS(2,-2) *+{\bullet}
\POS(2,-2) \ar_\beta@{-} +(0,-8)
\POS(5,-2) *{^{26}}
\POS(2,-10)  *+{\bullet}
\POS(2,-10) \ar_\gamma@{-} +(0,-8)
\POS(6,-10) *{^{\bf{352}}}
\POS(2,-18)  *+{\bullet}
\POS(5,-18) *{^{45}}
\endxy \eas 
\bas T_{4,d}(P)=
\xy
\POS(0,4) *+{\bullet} *\cir{}
\POS(0,4) \ar_\alpha@{-} +(4,-8)
\POS(3,4) *{^{16}}
\POS(4,-4) *+{\bullet}
\POS(4,-4) \ar_\beta@{-} +(4,8)
\POS(7,-5) *{^{26}}
\POS(8,4) *+{\bullet} *\cir{}
\POS(8,4) \ar^\gamma@{-} +(4,-8)
\POS(13,4) *{^{\bf{352}}}
\POS(12,-4) *+{\bullet}
\POS(15,-5) *{^{45}}
\endxy \eas

For a fixed $P$ and $d \in D(P)$, the generators $T_{1,d}(P)$ and
$T_{3,d}(P)$ differ only by an overall sign. The only multi-rooted
tree in this example is $T_{4,d}(P)$. The vertexes in
$T_{2,d}(P)$ do not correspond to the set of subpolygons associated to
the dissections $d$ of $P$. Instead of a vertex labeled with the
polygon $352$, there is a polygon labeled $532$.
\label{orderings}\end{eg}

Next I give examples of four dual tree generating sets $\phi_i$, with
$i \in \{1 \ldots 4 \}$ such that $\Tsub{\phi_i}$ is a dissection
compatible Hopf algebra. The generators $T_{i,d}$ in example
\ref{orderings} correspond to elements in $\phi_i$. Before defining
these Hopf algebras and the construction of the corresponding
generators, I establish some notation.

\begin{dfn} 
Let $\tau$ be a map that reverses the orientation of a
polygon. \end{dfn} 

Specifically, for $P= r_1\ldots r_{n+1}$, with $r_i
\in R$, $\tau \in D_{2(n+1)}$ and $\tau(P) = r_n \ldots r_1
r_{n+1}$. For example,

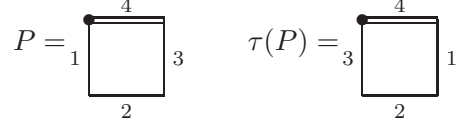
\xy
\hskip 80pt
$P=$
\hskip 5pt
\POS(0,4) *+{\bullet}
\POS(10,4) \ar@{=} +(-10,0)_4
\ar@{-} +(0,-10)^3
\POS(10,-6) \ar@{-} +(-10,0)^2
\POS(0,4) \ar@{-} +(0,-10)_1
\hskip 60pt
$\tau(P)= $
\hskip 5pt
\POS(0,4) *+{\bullet}
\POS(10,4) \ar@{=} +(-10,0)_4
\ar@{-} +(0,-10)^1
\POS(10,-6) \ar@{-} +(-10,0)^2
\POS(0,4) \ar@{-} +(0,-10)_3
\hskip 40pt
\endxy

\begin{dfn}Define $\chi(\alpha)$ to be the weight of the cut off
polygon of the arrow $\alpha$. That is, $P_\alpha^\rst \in
\mathcal{P}_{\chi(\alpha)}^{(1)}(R)$.\end{dfn}

\begin{eg}
The following are four dissection compatible Hopf algebras.

\begin{description}
\item[$\Tsub{\phi_1}$] This is a single rooted dual tree algebra
  generated by the set $\phi_1$. The root vertex of $T_{\phi_1, d}(P)
  \in \phi_1$ is labeled by the subpolygon that contains the original
  root side and first vertex of $P$. The edges, corresponding to
  arrows in $d$, are oriented to flow away from the root vertex. Since
  this is a single rooted tree, I need only consider the final
  vertexes of any edge (corresponding to the dissecting arrow
  $\alpha$.) The initial vertex is either the final vertex of a
  different arrow, or the unique root, whose label has been
  defined. Consider the edge corresponding to the dissecting arrow
  $\alpha$. The final vertex is labeled by the subpolygon
  corresponding to the region further away from the first vertex/root
  side of $P$. In example \ref{orderings}, $T_{1,d}(P) \in \phi_1$.
\item[$\Tsub{\phi_2}$] This is a single rooted dual tree algebra
  generated by the set $\phi_2$. The single root vertex of
  $T_{\phi_2,d}(P) \in \phi_2$ is labeled by the subpolygon that
  contains the original root side and first vertex of $P$. The edges,
  corresponding to arrows in $d$, are oriented to flow away from the
  root vertex.  If $\alpha$ is a forwards arrow, then the final vertex
  is labeled by the subpolygon corresponding to the region further
  from the root. If $\alpha$ is a backwards arrow, then the final
  vertex is labeled by the same polygon with reversed orientation. The
  generator has an overall sign \bas \sgn(T_{\phi_2,d}(P)) =
  (-1)^{\sum_{\alpha \in d \textrm{ backwards}}\chi(\alpha)}\; .\eas
  In example \ref{orderings}, $T_{2,d}(P) \in \phi_2$.
\item[$\Tsub{\phi_3}$] This is a single rooted dual trees algebra
  generated by the set $\phi_3$. The single root vertex of
  $T_{\phi_3,d}(P) \in \phi_3$ is labeled by the subpolygon that
  contains the original root side and first vertex of $P$. The edges,
  corresponding to arrows in $d$, are oriented to flow away from the
  root vertex. The final vertex is labeled by the subpolygon
  corresponding to the region further away from the first vertex/root
  side of $P$. The generator has an overall sign \bas
  \sgn(T_{\phi_3,d}(P)) = (-1)^{\# \textrm{ backwards arrows in } d}\;
  .\eas In example \ref{orderings}, $T_{3,d}(P)\in \phi_3$.
\item[$\Tsub{\phi_4}$]This is a multi-rooted dual tree algebra generated by
  the set $\phi_4$, with generators $T_{\phi_4,d}(P) \in \phi_4$. The
  edges, corresponding to arrows in $d$, flow from the region to the
  left of the arrow to the region to the right. The initial vertex of
  an edge is labeled with the subpolygon associated to the region to
  the left, and the final vertex by the subpolygon associated to the
  right. For the polygon pair in example \ref{orderings},
  $T_{4,d}(P)\in \phi_4$.
\end{description}
\label{keyexamples} \end{eg}

In the dissection compatible sub Hopf algebras in example
\ref{keyexamples}, write the generating set of $\Tsub{\phi_i}$ \bas \phi_i
= \{T_{\phi_i,d}(P) |d \in D(P) \quad ; \quad P\; R-\textrm{deco polygon}\}
\eas for $i \in \{1, 2, 3, 4,\}$. Notice that the different generating
sets $\phi_1$ and $\phi_3$ generate isomorphic Hopf algebras $\Tsub{\phi_1}
\simeq \Tsub{\phi_3}$ under the relation \bas T_{\phi_1,d}(P) =
\sgn_{\phi_3}(d) (T_{\phi_3,d}(P)) \;.\eas Each $\Tsub{\phi_i}$ satisfies
conditions 2-5 of definition \ref{discomp}. To see that these are
dissection compatible Hopf algebras, it remains to check that they are
sub-Hopf algebras. The Hopf algebra $\Tsub{\phi_2}$ is exactly the Hopf
algebra defined in \cite{GGL05}, section 6. It remains to check that
$\Tsub{\phi_1}$, $\Tsub{\phi_3}$ and $\Tsub{\phi_4}$ are Hopf algebras.

\begin{lem}
The algebras $\Tsub{\phi_1}$, $\Tsub{\phi_3}$ and $\Tsub{\phi_4}$ are Hopf algebras.
\end{lem}

\begin{proof}
Since all three are graded subalgebras of $\T$, it is sufficient to
show that these are sub bialgebras. The product structure and
coproduct structure on each are inherited from $\T$. It remains to
check that \bas \Delta: \Tsub{\phi_i} \rightarrow \Tsub{\phi_i} \otimes \Tsub{\phi_i}
\eas for $i \in \{1, 3, 4\}$. For this, it is sufficient to work only
with the generators.

Let $\{P_1, \ldots P_{|d|+1}\}$ be the vertexes labeling
$T_{\phi_4,d}(P)$. Since the edges of $T_{\phi_4, d}(P)$ correspond to arrows to
the dissection $d$, an admissible cut, $c$, of $T_{\phi_4, d}(P)$ can be
thought of as a subdissection $c\subset d$. Let $\{Q_1, \ldots Q_n\}$
be the polygons associated to the dissection $c \in D(P)$, with $R(c)
= \prod_{i = 1}^jT_i$ and $L(c) = \prod_{i=j+1}^n T_i$. Write $d =
(\cup_{k=1}^nd_i)\cup c$ with the subdissection $d_i$ corresponding to
the edges in $T_i$. It remains to check that \bas  T_i = \sgn_{\phi_4}(d_i)T_{\phi_4,
  d_i}(Q_i) \;. \eas 

For each $\beta \in d\setminus c$, let $P_k$ and $P_j$ be the regions
of $P$ to the left and right of $\beta$ respectively. By definition,
$\beta \in d_i$ for some $i$. Then $P_k$ and $P_j$ are also the
sub-regions of $Q_i$ to the left and right of $\beta \in d_i$.  Thus
$T_{d_i} = T_{4,d_i}(Q_i)$.

The argument is similar for $\Tsub{\phi_1}$ and $\Tsub{\phi_3}$. Since
they are isomorphic, it is sufficient to work only with $\phi_1$. Let
$\{P_1, \ldots, P_{|d|+1}\}$ be the vertex labels of $T_{\phi_1,d}(P)
\in \phi_1$. Consider the admissible cut $c$ with $L(c) =
\prod_{i=1}^{n-1}T_i$ and $R(c) = T_{n}$. Then $\{Q_1, \ldots ,
Q_{n}\} $ are the polygons associated to $c$.  The dissection $d$
can be written $d = (\cup_{i=1}^{n}d_i) \cup c$ with $d_i$
corresponding to the edges of $T_i$. For each $\beta \in d \setminus
c$, let $P_i$ be the region of $P$ on the root side of $\beta$ and
$P_j$ on the cutoff side. Let $\beta \in d_k$. Then $P_i$ and $P_j$
correspond to the regions corresponding to the root and the cutoff
sides of $\beta \in d_k \in D(Q_k)$. Thus $T_i =
\sgn_{\phi_1}(d_i)T_{\phi_1,d_i} (Q_i)$.
\end{proof}

\subsection{From generating sets to differentials\label{hopftodif}}

A dual tree generating set that defines a dissection compatible Hopf
algebra also defines a degree one differential on $\mcp$,  $\D:
\mathcal{P}_\bullet^{(i)}(R) \rightarrow
\mathcal{P}_\bullet^{(i+1)}(R)$ satisfying $\D \circ \D = 0$ and
the Leibniz rule \bas \D_\phi (a \wedge b) = (\D_\phi a) \wedge b +
(-1)^i a \wedge \D_\phi(b)\;, \eas where $a \in
\mathcal{P}^{(i)}_\bullet(R)$.

Let $\phi$ be a dual tree generating set. Consider the subset of
dissections $d \in D(P)$ such that $|d| =1$. Write the corresponding
elements of $\phi$ \bas T_{\phi, d} (P) = \sgn_\phi(d) \xy \POS(2,4)
*+{\bullet} *\cir{} \POS(2,4) \ar@{-} +(0,-8) \POS(6,4) *{^{P_d^1}}
\POS(2,-4) *+{\bullet} \POS(5,-4) *{^{P_d^2}} \endxy \;.\eas This
structure defines a an operator on $\mcp$.

\begin{dfn}
Define \bas \D_{\phi}(P) = \sum_{d \in D(P) ; |d|=1 } \sgn_{\phi}(d)
P_d^1 \wedge P_d^2 \; .\eas
\end{dfn}

\begin{thm}
If $\phi$ generates a dissection compatible Hopf
algebra, $\Tsub{\phi}$, then $\D_{\phi}$ is a degree one differential
operator on $\mcp$.\label{difdef}\end{thm} 

\begin{proof}

By construction, \bas \D_\phi \circ \D_\phi(P) =
\sum_{\begin{subarray}{c}\alpha, \;
    \beta\\ \textrm{dis. arrow}\end{subarray}} (P_\beta^1)_\alpha^1
\wedge (P_\beta^1)_\alpha^2 \wedge P_\beta^2 - P_\beta^1 \wedge
(P_\beta^2)_\alpha^1 \wedge (P_\beta^2)_\alpha^2 \\ +
(P_\alpha^1)_\beta^1 \wedge (P_\alpha^1)_\beta^2 \wedge P_\alpha^2 -
P_\alpha^1 \wedge (P_\alpha^2)_\beta^1 \wedge (P_\alpha^2)_\beta^2
\;. \eas Some of these terms are $0$. For instance, if $\beta \not \in
D(P_\alpha^i)$ then $D(P_\alpha^i)_\beta^i = 0$.

This can be calculated by considering the sum \ba (\Delta_\alpha
\otimes \id)\Delta_\beta - (\id \otimes \Delta_\alpha)\Delta_\beta +
(\Delta_\beta \otimes \id)\Delta_\alpha - (\id \otimes
\Delta_\beta)\Delta_\alpha \label{coprodform}\ea on the level of Hopf
algebras, and passing from the tensor product to the wedge product.

Ignoring the sign and the vertex labels, the generator
$T_{\phi,\{\alpha, \beta\}}(P)$ is one of three possible trees \bas
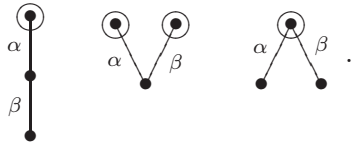
\xy \POS(2,6) *+{\bullet} *\cir{} \POS(2,6) \ar_{\alpha}@{-} +(0,-8)
\POS(2,-2) *+{\bullet} \POS(2,-2) \ar_\beta@{-} +(0,-8) \POS(2,-10)
*+{\bullet} \endxy \hspace{20pt} \xy \POS(0,4) *+{\bullet} *\cir{}
\POS(0,4) \ar_\alpha@{-} +(4,-8) \POS(4,-4) *+{\bullet} \POS(4,-4)
\ar_\beta@{-} +(4,8) \POS(8,4) *+{\bullet} *\cir{}
\endxy \hspace{20pt} \xy \POS(0,4) *+{\bullet} *\cir{} \POS(0,4)
\ar_\alpha@{-} +(-4,-8) \POS(-4,-4) *+{\bullet} \POS(0,4)
\ar^\beta@{-} +(4,-8) \POS(4,-4) *+{\bullet} \endxy \;.\eas

In the case of the linear tree, the first two terms of
\eqref{coprodform} are equal, and therefore cancel. The other two are
$0$. For the non-linear single rooted tree, the sum of the first and
third terms in \eqref{coprodform} gives $(P_\beta^1)_\alpha^1 \otimes P_\beta^2
\sha P_\alpha^2$. The other two are $0$. For the non-linear
multi-rooted tree, the sum of the second and fourth terms in
\eqref{coprodform} gives $P_\beta^1 \sha P_\alpha^2 \otimes
(P_\alpha^2)_\beta^2$. The other two are $0$. In both cases, the
shuffle product goes to zero as one passes to the wedge product.

Thus $\D_\phi \circ \D_\phi = 0$ as desired. 

 \end{proof}

Since the dual tree generating sets $\phi_i$, $i \in \{ 1, 2, 3, 4\}$
in example \ref{keyexamples} generate dissection compatible Hopf
algebras $\Tsub{\phi_i}$, they define degree 1 differential operators
$\D_{i}$, respectively on $\mcp$.

\begin{eg}
The differentials defined by the sets $\phi_1$, $\phi_2$, $\phi_3$,
and $\phi_4$ are
\begin{enumerate}
\item $\D_{1}(P) = \sum_{d\in D(P), |d|=1} P_d^\rt \wedge P_d^\rst$
\item $\D_{2}(P) = \sum_{d \textrm{ forwards arrow}}  P_d^\rt \wedge P_d^\rst + \sum_{d \textrm{ backwards arrow}}  (-1)^{\chi(d)}P_d^\rt \wedge \tau( P_d^\rst)$
\item $\D_{3}(P) = \sum_{d \textrm{ forwards arrow}}  P_d^\rt \wedge P_d^\rst - \sum_{d \textrm{ backwards arrow}}  P_d^\rt \wedge  P_d^\rst$ 
\item $\D_{4}(P) = \sum_{d\in D(P), |d|=1} P_d^l \wedge P_d^r$
\end{enumerate}
\end{eg}

These differentials defined by dual tree generating sets are not all
distinct. Specifically, \bas \sgn_{\phi_3}(\alpha)
= \begin{cases}\sgn_{\phi_4}(\alpha) & \textrm{if $\alpha$
    forwards}\\- \sgn_{\phi_4}(\alpha) & \textrm{if $\alpha$
    backwards.}  \end{cases} \eas As a result, $\D_{3} = \D_{4}$.

\begin{dfn} 
The difference set between two dual tree generating sets $\phi$ and
$\psi$ is \bas \mathcal{S} = \{ \alpha \in D(P) | |\alpha| = 1 ;\; P \; \;
R-\textrm{deco}; \; T_{\phi, \alpha}(P) = \sgn_{\phi}(\alpha)\xy
\POS(2,4) *+{\bullet} *\cir{} \POS(2,4) \ar@{-} +(0,-8) \POS(6,4)
*{^{P_\alpha^1}} \POS(2,-4) *+{\bullet} \POS(5,-4) *{^{P_\alpha^2}}
\endxy \\ \textrm{ and } T_{\psi,\alpha}(P) = -\sgn_{\phi}(\alpha)\xy
\POS(2,4) *+{\bullet} *\cir{} \POS(2,4) \ar@{-} +(0,-8) \POS(6,4)
*{^{P_\alpha^2}} \POS(2,-4) *+{\bullet} \POS(5,-4) *{^{P_\alpha^1}}
\endxy\}\; \;.\eas
\end{dfn}

This condition on single dissections can be generalized to general
trees.

\begin{dfn}
Let $\Tsub{\phi}$ be a dissection compatible Hopf algebra. Let
$\mathcal{S}$ be the difference set between $\phi$ and another dual
tree generating set $\psi$. The dual tree algebra $\Tsub{\psi}$ is
almost $\phi$ compatible if, for any dissection $d$ of any
$R$-deco polygons $P$, $\sgn_\psi(P,d) = (-1)^{|d \cap
  \mathcal{S}|}\sgn_\phi(P,d)$, and the tree underlying the generator
$T_{\psi, d}(P)$ is formed by reversing the orientation of the edges
of $T_{\phi, d}(P)$ in $d \cap \mathcal{S}$.
\end{dfn}

The dissection compatible Hopf algebra, $\Tsub{\phi}$ is trivially an
almost $\phi$ compatible algebra.

\begin{cor} 
Let $\Tsub{\phi}$ be a dissection compatible Hopf algebra, and
$\Tsub{\phi'}$ an almost $\phi$ compatible algebra. Then
$\D_{\phi} = \D_{\phi'}$. \label{samedifdef}\end{cor}
\begin{proof}
Let $\mathcal{S}$ be the difference set between $\phi$ and $\phi'$.
By definition, \bas \D_\phi(P) = \sum_{d \in D(P) ; |d|=1}
\sgn_{\phi}(d) P_d^1 \wedge P_d^2 \\ = \sum_{d\not \in D(P)\cap
  \mathcal{S}}\sgn_{\phi}(d) P_d^1\wedge P_d^2 - \sum_{d \in D(P)\cap
  \mathcal{S}} \sgn_{\phi}(d) P_d^2 \wedge P_d^1 \\ = \D_{\phi'}(P)
\;.\eas Since $\Tsub{\phi'}$ is almost $\phi$ compatible,
$\partial_{\phi'} \circ\partial_{\phi'} = 0$.
\end{proof}

Almost compatible algebras are particularly important for
the calculations in section \ref{permutesection}. I give an example of
a such below.

\begin{dfn}
Let $re(P)$ be the set of non-trivial arrows ending on the root side of an
$R$-deco polygon $P$ (the \emph{r}oot \emph{e}nding arrows). To fix
notation, for $P$ an $n$-gon, write $re(P) = \{{_2}\alpha, \ldots
{_{n-1}}\alpha\}$, where ${_i}\alpha$ starts at the $i^{th}$ vertex.
 \end{dfn}

\begin{eg} 
I define a dual tree algebra, $\Tsub{\phi_{re}}$ generated by
the set of single rooted trees $\phi_{re}$. For any dissection $d \in
D(P)$, the root vertex of $T_{\phi_{re},d}(P) \in \phi_{re}$ is
labeled by the subpolygon that contains the original root side and
\emph{last} vertex of $P$. The edges of the generator are oriented to
flow away from the root vertex. Consider the edge corresponding to the
dissecting arrow $\alpha$. The initial vertex of $\alpha$ is labeled
by the subpolygon corresponding to the region closer to the last
vertex/root side of $P$. If $\alpha$ is a forwards arrow, then the
final vertex is labeled by the subpolygon corresponding to the region
further away. If $\alpha$ is a backwards arrow, then the final vertex
is labeled by the same polygon with reversed orientation. The
generator has an overall sign \bas \sgn(T_{\phi_{re},d}(P)) = (-1)^{| d\cap
  re(P)|}(-1)^{\sum_{\alpha \in d \textrm{ backwards}}\chi(\alpha)} =
(-1)^{| d\cap re(P)|} \sgn_{\phi_2}(d) \;.\eas \label{re}\end{eg}

The algebra $\Tsub{\phi_{re}}$ is almost $\phi_2$ compatible. The
difference set between $\phi_{re}$ and $\phi_2$ is \bas \mathcal{S} =
\bigcup_{P \; R-\textrm{deco}} re(P) \;.\eas If $\alpha$ is a
backwards arrow, then $\alpha \not \in re(P)$ by construction, and \bas
T_{\phi_{re}, \alpha}(P) = (-1)^{\chi(\alpha)}\xy \POS(2,4) *+{\bullet}
*\cir{} \POS(2,4) \ar@{-} +(0,-8) \POS(6,4) *{^{P_\alpha^\rt}}
\POS(2,-4) *+{\bullet} \POS(7,-4) *{^{\tau (P_\alpha^\rst)}}\endxy
\quad ; \quad T_{\phi_2, \alpha}(P) = (-1)^{\chi(\alpha)}\xy \POS(2,4)
*+{\bullet} *\cir{} \POS(2,4) \ar@{-} +(0,-8) \POS(6,4)
*{^{P_\alpha^\rt}} \POS(2,-4) *+{\bullet} \POS(7,-4)
*{^{\tau(P_\alpha^\rst)}} \endxy \;.\eas Similarly, if $\alpha$ is a
forwards arrow $\alpha \not \in re(P)$, \bas T_{\phi_{re}, \alpha}(P) = \xy
\POS(2,4) *+{\bullet} *\cir{} \POS(2,4) \ar@{-} +(0,-8) \POS(6,4)
*{^{P_\alpha^\rt}} \POS(2,-4) *+{\bullet} \POS(5,-4)
*{^{P_\alpha^\rst}}\endxy \quad ; \quad T_{\phi_2, \alpha}(P) = \xy
\POS(2,4) *+{\bullet} *\cir{} \POS(2,4) \ar@{-} +(0,-8) \POS(6,4)
*{^{P_\alpha^\rt}} \POS(2,-4) *+{\bullet} \POS(5,-4)
*{^{P_\alpha^\rst}} \endxy \;.\eas On the other hand, if $\alpha \in
re(P)$, \bas T_{\phi_{re}, \alpha}(P) = - \xy \POS(2,4) *+{\bullet} *\cir{}
\POS(2,4) \ar@{-} +(0,-8) \POS(6,4) *{^{P_\alpha^\rst}} \POS(2,-4)
*+{\bullet} \POS(5,-4) *{^{P_\alpha^\rt}}\endxy \quad ; \quad T_{\phi_2,
  \alpha}(P) = \xy \POS(2,4) *+{\bullet} *\cir{} \POS(2,4) \ar@{-}
+(0,-8) \POS(6,4) *{^{P_\alpha^\rt}} \POS(2,-4) *+{\bullet} \POS(5,-4)
*{^{P_\alpha^\rst}} \endxy \;.\eas To see that $\Tsub{\phi_{re}}$ is not a
dissection compatible Hopf algebra of $\T$, consider the
polynomial dissection pair \bas P = 
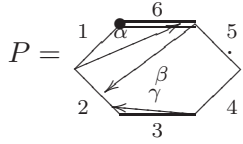
\xy \POS(0,4) *+{\bullet}
\POS(10,4) \ar@{=} +(-10,0)_6 \ar@{-} +(6,-6)^5 \POS(16,-2)\ar@{-}
+(-6,-6)^4 \POS(0,-8)\ar@{-} +(10,0)_3 \ar@{-} +(-6,6)^2
\POS(0,4)\ar@{-} +(-6,-6)_1 \POS(-6,-2) \ar^{\alpha}@{->} +(14,6)
\POS(10,4) \ar^{\beta}@{->} +(-12,-9) \POS(10,-8) \ar_{\gamma}@{->}
+(-11,1)
\hskip 40pt.
\endxy \eas In $\Tsub{\phi_{re}}$, the generator  \bas
 T_{\phi_{re},d} (P)=
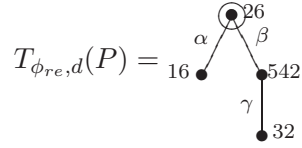
\xy
\POS(-2,-2) *+{\bullet}
\POS(-2,-2) \ar^\alpha@{-} +(4,8)
\POS(-5,-2) *{^{16}}
\POS(2,6) *+{\bullet} *\cir{}
\POS(2,6) *+{\bullet}
\POS(2,6) \ar^\beta@{-} +(4,-8)
\POS(5,6) *{^{26}}
\POS(6,-2)  *+{\bullet}
\POS(6,-2) \ar_\gamma@{-} +(0,-8)
\POS(9,-2) *{^{542}}
\POS(6,-10)  *+{\bullet}
\POS(9,-10) *{^{32}}
\endxy \eas

Consider the admissible cut $c =\{\beta\}$. The corresponding summand
of the coproduct is \bas \Delta_c (T_{\phi_{re},d} (P)) = \xy \POS(2,4)
*+{\bullet} *\cir{} \POS(2,4) \ar@{-} +(0,-8) \POS(6,4) *{^{26}}
\POS(2,-4) *+{\bullet} \POS(5,-4) *{^{16}} \endxy \otimes \xy
\POS(2,4) *+{\bullet} *\cir{} \POS(2,4) \ar@{-} +(0,-8) \POS(6,4)
*{^{542}} \POS(2,-4) *+{\bullet} \POS(5,-4) *{^{32}} \endxy \;.\eas
While $\xy \POS(2,4) *+{\bullet} *\cir{} \POS(2,4) \ar@{-} +(0,-8)
\POS(6,4) *{^{26}} \POS(2,-4) *+{\bullet} \POS(5,-4) *{^{16}} \endxy =
T_{\phi_{re},\alpha}(126)$, $\xy \POS(2,4) *+{\bullet} *\cir{} \POS(2,4)
\ar@{-} +(0,-8) \POS(6,4) *{^{542}} \POS(2,-4) *+{\bullet} \POS(5,-4)
*{^{32}} \endxy = T_{\phi_2,\gamma}(5432)$. However,
the generator $T_{\phi_2, \gamma}(5432) \not \in
\Tsub{\phi_{re}}$.

\begin{lem}
The algebra $\Tsub{\phi_{re}}$ is not a dissection compatible Hopf algebra.
\end{lem}
\begin{proof}
In fact, I show that $\Tsub{\phi_{re}}$ is not a Hopf algebra at all.
Consider the generator $T_{\phi_{re},d}(P) \in \Tsub{\phi_{re}}$, and an admissible cut
consisting of a single arrow, $c = \alpha \not \in re(P)$. The
dissection $d \in D(P)$ can be decomposed into the sets \bas d = d^\rt
\cup c \cup d^\rst \eas where $d^\rt \in D(P^\rt_\alpha)$ and $d^\rst
\in D(P^\rst_\alpha)$. Since $c=\alpha \not \in re(P)$, the root side
of $P^\rst$ does not correspond to the root side of $P$. Furthermore,
$d^\rst \cap re(P) = \emptyset$. Thus the above discussion shows that
the corresponding summand in the coproduct is \bas \Delta_c(T_d(P)) =
T_{\phi_{re}, d^\rt}(P_\alpha^\rt) \otimes T_{\phi_2, d^\rst}(P_\alpha^\rst) \;. \eas If
$d^\rst \cap re(P_\alpha^\rst) \neq \emptyset$, then $T_{\phi_2,
  d^\rst}(P_\alpha^\rst) \not \in \Tsub{\phi_{re}}$, and $\Delta(T_d(P)) \not \in
\Tsub{\phi_{re}} \otimes \Tsub{\phi_{re}}$.
\end{proof}

A class of differentials on $\mcp$, \bas \{\D_\phi | \Tsub{\phi}
\textrm{ dissection compatible Hopf algebra}\}\; ,\eas defines a class
of bar complexes $\{B_{\D_{\phi}}(\mcp)\}$. Continuing to generalize
the construction of \cite{GGL05}, I associate to each $R$-deco
polygon an element of each $B_{\D_{\phi}}(\mcp)$.

\begin{dfn}
If $\Tsub{\phi}$ defines a differential $\D_\phi$, define \bas
\Lambda_{\phi} \quad : V_\bullet(R) & \rightarrow B_{\D_{\phi}}(\mcp)\\
 P & \rightarrow \sum_{d \in D(P)}
\Lambda (T_{\phi,d}(P)) \;.\eas
\end{dfn}

There is a natural way of identifying a subalgebra of $\bar{T}(V(R))$
that is generated by $R$-deco polygons

\begin{dfn}
A polygon algebra defined by the dual tree generating set $\phi$ is
\bas B_\phi = \Q[\{ \Lambda_\phi(P) | P \; \; R-\textrm{deco}\}] \;.\eas
Its generators are in one to one correspondence with $R$ -deco
polygons.
\end{dfn}

In general $B_\phi$ is a subalgebra of $\bar{T}((V(R))$

\begin{prop}
If the dual tree generating set $\phi$ defines a dissection compatible
Hopf algebra, then \bas B_{\phi} = \Q[\{\Lambda_{\phi}(P) | P \;
  R-\textrm{deco polygon}\}] \eas is a sub-Hopf algebra of
$\bar{T}(V(R))$. \end{prop}

\begin{proof}
This is evident from the fact that the linearization map $\Lambda: \T
\rightarrow T(V_\bullet(R))$ is a Hopf algebra homomorphism.
\end{proof}

For $\Tsub{\phi}$ a dissection compatible Hopf algebra, I show that
$\Lambda_{\phi}(P)$ is contained in the $0^{th}$ cocycles
$B_{\D_\phi}$, that is $\Lambda_\phi(P) \in H^0(B_{\D_{}} (\mcp)$.

\begin{thm}
Let the dual tree generating set $\phi$ generate a dissection
compatible Hopf algebra $\Tsub{\phi}$. Let $\D_\phi$ be the associated
differential. For $P$ an $R$-deco polygon, $\Lambda_{\phi}(P)$ is a
$0$ cocycle of $D_1 + D_2$ in
$B_{\D_{\phi}}(\mcp)$.  \label{Bcocycle}\end{thm}

\begin{proof}
Let $\pi_k$ be the projection of $\Lambda_{\phi}(P)$ onto its $k^{th}$
direct sum component, \bas \pi_k: \Lambda_{\phi}(P) \rightarrow
\mcp^{|k} \;. \eas

The elements of $\Lambda_{\phi}(P) \in
\bar{T}(\mathcal{P}_\bullet^{(1)}(R))$. In general, if all $a_i \in
P_\bullet^{(1)}(R)$, the differentials $D_1$ and $D_2$ are \bas
D_1([a_1| \ldots | a_n]) = \sum_{i=1}^{n-1} - [a_1| \ldots| a_i \wedge
  a_{i+1}| \ldots | a_n] \; \\ D_2([a_1| \ldots | a_n]) = \sum_{j=1}^n
[a_1| \ldots| \D_{\phi}(a_j)| \ldots | a_n] \; \eas

Let $P$ be a polygon of weight $n$. The term $\pi_n \circ
\Lambda_{\phi}(P)$ is a sum of $n$-fold tensors of $2$-gons. Therefore
$D_2(\pi_n\circ\Lambda_{\phi}(P))= 0$. Furthermore, the term
$D_1(\pi_1\circ \Lambda_{\phi}(P))= 0$ by construction.

For $d \in D(P)$ with $|d|= k-1$, the dual tree $T_{\phi,d}(P)$ has
$k$ vertexes, labeled by the set $\{P_d^1, \ldots, P_d^k\}$. Write
\bas \Lambda(T_{\phi,d}(P)) = \sgn_{\phi}(d)
\sum_{\textrm{Lin.}(T_{\phi,d}(P))} [P_d^{\lambda_1}| \ldots|
  P_d^{\lambda_k}] \;.\eas Comparing $D_1(\pi_k\circ
\Lambda_{\phi}(P))$ to $D_2(\pi_{k-1}\circ \Lambda_{\phi}(P))$ for $k
\in \{2, \ldots, n\}$ gives the expressions \bas D_1(\pi_k\circ
\Lambda_{\phi}(P)) = \sum_{\begin{subarray}{c} d\in
    D(P)\\ |d|=k-1\end{subarray}}\sum_{\textrm{Lin.}(T_{\phi,d}(P))}\sum_{i=1}^k
-\sgn_\phi(d) [P_d^{\lambda_1}| \ldots|P_d^{\lambda_i} \wedge
  P_d^{\lambda_{i+1}}| \ldots | P_d^{\lambda_k}] \eas and \bas
D_2(\pi_{k-1}\circ \Lambda_{\phi}(P))= \sum_{\begin{subarray}{c} d'\in
    D(P)\\ |d'|=k-2\end{subarray}}\sum_{\textrm{Lin.}(T_{\phi,
    d'}(P))} \sum_{i=1}^{k-1} \sgn_\phi(d') [P_{d'}^{\lambda'_1}|
  \ldots| \D_\phi P_{d'}^{\lambda'_i}| \ldots | P_{d'}^{\lambda'_k}]
\; .\eas If $P_d^{\lambda_i}$ and $P_d^{\lambda_{i+1}}$ are not
adjacent in $T_{\phi, d}(P)$, then there exists a unique partial order
preserving linearization $\rho \in \textrm{Lin}(T_{\phi,d}(P))$ that
switches only the order in which those two terms are written:
$P_d^{\lambda_i} = P_d^{\rho_{i+1}}$, $P_d^{\lambda_{i+1}} =
P_d^{\rho_i}$ and $P_d^{\lambda_j} = P_d^{\rho_j}$ if $j \not \in \{i,
\, i+1\}$.  Therefore the terms \bas[P_d^{\lambda_1}| \ldots|
  P_d^{\lambda_i} \wedge P_d^{\lambda_{i+1}}| \ldots |
  P_d^{\lambda_k}] + [P_d^{\rho_1}| \ldots| P_d^{\rho_i} \wedge
  P_d^{\rho_{i+1}}| \ldots | P_d^{\rho_k}] = 0 \eas in the sum for
$D_1$. In the remaining terms for $D_1$, the polygons $P_d^{
  \lambda_i}$ is adjacent to $P_d^{\lambda_{i+1}}$.

Consider $d \in D(P)$ such that $|d|= k-1$ and $d'= d \setminus
\alpha$.  For each such pair of dissections, $\alpha$ is a dissecting
arrow of $P_{d'}^{\lambda_i'}$ for some $i$. Therefore, the term \bas
\D_\phi(P_{d'}^{\lambda_i'}) = \sum_{\alpha \in
  D(P_{d'}^{\lambda_i'})} \sgn_{\phi}(\alpha)
(P_{d'}^{\lambda_i'})_\alpha^1 \wedge (P_{d'}^{\lambda_i'})_\alpha^2
\eas appears in the expression for $D_2$. The polygons
$(P_{d'}^{\lambda_i'})_\alpha^1$ and $(P_{d'}^{\lambda_i'})_\alpha^2$
are adjacent in $T_{\phi, d}(P)$. There is a unique linear order of
$T_{\phi,d}(P)$, $\lambda$, such that \bas P_d^{\lambda_j}
= \begin{cases} P_{d'}^{\lambda_j'} \textrm{ if } j <
  i\\ (P_{d'}^{\lambda_j'})_\alpha^1 \textrm{ if } j = i
  \\ (P_{d'}^{\lambda_i'})_\alpha^2 \textrm{ if } j = i+ 1
  \\ P_{d}^{\lambda_{j-1}'} \textrm{ if } j > i+1 \end{cases} \; .
\eas That is \bas [P_d^{\lambda_1'}| \ldots|
  (P_{d'}^{\lambda_i'})_\alpha^1 \wedge (P_{d'}
  ^{\lambda_i'})_\alpha^2 | \ldots | P_d^{\lambda_k'}] =
     [P_d^{\lambda_1}| \ldots| P_d^{\lambda_i} \wedge
       P_d^{\lambda_{i+1}}| \ldots | P_d^{\lambda_k}] .\eas Since the
     right hand side appears with the sign $(-\sgn_\phi(d))$ in
     $D_1(\Lambda_\phi(P))$ and the left hand side with the sign
     $\sgn_\phi(d')\sgn_\phi( \alpha)$ in $D_2(\Lambda_\phi(P))$ these
     terms cancel.
\end{proof}

The coproduct on the Hopf algebras $B_{\phi}$ has a particularly
nice form. If the dual tree generating set $\phi$ generates a
dissection compatible Hopf algebra $\Tsub{\phi}$, the duality between
dissections $d \in D(P)$ and edges of a generator $T_{\phi,d}(P)
\in \phi$ gives a concept of an admissible dissection for the polygon
$P$.

\begin{lem}
Let $\Tsub{\phi}$ be a dissection compatible Hopf algebra. The
following are equivalent : \begin{enumerate}
\item \label{one} The dissection $c \in D(P)$ has a generator, $T_{\phi,c}(P)$,
  with only leaf and root vertexes labeled $\{P_c^1, \ldots , P_c^{|c|+1}\}$.
\item \label{two} The dissection $c \in D(P)$ is an admissible cut of
  $T_{\phi, d}(P)$, for any $d = c \cup (\bigcup_1^n d_i)$. Furthermore, \bas
  \Delta_c(T_{\phi, d}(P)) = \sgn_\phi(c)\prod_{j=1}^m T_{\phi, d_j}(P_c^j)
  \otimes \prod_{k=m+1}^{|c|+1} T_{\phi, d_k}(P_c^k) \eas \end{enumerate}
\label{specialdissections}\end{lem}

\begin{proof}
For $\ref{two} \Rightarrow \ref{one}$, notice that from the definition
of admissible cut, if $c$ is an admissible cut of $T_{\phi,d}(P)$, the
generator $T_{\phi,c}(P)$ has only root and leaf vertexes. The vertex
labels $ \{P_c^i\}$ come from the fact that $\Tsub{\phi}$ is a
dissection compatible Hopf algebra.

For $\ref{one} \Rightarrow \ref{two}$, let $\{P_c^1, \ldots ,
P_c^{|c|+1}\}$ be the $R$-deco polygons decorating
$T_{\phi,c}(P)$. For $c \subseteq d$, let $\{T_{d_1}, \ldots,
T_{d_{|c|+1}}\}$ be the root and leaf subtrees of $T_{\phi,d}(P)$
formed by removing the edges in $c$. Since the subtrees $T_{d_i}$ are
either only connected to initial vertexes of the edges in $c$, or only
to the final vertexes, $c$ is an admissible cut of $T_{\phi,d}(P)$.
Therefore, \bas \Delta_c(T_{\phi,d}(P))= \sgn_\phi(d) \prod_{i=1}^m
T_{d_i} \otimes \prod_{j=m+1}^{|c|+1} T_{d_j} \;.\eas The dissection
$d \in D(P)$ can be written \bas d = c \cup (\bigcup_1^{|c|+1} d_i)
\;,\eas with $d_i$ corresponding to the edges in $T_{d_i}$. Since
$\Tsub{\phi}$ is a dissection compatible Hopf algebra, $d_i \in
D(P_i)$ and $\sgn_\phi(d_i) T_{d_i} = T_{\phi, d_i}(P_c^i)$, and
$\sgn_\phi(d)= \sgn_\phi(c) \prod_{i=1}^{|c|+1} \sgn_\phi(d_i)$.
\end{proof}

\begin{dfn}
The dissection $c \in D(P)$ is an admissible dissection of $P$ in
$\phi$ if the generator $T_{\phi,c}(P)$, has only leaf or root
vertexes. \label{admisdis} \end{dfn}

This definition is Hopf algebra (i.e. $\phi$) specific. Consider two
dissection compatible Hopf algebras, $\Tsub{\phi}$ and
$\Tsub{\phi'}$. An admissible cut of $T_{\phi,d}(P)$ need not be an
admissible cut of $T_{\phi',d}(P)$.  For instance, the dissection $d =
\{\alpha, \beta\}$ of the polygon \bas P = 
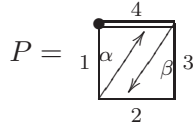
\xy \POS(0,4) *+{\bullet}
\POS(10,4) \ar@{=} +(-10,0)_4 \ar@{-} +(0,-10)^3 \POS(10,-6) \ar@{-}
+(-10,0)^2 \POS(0,4) \ar@{-} +(0,-10)_1 \POS(0,-6) \ar^{\alpha}@{->}
+(6,9) \POS(10,4) \ar^{\beta}@{->} +(-6,-9) \endxy \eas is an
admissible cut in $\Tsub{\phi_4}$ but not in $\Tsub{\phi_2}$.

\begin{thm}
Let $\Tsub{\phi}$ be a dissection compatible Hopf algebra. For $c$ an
admissible dissection of $P$ in $\phi$, let $\{P^1_c,\ldots,
P^m_c\}$ be the labels of roots vertexes, and
$\{P^{m+1}_c,\ldots,P^{|c|+1}_c\}$ the decorations of the leaf
vertexes. Then \bas \Delta ( \Lambda_{\phi}(P)) = \sum_{c \textrm{
    admis.}} \sha_{i=1}^m \Lambda_{\phi}(P^i_c) \otimes
\sha_{j=m+1}^{|c|+1} \Lambda_{\phi}(P^j_c) \; . \eas
\label{polydis}\end{thm}

\begin{proof}
Fix an admissible dissection $c$ of $P$ in $\Tsub{\phi}$. By lemma
\ref{specialdissections} and definition \ref{admisdis}, for any $d \in
D(P)$ such that $c \subset d$, write  $d=c\cup_{j=1}^{|c|+1}
d_i$ with $d_i \in D(P_c^i)$ and \bas \Delta T_{\phi,d}(P) = \sum_{ c
  \textrm{ admis.}}\prod_{i=1}^m T_{\phi, d_i}(P_c^i) \otimes
\prod_{j=m+1}^{|c|+1} T_{\phi, d_j}(P_c^j) \;. \eas

Write \bas \Delta \circ \Lambda_{\phi}(P) = \Delta \circ \Lambda
(\sum_{d \in D(P)}T_{\phi, d}(P)) = (\Lambda \otimes \Lambda)\circ
\Delta (\sum_{d \in D(P)}T_{\phi,d}(P))\;, \eas where the first
equality comes from the definition of $\Lambda_{\phi}(P)$ and the
second equality from theorem \ref{bialghomo}. Expanding
this, \begin{multline*} (\Lambda \otimes \Lambda) \Delta \sum_{d\in
    D(P)}T_{\phi,d}(P) = \\ (\Lambda \otimes \Lambda) (\sum_{d \in
    D(P)}\sum_{c \subset d \text{ admis.}}  \prod_{i=1}^m T_{\phi,
    d_i}(P_c^i) \otimes \prod_{j=m+1}^{|c|+1} T_{\phi, d_j}(P_c^j))
  \;.\end{multline*} Reorganizing terms and changing the order of
summation gives \bas \sum_{c \textrm{ admis.}}  \sha_{i=1}^m
\sum_{d_i \in D(P_c^i)}\Lambda(T_{\phi, d_i}(P^i_c)) \otimes
\sha_{j=m+1}^{|c|+1} \sum_{d_j \in D(P_c^j)}\Lambda(T_{\phi,
  d_j}(P^j_c)) = \\ \sum_{c \textrm{ admis.}}  \sha_{i=1}^m
\Lambda_{\phi}(P^i_c) \otimes \sha_{j=m+1}^{|c|+1}
\Lambda_{\phi}(P^j_c) \;.\eas
\end{proof}

\subsection{Properties of almost compatible algebras\label{props}}

In this subsection, I show that if the dual tree generating set $\psi$
defines an almost $\phi$ compatible algebra, $\Tsub{\psi}$, the
polygon algebra $B_\psi$ is a Hopf algebra, even if $\Tsub{\psi}$ is
not. First I show a relationship between the image of $\Lambda$ acting on general trees that differ on the orientation of the edges connecting certain
vertexes.

\begin{dfn}
Let $I$ be a subset of the edges of a tree $T$. Let $T^I$ be the tree
obtained from $T$ by reversing the orientation of the edges in $I$.
\end{dfn}

\begin{lem}
Let $T\in \T$ be a decorated multi-rooted tree. Let $I$ be a subset of
$n$ edges of the tree $T$. Let $F= \prod_{i=1}^{n+1} t_i$ be the forest
of multi-rooted trees created by removing the edges in $I$ in
$T$. Then \bas \sum_{J \subseteq I}\Lambda(T^J) =
\sha_{i=1}^{n+1} \Lambda(t_i) \; .\eas\label{trees}
\end{lem}

\begin{rem} The set $I$ need not be an admissible cut of $T$. \end{rem}

\begin{proof}
Write $I = \{e_1, \ldots, e_n\}$ with $v_{j1}$ and $v_{j2}$ the
endpoints of the edge $e_j$, such that $v_{j1}\prec v_{j2}$ in
$T$. If $e_j \in J$ then $v_{j2}\prec v_{j1}$ in $T^J$, and $v_{j1}
\prec v_{j2}$ if not.

If $t_k$ and $t_l$ are two trees in the forest $F$, the vertexes of
$t_k$ are incomparable to the vertexes of $t_l$ in $F$.  By
construction, each tree $t_k$ does not have both $v_{j1}$ and $v_{j2}$
as vertexes, for any $j$. Group $\Lambda(F)$ into sums of those terms
where $v_{j1}$ is to the left of $v_{j2}$ and sums of those where the
opposite is true. The relative positions of $v_{j2}$ and $v_{j1}$
correspond to the two orientations of the edge $e_j$. Since there are
2 choices for each pair, this divides the terms of $\Lambda(F)$ into
$2^n$ sums. This groups $\Lambda(F)$ into the sums in the statement of
the lemma.
\end{proof}

\begin{eg}
\begin{enumerate}
\item If $I = \{e\}$ is a single edge of a tree $T$, removing $e$ gives two subtrees, $\{R,P\}$. Under this notation,
  \bas \Lambda(T)+\Lambda(T^e)= \Lambda(R) \sha \Lambda(P)\; \eas
\item Let $T$ be multi-rooted tree formed by
  connecting the trees in the forest $F = \prod_{i=1}^n T_i$ to a new
  root with label $s$ at the vertex $v_i$ of $T_i$. For $n=3$, \bas T = 
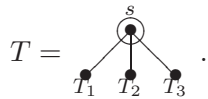
\xy
  \POS(0,3) *+{\bullet} *\cir{} \POS(0,3) \ar@{-} +(-6,-6) \ar@{-}
  +(0,-6) \ar@{-} +(6,-6) \POS(0,5) *{^{s}} \POS(-6,-3) *+{\bullet}
  \POS(0,-3) *+{\bullet} \POS(6,-3) *+{\bullet} \POS(-6,-5) *{^{T_1}}
  \POS(0,-5) *{^{T_2}} \POS(6, -5) *{^{T_3}} \endxy \;.\eas  Define $I =
  \{e_1, e_2, e_3\}$ to be the edges that connect $s$ to $v_i$. Then \bas
  \Lambda(T) + \Lambda( \xy \POS(0,6) *+{\bullet} *\cir{} \POS(0,6)
  \ar@{-} +(0,-6) \POS(0,8) *{^{T_1}} \POS(0,0) *+{\bullet} \POS(-2,0)
  *+{^s} \POS(0,0) \ar@{-} +(6,-6) \ar@{-} +(-6,-6) \POS(-6,-6)
  *+{\bullet} \POS(6,-6) *+{\bullet} \POS(6,-8) *{_{T_2}} \POS(-6, -8)
  *{_{T_3}} \endxy) + \Lambda( \xy \POS(0,6) *+{\bullet} *\cir{}
  \POS(0,6) \ar@{-} +(0,-6) \POS(0,8) *{^{T_2}} \POS(0,0) *+{\bullet}
  \POS(-2,0) *+{^s} \POS(0,0) \ar@{-} +(6,-6) \ar@{-} +(-6,-6)
  \POS(-6,-6) *+{\bullet} \POS(6,-6) *+{\bullet} \POS(6,-8) *{^{T_1}}
  \POS(-6, -8) *{^{T_3}} \endxy) + \Lambda( \xy \POS(0,6) *+{\bullet}
  *\cir{} \POS(0,6) \ar@{-} +(0,-6) \POS(0,8) *{^{T_3}} \POS(0,0)
  *+{\bullet} \POS(-2,0) *+{^s} \POS(0,0) \ar@{-} +(6,-6) \ar@{-}
  +(-6,-6) \POS(-6,-6) *+{\bullet} \POS(6,-6) *+{\bullet} \POS(6,-8)
  *{_{T_1}} \POS(-6, -8) *{_{T_2}} \endxy) + \Lambda( \xy \POS(-6,6)
  *+{\bullet} *\cir{} \POS(-6,6) \ar@{-} +(6,-6) \POS(-6,8) *{^{T_1}}
  \POS(6,6) *+{\bullet} *\cir{} \POS(6,6) \ar@{-} +(-6,-6) \POS(6,8)
  *{^{T_3}} \POS(0,0) *+{\bullet} \POS(-2,0) *+{^s} \POS(0,0) \ar@{-}
  +(0,-6) \POS(0,-6) *+{\bullet} \POS(0,-8) *{_{T_2}} \endxy) +\\
  \Lambda( \xy \POS(-6,6) *+{\bullet} *\cir{} \POS(-6,6) \ar@{-}
  +(6,-6) \POS(-6,8) *{^{T_3}} \POS(6,6) *+{\bullet}*\cir{} \POS(6,6)
  \ar@{-} +(-6,-6) \POS(6,8) *{^{T_2}} \POS(0,0) *+{\bullet}
  \POS(-2,0) *+{^s} \POS(0,0) \ar@{-} +(0,-6) \POS(0,-6) *+{\bullet}
  \POS(0,-8) *{^{T_1}} \endxy) +  \Lambda( \xy \POS(-6,6)
  *+{\bullet} *\cir{} \POS(-6,6) \ar@{-} +(6,-6) \POS(-6,8) *{^{T_2}}
  \POS(6,6) *+{\bullet} *\cir{} \POS(6,6) \ar@{-} +(-6,-6) \POS(6,8)
  *{^{T_1}} \POS(0,0) *+{\bullet} \POS(-2,0) *+{^s} \POS(0,0) \ar@{-}
  +(0,-6) \POS(0,-6) *+{\bullet} \POS(0,-8) *{_{T_3}} \endxy) +
  \Lambda( \xy \POS(0,-3) *+{\bullet} \POS(0,-3) \ar@{-} +(-6,6)
  \ar@{-} +(0,6) \ar@{-} +(6,6) \POS(0,-5) *{_{s}} \POS(-6,3)
  *+{\bullet} *\cir{} \POS(0,3) *+{\bullet} *\cir{} \POS(6,3)
  *+{\bullet} *\cir{} \POS(-6,5) *{^{T_1}} \POS(0,5) *{^{T_2}}
  \POS(6,5) *{^{T_3}} \endxy) \\ = [s \sha \Lambda(T_1) \sha \Lambda(T_2)
    \sha \Lambda(T_3)]\; . \eas 
\end{enumerate}

\end{eg}

I use lemma \ref{trees} to show a relationship between the polygon
algebras $B_\phi$ and $B_\psi$, where $\Tsub{\psi}$ is almost
$\phi$ compatible, and $\Tsub{\phi}$ is a dissection compatible
Hopf algebra.

\begin{thm}
Let $\phi$ and $\psi$ be two dual tree generating sets that define the
same differential $\D_\phi= \D_\psi$. Let $\Tsub{\phi}$ be a
dissection compatible Hopf algebra and $\Tsub{\psi}$ an almost $\phi$
compatible algebra. Let $\mathcal{S}$ be their difference set. Let $P$
be an $R$-deco polygon. For $d \in D(P)$, let $\{P^1_d, \ldots ,
P^{|d|+1}_d\}$ be the polygons decorating $T_{\phi, d}(P)$. Then \ba
\Lambda_\psi(P) - \Lambda_\phi(P) =  \sum_{\begin{subarray}{c}d
     \subset \mathcal{S}  \\
    \emptyset \neq d \in D(P) \end{subarray}} (-1)^{|d|}
\sgn_\phi(d)\sha_{j=1}^{|d|+1} \Lambda_\phi(P_d^j)
\;. \label{edgeswitchrel} \ea for all $P \in
V_\bullet(R)$. \label{edgeswitch} \end{thm}

\begin{proof}
For any $d \in D(P)$, let $I(d) = d \cap \mathcal{S}$.  Write the
generator \bas T_{\phi, d}(P) = \sgn_\phi(d) T\;.\eas Using the
notation from lemma \ref{trees}, \bas T^{I(d)}_{\phi, d}(P) =
\sgn_\phi(d) T^{I(d)} \;.\eas By the definition of almost $\phi$
compatible algebras, \bas (-1)^{|I(d)|} T_{\phi, d}(P) =
T^{I(d)}_{\psi, d}(P)\; \eas and \ba T_{\phi, d}(P) = T_{\psi, d}(P)
\Leftrightarrow I(d) = \emptyset \; \label{noedgeswitch}.\ea Write the
left hand side of \eqref{edgeswitchrel} as \bas \sum_{d' \in D(P)}
\left( (-1)^{|I(d')|}\Lambda(T_{\phi, d'}^{I(d')}(P))-
\Lambda(T^\emptyset_{\phi, d'}(P)) \right) \;. \eas By
equation \eqref{noedgeswitch}, I can ignore the
dissections $d'$ that don't intersect the set $\mathcal{S}$, \ba \sum_{d' \in D(P), I(d') \neq
  \emptyset}\left( (-1)^{|I(d')|} \Lambda (T_{\phi, d'}^{I(d')}(P)) -
\Lambda (T^\emptyset_{\phi, d'}(P)) \right) \;.\label{LHS}\ea

Write the right hand side of \eqref{edgeswitchrel} as \ba
\sum_{\begin{subarray}{c}d \subset \mathcal{S} \\ \emptyset \neq d \in
    D(P) \end{subarray}} (-1)^{|d|}\sgn_\phi(d) \sha_{j=1}^{|d|+1}
\sum_{d_j \in D(P_d^j)} \Lambda(T_{\phi, d_j}(P_d^j)) \;
,\label{rhs1}\ea Since $\Tsub{\phi}$ is a dissection compatible Hopf
algebra, write any subdissection $d \subset d' \in D(P)$, as $d' = d
\cup (\cup_{j=1}^{|d|+1} d_j)$ with $d_j \in D(P_d^j)$. The forest
\bas \prod_{j=1}^{|d|+1}T_{\phi, d_j}(P_d^j)\eas comes from to cutting
$T_{\phi, d'}(P)$ at the edges corresponding to $d$. Furthermore, \bas
\sgn_\phi(P,d') = \sgn_\phi(P,d)(\prod_{j=1}^{d+1} \sgn_\phi(P_d^j,
d_j)) \;. \eas Using lemma \ref{trees}, rewrite equation \eqref{rhs1}
as \ba \sum_{\begin{subarray}{c}d \subset \mathcal{S} \\ \emptyset
    \neq d \in D(P) \end{subarray}} \sum_{d\subseteq d'\in D(P)}
(-1)^{|d|} \left( \sum_{\delta \subseteq d} \Lambda(T^\delta _{\phi,
  d'}(P))\right) \; .  \label{rhs2}\ea I do not consider dissections
$d'$ such that $d' \cap \mathcal{S} = I(d') = \emptyset$, since this
implies $\delta = d = 0$. I reorganize the expression in \eqref{rhs2}
to \bas \sum_{\begin{subarray}{c} d'\in D(P)\\ I(d') \neq
    \emptyset \end{subarray}} \sum_{\begin{subarray}{c} \emptyset \neq
    d \subseteq I(d') \\ \delta \subseteq d \end{subarray}}(-1)^{ |d|}
\Lambda(T^\delta_{\phi, d'}(P)) = \\ \sum_{\begin{subarray}{c} d'\in
    D(P)\\ I(d') \neq \emptyset \end{subarray}} \sum_{\delta}
\sum_{ i = \max \{1,|\delta|\}}^{|I(d')|} (-1)^i
    {|I(d')\setminus \delta| \choose i - |\delta|}
    \Lambda(T^\delta_{\phi, d'}(P)) \;. \eas For any fixed $\delta
    \not \in \{ \emptyset, I(d')\}$, this expression vanishes, leaving
    \bas \sum_{\begin{subarray}{c} d'\in D(P)\\ I(d') \neq
        \emptyset \end{subarray}} \left(
    (-1)^{|I(d')|}\Lambda(T^{I(d')}_{\phi, d'}(P)) -
    \Lambda(T^\emptyset_{\phi, d'}(P))\right) \;, \eas to match
    \eqref{LHS}. The first term corresponds to the case where $\delta=
     I(d')$. The second term corresponds to the case where $\delta =
    \emptyset$.
\end{proof}

The following corollary shows that that $B_\phi$ and $B_\psi$ are
isomorphic as Hopf algebras. The generators $\{\Lambda_\phi(P)\}$
and $\{\Lambda_\psi(P)\}$ define different bases of this vector space
underlying $B_\phi$.

\begin{cor}
Let $\phi$, $\psi$ and $\mathcal{S}$ be as in theorem \ref{edgeswitch}.  Then
$B_\psi \simeq B_\phi$ as Hopf algebras  \end{cor}

\begin{proof}
The product and coproduct structure on $B_\psi$ is induced by the
product and coproduct structures on $B_\phi$ and equation
\eqref{edgeswitchrel}. Therefore $B_\psi$ is a Hopf algebra. In fact,
equation \eqref{edgeswitchrel} shows that every generator,
$\Lambda_\psi(P)$ of $B_\psi$, can be written in terms of sums of
shuffles of generators of $B_\phi$. It remains to show that the map
defined by this equation can be inverted.

By theorem \ref{edgeswitch}, write \bas \Lambda_\phi(P) =
\Lambda_\psi(P) - \sum_{\begin{subarray}{c}d \subset \mathcal{S}
    \\ \emptyset \neq d \in D(P) \end{subarray}}
(-1)^{|d|}\sgn_\phi(d) \sha_{j=1}^{|d|+1} \Lambda_\phi(P_d^j) \;. \eas
where $\{P_d^1, \ldots, P_d^{|d|+1|}\}$ is the set of polygons
decorating $T_{\phi,d}(P)$ and $T_{\psi, d}(P)$. If $P \in V_1(R)$,
that is, it is a 2-gon, \bas \Lambda_\phi(P) = \Lambda_\psi(P)
\;. \eas If $P \in V_2(R)$, \bas \Lambda_\phi(P) = \Lambda_\psi(P)+
\sum_{\begin{subarray}{c}d \subset \mathcal{S} \\ \emptyset \neq d \in
    D(P) \end{subarray}} \sgn_\phi(d) \Lambda_\phi(P_d^1)\sha
\Lambda_\phi(P_d^2) \eas where $P_d^1, P_d^2 \in V_1(R)$. Therefore,
\bas \Lambda_\phi(P) = \Lambda_\psi(P) + \sum_{\begin{subarray}{c}d
    \subset \mathcal{S} \\ \emptyset \neq d \in D(P) \end{subarray}}
\sgn_\phi(d) \Lambda_\psi(P_d^1)\sha \Lambda_\psi(P_d^2) \;.  \eas By
induction, suppose $\Lambda_\phi(P) \in B_\psi$ for all $P$ of weight
less than $n$. If $P \in V_n(R)$, \bas \Lambda_\phi(P) =
\Lambda_\psi(P) - \sum_{\begin{subarray}{c}d \subset \mathcal{S} \\ \emptyset
      \neq d \in D(P) \end{subarray}} (-1)^{|d|}\sgn_\phi(d)
\sha_{j=1}^{|d|+1} \Lambda_\phi(P_d^j) \;. \eas Since each
$P_d^j$ has weight less than $n$, $\Lambda_\phi(P_d^j) \in
B_\psi$. Thus $\Lambda_\phi(P)$ can be written in terms of sums of
shuffles of elements in $B_\psi$. Therefore, $\Lambda_\phi(P) \in
B_\psi$, and there is a one to one correspondence between the
generators of $B_\phi$ and $B_\psi$.  \end{proof}

Under these conditions, if, for all polygon dissection pairs $(P,d)$
every tree of the form $T_{\phi, d \cap \mathcal{S}}(P)$ is linear, then the
result of theorem \ref{edgeswitch} simplifies greatly.

\begin{cor}
If in addition to the conditions for $\phi$, $\psi$ and $\mathcal{S}$
above, $T_{\phi, d\cap \mathcal{S}}(P)$ is a linear tree with sign, for all $P \in
V_\bullet(R)$, then \bas \Lambda_\phi(P) - \Lambda_\psi(P) =
\sum_{\alpha\in \mathcal{S}}\sgn_\phi(\alpha) \Lambda_\phi
(P_{\alpha}^1)\sha \Lambda_\psi (P_{\alpha}^2)\;,
\eas\label{lineartrees}
\end{cor}

\begin{proof}
For $P$ an $R$-deco polygon, fix an $\alpha \in D(P)$ such that
$\alpha \in \mathcal{S}$. Consider all $d \subset \mathcal{S}$ such
that $\alpha$ is dual to the edge attached to the root in $T_{\phi,
  d}(P)$.  Then the dissection $\{d \setminus \alpha \} \in
D(P_\alpha^2)$.  Let \bas \rho_\alpha = \{\emptyset \neq d \in D(P) | d \subset
\mathcal{S}, \alpha \in d, P_\alpha^1 \textrm{ root label of }
T_{\phi, d}(P) \}\eas be the set of all such $d$. Let $ \{P_d^1,
\ldots,P_d^{|d|+1}\}$ be the polygons decorating the generator
$T_{\phi,d}(P)$ enumerated such that $P_\alpha^1 = P_d^1$ labels the
vertex of $T_{\phi, d}(P)$. From theorem \ref{edgeswitch}, \begin{multline}
\Lambda_\phi(P) - \Lambda_\psi(P) =  \\-\sum_{\emptyset \neq d \subseteq
  \mathcal{S}\cap D(P)} (-1)^{|d|}\sgn_\phi(d) \sha_{j=1}^{|d|+1}\Lambda_\phi
(P^j_d) = \\ -\sum_{\alpha\in \mathcal{S}\cap D(P)} \sum_{ d \in
  \rho_\alpha}(-1)^{|d|}\sgn_\phi(d) \Lambda_\phi (P_\alpha^1)
\sha_{j=2}^{|d|+1} \Lambda_\phi (P^j_d) \; .\label{intermlinres}\end{multline} If
$d= \alpha$, then $P^2_d = P^2_\alpha$. Break the sum in the last line
of equation \eqref{intermlinres} as \begin{multline*} \sum_{\alpha\in
  \mathcal{S}\cap D(P)}\sgn_\phi(\alpha) \Lambda_\phi (P_\alpha^1) \sha
\left(\Lambda_\phi (P_\alpha^2) + \right.\\ \left. \sum_{ d \in \rho_\alpha, d\neq
  \alpha} (-1)^{|d|-1}\sgn_\phi(d \setminus \alpha))
\sha_{j=2}^{|d|+1}\Lambda_\phi (P^j_d) \right)\;. \end{multline*} Since
$d\setminus \alpha \in D(P_\alpha^2)$, by theorem \ref{edgeswitch}
\bas \Lambda_\psi (P_\alpha^2) = \Lambda_\phi(P_\alpha^2) + \sum_{ d
  \in \rho_\alpha, d\neq \alpha} (-1)^{|d|-1}\sgn_\phi(d \setminus
\alpha)) \sha_{j=2}^{|d|+1}\Lambda_\phi (P^j_d) \;. \eas This gives
\bas \Lambda_\phi(P)-\Lambda_\psi(P) = \sum_{\alpha\in
  \mathcal{S}\in D(P)}\sgn_\phi(\alpha) \Lambda_\phi (P_\alpha^1) \sha
\Lambda_\psi (P_\alpha^2)\;. \eas
\end{proof}

\begin{cor}
Under the same conditions as the previous corollary, one can also
write \bas \Lambda_\phi(P) - \Lambda_\psi(P) = \sum_{\alpha \in
  \mathcal{S}\cap D(P)}\sgn_\phi(\alpha) \Lambda_\psi (P_\alpha^1) \sha \Lambda_\phi
(P_\alpha^2)\;. \eas\label{lineartreesleaf}
\end{cor}

\begin{proof}
This result comes from the same argument as above, replacing
the edge connected to the single root vertex of $T_{\phi,d}(P)$, for
$d \subset \mathcal{S}$ with the edge connected to the single leaf vertex.
\end{proof}

\section{Permutations of a polygon\label{permutesection}}

In this section, I examine the actions of $\sigma$ and $\tau$ on the
Hopf algebra $\Lambda_{\phi_2}$. Recall that $\sigma$ and $\tau$ are
linear automorphisms on $V_\bullet(R)$ such that for $P = 12\ldots n$,
$\tau(P) = (n-1)\ldots 21n$ reverses the orientation of $P$ and
$\sigma(P) = 2\ldots n1$ rotates the labels of the edges one
position. Restricted to the sub-vector space $V_n(R)$,
$\sigma|_{V_n(R)}$ and $\tau|_{V_n(R)}$ generate the dihedral group
$D_{2n+2}$. I can extend $\sigma$ and $\tau$ to automorphisms of
$B_{\phi_2}$ by defining $\sigma(\Lambda_{\phi_2}(P)) =
\Lambda_{\phi_2}(\sigma P))$ and $\tau(\Lambda_{\phi_2}(P)) =
\Lambda_{\phi_2}(\tau P)$. After defining relations between
$\Lambda_{\phi_2}(\sigma P)$, $\Lambda_{\phi_2}(\tau P)$ and
$\Lambda_{\phi_2}(P)$, one can apply the coalgebra homomorphism \bas
\Phi : \Lambda(\Tsub2) \rightarrow \mathcal{I}_\bullet(R)\eas defined in
\cite{GGL05} to establish relationships between iterated integrals
with the appropriate dihedral action on the arguments.

\subsection{Order $2$ generator of the dihedral group}

First I calculate $\Lambda_{\phi_2}(P) \pm \Lambda_{\phi_2}(\tau
P)$. Since $\tau$ fixes the label of the root side of the polygon $P$,
it is useful to examine an almost $\phi_2$ compatible algebra
$\Tsub{\psi}$ such that the difference set between $\phi_2$ and $\psi$
consists of arrows ending on the root side. This is exactly the
algebra $\Tsub{\phi_{re}}$ discussed in example \ref{re}.

\begin{lem}
If $P$ is an $R$-deco polygon of weight $n$, \ba \Lambda_{\phi_2}(P)-
\Lambda_{\phi_{re}}(P) = \sum_{i=2}^{n} \Lambda_{\phi_2} (P_{_i\alpha}^\rt)) \sha
\Lambda_{\phi_{re}} (P_{_i\alpha}^\rst)\; . \label{redifeq}\ea \label{redif}
\end{lem}

\begin{proof}
For $i < j$, ${_j}\alpha \in D(P_{_i\alpha}^\rst)$ and $_i\alpha \in
D(P^\rt_{_j\alpha})$. If $d = \{_i\alpha, {_j}\alpha\} \in D(P)$,
$T_{\phi_{re}, d}(P)$ and $T_{\phi_2, d}(P)$ are linear, as are the
trees for and dissection $d \subset re(P)$. The arrows ${_j}\alpha$ are
forwards, so $\sgn_{\phi_2}({_j}\alpha) = 1$. The result follows from
corollary \ref{lineartrees}.
\end{proof}

Extend the action of $\tau$ to dissecting arrows and their associated
sets of subpolygons. 

Let $\alpha$ be a dissecting arrow of $P$, a polygon of weight $n$.
Using the notation in definition \ref{arrowdef}, if $\alpha \not \in
re(P)$, write $\alpha = {_i}\alpha_j$ (for $j \neq n+1$). The map
$\tau$ sends the polygon $P$ to $\tau(P)$ and the arrow ${_i}\alpha_j$
to $\tau\alpha = {_{n-i+2}}\alpha_{n-j+1} \in D(\tau(P))$. For a root
ending arrow $\alpha = {_i}\alpha_{n+1} \in re(P)$, $\tau\alpha =
{_{n-i+2}}\alpha_{n+1} \in D(\tau(P))$. For a forward (backward) arrow
$\alpha \not \in re(P)$, the arrow $\tau\alpha$ is backward
(forward). The map $\tau$ take $re(P)$ to itself. All arrows in
$re(P)$ are forward. The following is an example for $P$ of weight $5$.

\begin{eg}
Recall that I use the shorthand ${_i}\alpha$ to indicate the root
ending arrow ${_i}\alpha_{n+1} \in re(P)$ (for $ P \in V_n(R)$). Let
$P= 123456$, and $d = \{{_3}\alpha, \beta\}$. Then
$\tau{d}= \{{_4}\alpha, \tau\beta\}$. Below are diagrams of $P$ and
$\tau P$ with the dissections $d$ and $\tau d$ drawn in.

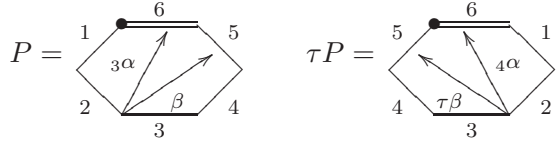
\bas 
P=
\xy
\POS(0,4) *+{\bullet}
\POS(10,4) \ar@{=} +(-10,0)_6
\ar@{-} +(6,-6)^5
\POS(16,-2)\ar@{-} +(-6,-6)^4
\POS(0,-8)\ar@{-} +(10,0)_3
\ar@{-} +(-6,6)^2
\POS(0,4)\ar@{-} +(-6,-6)_1
\POS(0,-8) \ar^{{_3}\alpha}@{->} +(6,11)
\POS(0,-8) \ar_{\beta}@{->} +(12, 8)
\endxy
\quad \quad \tau{P}= 
\xy
\POS(0,4) *+{\bullet}
\POS(10,4) \ar@{=} +(-10,0)_6
\ar@{-} +(6,-6)^1
\POS(16,-2)\ar@{-} +(-6,-6)^2
\POS(0,-8)\ar@{-} +(10,0)_3
\ar@{-} +(-6,6)^4
\POS(0,4)\ar@{-} +(-6,-6)_5
\POS(10,-8) \ar_{{_4}\alpha}@{->} +(-6,11)
\POS(10,-8) \ar^{\tau \beta}@{->} +(-12, 8)
\endxy
\eas

The arrows ${_3}\alpha, {_4}\alpha$ are in $re(P)$ and $re(\tau(P))$
respectively. The subpolygons associated to ${_3}\alpha$ and
${_4}\alpha$ are \bas \tau(P_{_3\alpha}^\rt) = 216 = (\tau
P)^\rst_{{_4}\alpha} \quad ; \quad (\tau P)^\rt_{{_4}\alpha}=
5436 = \tau(P_{_3\alpha}^\rst) \eas The subpolygons associated to
$\beta$ and $\tau \beta$ are \bas (\tau P)_{\tau\beta}^\rt= 5216 =
\tau(P_\beta^\rt)\quad ; \quad \tau P_\beta^\rst = 435 = (\tau
P)_{\tau\beta}^\rst \; .\eas
\end{eg}

For a general dissection of an arbitrary polygon, $d \in D(P)$, write
$d \cap re(P) = \{{_{i_1}}\alpha \ldots {_{i_j}}\alpha\}$, with $i_1 <
i_2 \ldots < i_j$. Let $\{P_d^0, \ldots P_d^{|d|}\}$ be the set of
polygons labeling the vertexes of $T_{\phi_2, d}(P)$ and
$T_{\phi_{re}, d}(P)$. Enumerate the set such that for $m \leq j$,
$P_d^{m-1}$ and $P_d^m$ are adjacent, connected by ${_{i_m}}\alpha$
with $P_d^{m-1} \prec P_d^m$ in $T_{\phi_2, d}(P)$ and $P_d^m \prec
P_d^{m-1}$ in $T_{\phi_{re}, d}(P)$. The set of polygons labeling
$T_{\phi_2, \tau d}(\tau P)$ is  $\{\tau P_d^0, \ldots \tau
P_d^j, P_d^{j+1}, \ldots, P_d^{|d|}\}$, with $\tau P_{m} \prec \tau
P_{m-1}$ connected by $\tau {_{i_m}}\alpha$. If $d\cap re(P)=
\emptyset$ then $P_0$ is the label of the single root of all three
generators. If $\beta \not \in re(P)$, write $d_\beta = d \cap
D(P_\beta^\rst)$. The pruned subtrees of the admissible cut $c=\beta$
in $T_{\phi_2, d}(P)$, $T_{\phi_{re}, d}(P)$ and $T_{\phi_2, \tau
  d}(\tau P)$ are the same: $T_{\phi_2, d_\beta}(P_\beta^\rst)$ if
$\beta$ is a forwards arrow, and $T_{\phi_2, d_\beta}(\tau
P_\beta^\rst)$ if $\beta$ is a backwards arrow.

I summarize this in the following diagrams. Here $\beta$, $\gamma$,
$\delta$, and $\epsilon$ are assumed to be forward arrows. For
backwards arrows, replace $P_\beta^\rst$ with $\tau(P_\beta^\rst)$ and
$d_\beta$ with $\tau d_\beta$.

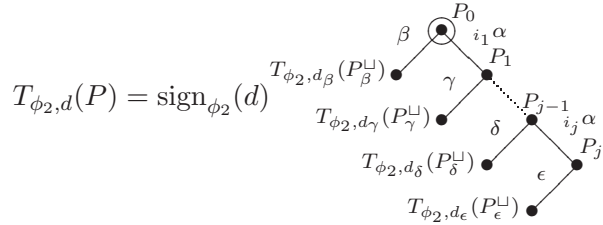
\bas T_{\phi_2, d}(P) = \sgn_{\phi_2}(d) \xy\POS(0,9)
*+{\bullet} *\cir{} \POS(3,11) *{^{P_0}} \POS(0,9) \ar@{-}_{\beta}
+(-6,-6) \ar@{-}^{{_{i_1}}\alpha} +(6,-6) \POS(6,3) *+{\bullet}
\POS(-6,3) *+{\bullet} \POS(8,5) *{^{P_1}} \POS(-15,3) *{^{T_{\phi_2,
      d_\beta}(P_\beta^\rst)}} \POS(6,3) \ar@{-}_{\gamma} +(-6,-6)
\ar@{.} +(6,-6) \POS(0,-3) *+{\bullet} \POS(-9,-3) *{^{T_{\phi_2,
      d_\gamma}(P_\gamma^\rst)}} \POS(12, -3) *+{\bullet} \POS(14,-1)
*{^{P_{j-1}}} \POS(12, -3) \ar@{-}_{\delta} +(-6,-6)
\ar@{-}^{{_{i_j}}\alpha} +(6,-6) \POS(6,-9) *+{\bullet} \POS(18,-9)
*+{\bullet} \POS(-3,-9) *{^{T_{\phi_2, d_\delta}(P_\delta^\rst)}}
\POS(20,-7) *{^{P_j}} \POS(3,-15) *{^{T_{\phi_2,
      d_\epsilon}(P_\epsilon^\rst)}} \POS(18,-9) \ar@{-}_{\epsilon}
+(-6,-6) \POS(12,-15) *+{\bullet} \endxy \eas
\vspace{.5cm} 
\bas T_{\phi_{re}, d}(P) =
(-1)^j\sgn_{\phi_2}(d) 
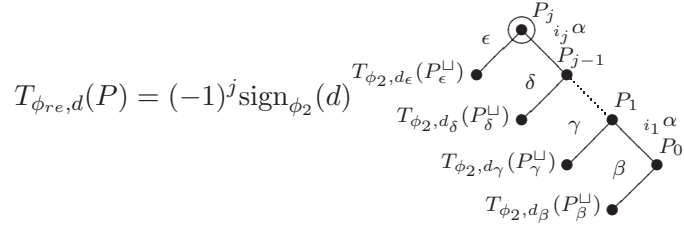
\xy\POS(0,9) *+{\bullet} *\cir{} \POS(3,11)
*{^{P_j}} \POS(0,9) \ar@{-}_{\epsilon} +(-6,-6)
\ar@{-}^{{_{i_j}}\alpha} +(6,-6) \POS(6,3) *+{\bullet} \POS(-6,3)
*+{\bullet} \POS(8,5) *{^{P_{j-1}}} \POS(-15,3) *{^{T_{\phi_2,
      d_\epsilon}(P_\epsilon^\rst)}} \POS(6,3) \ar@{-}_{\delta}
+(-6,-6) \ar@{.} +(6,-6) \POS(0,-3) *+{\bullet} \POS(-9,-3)
*{^{T_{\phi_2, d_\delta}(P_\delta^\rst)}} \POS(12, -3) *+{\bullet}
\POS(14,-1) *{^{P_1}} \POS(12, -3) \ar@{-}_{\gamma} +(-6,-6)
\ar@{-}^{{_{i_1}}\alpha} +(6,-6) \POS(6,-9) *+{\bullet} \POS(18,-9)
*+{\bullet} \POS(-3,-9) *{^{T_{\phi_2, d_\gamma}(P_\gamma^\rst)}}
\POS(20,-7) *{^{P_0 }} \POS(3,-15) *{^{T_{\phi_2,
      d_\beta}(P_\beta^\rst)}} \POS(18,-9) \ar@{-}_{\beta} +(-6,-6)
\POS(12,-15) *+{\bullet} \endxy \eas 
\vspace{.5cm} 
\bas T_{\phi_2, \tau d}(\tau P) =
\sgn_{\phi_2}(\tau d) 
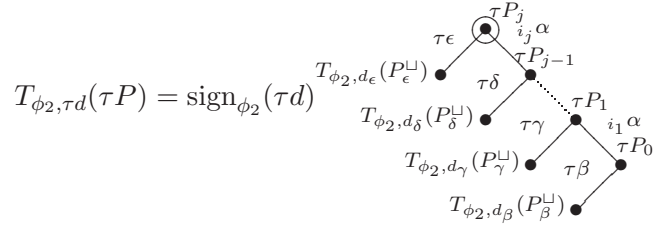
\xy\POS(0,9) *+{\bullet} *\cir{} \POS(3,11)
*{^{\tau P_j}} \POS(0,9) \ar@{-}_{\tau \epsilon} +(-6,-6)
\ar@{-}^{{_{i_j}}\alpha} +(6,-6) \POS(6,3) *+{\bullet} \POS(-6,3)
*+{\bullet} \POS(8,5) *{^{\tau P_{j-1}}} \POS(-15,3) *{^{T_{\phi_2,
      d_\epsilon}(P_\epsilon^\rst)}} \POS(6,3) \ar@{-}_{\tau \delta}
+(-6,-6) \ar@{.} +(6,-6) \POS(0,-3) *+{\bullet} \POS(-9,-3)
*{^{T_{\phi_2, d_\delta}(P_\delta^\rst)}} \POS(12, -3) *+{\bullet}
\POS(14,-1) *{^{\tau P_1}} \POS(12, -3) \ar@{-}_{\tau \gamma} +(-6,-6)
\ar@{-}^{{_{i_1}}\alpha} +(6,-6) \POS(6,-9) *+{\bullet} \POS(18,-9)
*+{\bullet} \POS(-3,-9) *{^{T_{\phi_2, d_\gamma}( P_\gamma^\rst)}}
\POS(20,-7) *{^{\tau P_0}}\POS(3,-15) *{^{T_{\phi_2,
      d_\beta}(P_\beta^\rst)}} \POS(18,-9) \ar@{-}_{\tau \beta}
+(-6,-6) \POS(12,-15) *+{\bullet} \endxy \eas

The generators $T_{\phi_2, d}( P)$ and $T_{\phi_{re}, d}(P)$ have
different signs and different underlying tree structures, with labels
$\{P_d^0, \ldots P_d^{|d|}\}$. On the other hand, the generators
$T_{\phi_2,\tau d}(\tau P)$ and $T_{\phi_{re}, d}(P)$ have different
signs, but the same underlying tree structure, if one exchanges
$P_d^m$ with $\tau P_d^m$ for $m\leq j$.  For any dissection $d \in
D(P)$, \bas \sgn_{\phi_{re}}(d) = (-1)^{|d \cap
  re(P)|}\sgn_{\phi_2}(d) = (-1)^{\sum_{\alpha \in d, \alpha
    \textrm{bw}}\chi(\alpha)}(-1)^{|d \cap re(P)|}\eas and \bas
\sgn_{\phi_2}(\tau d) = (-1)^{\sum_{\alpha\in d \setminus re(P),
    \alpha \textrm{fw}}\chi(\alpha)} \;. \eas

Recall that a coideal, $C$, of a coalgebra $(\mathcal{H}, \Delta,
\varepsilon)$ has the structure \bas \Delta (C) \subset \mathcal{H}
\otimes C + C\otimes \mathcal{H} \eas and $\varepsilon(C) = 0$.  It is
a primitive coideal if \bas \Delta(c_i) = 1 \otimes c_i + c_i \otimes
1 \; \eas for all generators of $C$.

\begin{dfn}
Let $I_n\subset \mathcal{P}_\bullet^{(1)}(R)$ be the linear subspace
generated by $\{P +(-1)^n \tau P | P \textrm{ polygon of weight }
n\}$. \end{dfn}

Notice that $I_1 = 0$ is the trivial co-ideal. Each $I_n$ is a
primitive co-ideal in $B_{\phi_2}$.

\begin{dfn}
Define a set of quotient maps \bas q_n: \bar T(V(R)) \rightarrow
\bar T(V(R))/(\sum_{k=1}^n I_k) \;.\eas \label{qndef}
\end{dfn}

\begin{thm}
Let $P$ be an $R$-deco polygon of weight $n$. Let
$q_n$ be the quotient map defined above. For $P \in V_n(R)$,
\ba \Lambda_{\phi_{re}}(P) + (-1)^n \Lambda_{\phi_2}(\tau
P)\in \ker q_n \;.\label{orientsigneq}\ea \label{orientsign}\end{thm}

\begin{proof}
If $P$ is a polygon of weight 2, $\Lambda_{\phi_{re}}(P) -
\Lambda_{\phi_2}(\tau P) = 0$.  For $P = abc \in
V_2(R)$,
\begin{equation*} \begin{array}{ccccc} \; & \threegon{a}{b}{c}&
  \threegonarrowb{a}{b}{c} & \threegonarrowc{a}{b}{c} &
  \threegonarrowa{a}{b}{c}\\ &&&& \\ \Lambda_{\phi_{re}}(P) = & P &+
                 [bc|ab] & - [bc|ac] & - [ac|ba] \; , \end{array}
\end{equation*}
While for $\tau P$, 
\begin{equation*}
\begin{array}{ccccc} \; & \threegon{b}{a}{c} &
  \threegonarrowb{b}{a}{c}& \threegonarrowc{b}{a}{c} &
  \threegonarrowa{b}{a}{c} \\ &&&& \\\Lambda_{\phi_2}(\tau P) = & \tau P
  &+ [ac|ba] &+ [bc|ac] &- [bc|ab] \; .\end{array}
\end{equation*}
 Therefore \bas \Lambda_{\phi_{re}}(P) + \Lambda_{\phi_{2}}(\tau P)= P
 + \tau P \in I_2\;. \eas Suppose the theorem holds for all $k < n$.

Let $P$ be an $R$-deco polygon of weight $n$. Consider the
dissections $d \in D(P)$, with $d \cap re(P) = \{{_{i_1}}\alpha \ldots
{_{i_j}}\alpha\}$. Let $\{P_d^0, \ldots P_d^{|d|}\}$ be the set of
polygons decorating $T_{\phi_{re}, d}(P)$ and $T_{\phi_2, d}(P)$, with
each $P_d^i \in V_{n_i}(R)$ with $\sum_{i=0}^{|d|} n_i = n$. For
$m\leq j$, the polygons $P_d^{m-1}$ and $P_d^m$ are adjacent,
connected by ${_{i_m}}\alpha$, and $P_d^m \prec P_d^{m-1}$ in
$T_{\phi_{re},d}(P)$ and $P_d^{m-1} \prec P_d^m$ in
$T_{\phi_2,d}(P)$. I define a series of trees (with sign)
$\{T_{m,d}(P)\}$, with $m \leq j$ formed by replacing the polygons
$\{P_d^0, \ldots , P_d^m\}$ in $T_{\phi_{re}, d}(P)$ with the polygons
$\{(-1)^{n_0}\tau P_d^0, \ldots ,(-1)^{n_m} \tau P_d^m\}$. In this
series, $T_{\phi_{re}, d}(P) = T_{-1, d}(P)$. For example \bas
T_{0,d}(P) = \sgn_{\phi_{re}}(d)(-1)^{n_0} \xy\POS(0,9) *+{\bullet}
*\cir{} \POS(3,11) *{^{P_j}} \POS(0,9) \ar@{-} +(-6,-6) \ar@{-}
+(6,-6) \POS(6,3) *+{\bullet} \POS(-6,3) *+{\bullet} \POS(9,5)
*{^{P_{j-1}}} \POS(-15,3) *{^{T_{\phi_2,
      d_\epsilon}(P_\epsilon^\rst)}} \POS(6,3) \ar@{-} +(-6,-6)
\ar@{.} +(6,-6) \POS(0,-3) *+{\bullet} \POS(-9,-3) *{^{T_{\phi_2,
      d_\delta}(P_\delta^\rst)}} \POS(12, -3) *+{\bullet} \POS(15,-1)
*{^{P_1}} \POS(12, -3) \ar@{-} +(-6,-6) \ar@{-} +(6,-6) \POS(6,-9)
*+{\bullet} \POS(18,-9) *+{\bullet} \POS(-3,-9) *{^{T_{\phi_2,
      d_\gamma}( P_\gamma^\rst)}} \POS(22,-7) *{^{\tau
    P_0}}\POS(3,-15) *{^{T_{\phi_2, d_\beta}(P_\beta^\rst)}}
\POS(18,-9) \ar@{-} +(-6,-6) \POS(12,-15) *+{\bullet} \endxy \eas
\vspace{.5cm} \bas T_{1,d}(P) = \sgn_{\phi_{re}}( d)(-1)^{n_0+n_1}
\xy\POS(0,9) *+{\bullet} *\cir{} \POS(3,11) *{^{P_j}} \POS(0,9)
\ar@{-} +(-6,-6) \ar@{-} +(6,-6) \POS(6,3) *+{\bullet} \POS(-6,3)
*+{\bullet} \POS(9,5) *{^{P_{j-1}}} \POS(-15,3) *{^{T_{\phi_2,
      d_\epsilon}(P_\epsilon^\rst)}} \POS(6,3) \ar@{-} +(-6,-6)
\ar@{.} +(6,-6) \POS(0,-3) *+{\bullet} \POS(-9,-3) *{^{T_{\phi_2,
      d_\delta}(P_\delta^\rst)}} \POS(12, -3) *+{\bullet} \POS(16,-1)
*{^{\tau P_1}} \POS(12, -3) \ar@{-} +(-6,-6) \ar@{-} +(6,-6)
\POS(6,-9) *+{\bullet} \POS(18,-9) *+{\bullet} \POS(-3,-9)
*{^{T_{\phi_2, d_\gamma}( P_\gamma^\rst)}} \POS(22,-7) *{^{ \tau
    P_0}}\POS(3,-15) *{^{T_{\phi_2, d_\beta}(P_\beta^\rst)}}
\POS(18,-9) \ar@{-} +(-6,-6) \POS(12,-15) *+{\bullet} \endxy\;. \eas
For $m \leq j$, $\Lambda (T_{m-1, d}(P) + T_{m, d}(P)) \in \ker
q_{n_m}.$ The alternating sum, \bas \sum_{m=0}^j(-1)^m\Lambda (T_{m-1,
  d}(P) + T_{m, d}(P)) = \Lambda(T_{\phi_{re}, d}(P) + (-1)^j T_{j,
  d}(P)) \; \eas is in $ \ker q_{\sum_{m= 1}^jn_m}$. Since \bas
(-1)^j\sgn(T_{j,d}) = (-1)^{\sum_{i=0}^jn_i}\sgn_{\phi_2}(d) =
\\ (-1)^{\sum_{i=0}^jn_i}(-1)^{\sum_{\alpha \in d, \alpha
    \textrm{bw}}\chi(\alpha)} = (-1)^n \sgn_{\phi_2}(\tau d)\;, \eas
for all $d \in D(P)$, $d \neq \emptyset$, \bas (-1)^jT_{j, d}(P) =
(-1)^n T_{\phi_2, \tau d}(\tau P)\;. \eas Applying $\Lambda$ on these
trees and generators gives \ba \Lambda \left(\mathop{\sum_{d\in
    D(P)}}_{d\neq \emptyset} T_{\phi_{re}, d}(P) + (-1)^nT_{\phi_2,
  \tau d}(\tau P) \right) \in \ker(q_{n-1})\;.\label{sansempty}\ea
Writing $\Lambda(T_{\phi_{re}, \emptyset}(P)) = [P]$ equation \eqref{sansempty}
gives \bas \Lambda_{\phi_{re}}(P) - [P] +(-1)^n \left( \Lambda_{\phi_2}(\tau
  P) - [\tau P]\right) \in \ker q_{n-1} \;, \eas and \bas
\Lambda_{\phi_{re}}(P) + (-1)^n\Lambda_{\phi_2}(\tau P) \in \ker q_n
\;. \eas
\end{proof}

 Combining theorem \ref{orientsign} with lemma \ref{redif} gives the
 following result.

\begin{thm}
If $P$ is an $R$-deco polygon of weight $n$, \bas q_n(\Lambda_{\phi_2}(P)+
(-1)^n\Lambda_{\phi_2}(\tau P)) = q_n( \sum_{i=2}^{n-1} (-1)^{n-i}
\Lambda_{\phi_2}(P_{_i\alpha}^\rt) \sha
\Lambda_{\phi_2}(\tau(P_{_i\alpha}^\rst))) \; . \eas
 \label{polyorientthm}\end{thm}

\begin{proof} 
Let $P$ be a polygon of weight $n$. Apply $q_n$ to both sides of
\eqref{redifeq} \bas q_n(\Lambda_{\phi_2}(P) - \Lambda_{\phi_{re}}(P)) = q_n(
\sum_{i=2}^{n} \Lambda_{\phi_2}(P_{_i\alpha}^\rt) \sha
\Lambda_{\phi_{re}}(P_{_i\alpha}^\rst)) \;. \eas By equation
\eqref{orientsigneq} replace the terms $\Lambda_{\phi_{re}}(P)$ and
$\Lambda_{\phi_{re}}(P_{{_i}\alpha}^\rst)$ with
$(-1)^{n-1}\Lambda_{\phi_2}(\tau P)$ and
$(-1)^{n-i}\Lambda_{\phi_2}(\tau P_{{_i}\alpha}^\rst)$ to get \bas
q_n(\Lambda_{\phi_2}(P)+ (-1)^n\Lambda_{\phi_2}(\tau P)) = q_n(
\sum_{i=2}^{n} (-1)^{n-i} \Lambda_{\phi_2}(P_{_i\alpha}^\rt) \sha
\Lambda_{\phi_2}(\tau(P_{_i\alpha}^\rst))) \eas as desired.
\end{proof}

This show that $\Lambda_{\phi_2}(P)$ and $\Lambda_{\phi_2}(\tau P)$
can be compared up to a primitive coideal. This relation between
decorated polygons of different orientation is reminiscent of a
relation between iterated integrals on $R \subset \C^\times$. Recall
that for iterated integrals, there is the relation \cite{Gonch05} \bas
I(0; x_1, \ldots x_n; y) I(0; w_1, \ldots w_m; y) = I(0; (x_1, \ldots
x_n) \sha (w_1, \ldots, w_m) ; y) \;.\eas

\begin{lem}
For $r_i \in R$, \ba I(0; r_1, \ldots, r_n;r_{n+1}) + (-1)^nI(0;r_n,
\ldots r_1;r_{n+1}) = \nonumber \\ \sum_{i=2}^n (-1)^{n-i} I(0; r_1,
\ldots, r_{i-1};r_{n+1})I(0;r_n, \ldots
r_i;r_{n+1}) \label{itershuffle}\; .\ea
\end{lem}

\begin{proof} 
This proof is also presented in \cite{Gonch05}. Rewrite the right hand
side of \eqref{itershuffle} as \bas \sum_{i=2}^n (-1)^{n-i} I(0; (r_1,
\ldots, r_{i-1}) \sha (r_n, \ldots r_i);r_{n+1}) \;. \eas For a fixed
$i$ each term in the shuffle product in equation \eqref{itershuffle}
can be broken down into two groups, the terms where $r_{i-1}$ comes
before $r_i$ and the terms where it comes after. By the alternating
signs, the former cancel with a term in the shuffle \bas I(0; (r_1,
\ldots, r_i) \sha (r_n, \ldots r_{i+1}); r_{n+1})\;,\eas and the
latter in the shuffle \bas I(0; (r_1, \ldots, r_{i-2}) \sha (r_n,
\ldots r_i, r_{i-1}); r_{n+1})\;.\eas What remains are the terms (from
$i=2$) $(-1)^nI(0; r_n, \ldots r_1; r_{n+1})$ and (from $i = n$) $
I(0; r_1, \ldots r_n; r_{n+1})$, which are the terms on the left hand
side of \eqref{itershuffle}.

\end{proof}

\begin{rem} The relationship expressed in \eqref{itershuffle} is exact on iterated integrals, while there is a relation only up to a
primitive coideal on the level of polygons, as shown in theorem
\ref{polyorientthm}. This is in contradiction to the conjecture by
Gangl and Brown that relationships between dihedral symmetries of
$R$-deco polygons can be expressed purely in terms of shuffle products
of polygons of lower weights. It also shows that the coalgebra map
between the bar elements associated to polygons and iterated
integrals, \eqref{GGLmap} \bas \Phi : B_{\phi_2} \rightarrow
\mathcal{I}_\bullet(R) \;,\eas is not injective: the coideals $I_n \in \ker
\Phi$ for $n \geq 1$. \end{rem}

\subsection{Order $n$ generator of the dihedral group}

In this subsection, I consider the rotation map, $\sigma$ on $\mcp$
that sends the $R$-deco polygon $P$ to $\bP$. For $P = 12\ldots n$,
$\bP= 2\ldots n 1$ is the polygon rotated clockwise, changing the root
side. When restricted to $V_n(R)$, $\sigma|_{V_n(R)}$ is the order $n$
generator of the dihedral group. In order to examine this rotation, I
work with $\D_{\phi_4}$, which reflect the symmetry of the change, and
relate the corresponding elements of the bar construction to $B_{\phi_2}$.

\subsubsection{Relating $B_{\phi_2}$ to $B_{\phi_4}$}
To understand the action of $\sigma$ on $I(R)$, one wants to calculate
$\Lambda_{\phi_2}(P) - \Lambda_{\phi_2} (\sigma P))$. This is a
difficult calculation, and it is easier to break down into
intermediate steps. I use the results of the last section to relate
the algebras $\Lambda_{\phi_2}(P) - \Lambda_{\phi_4} (P)$. I then
study the action of $\sigma$ on the algebra $B_{\phi_4}$.

\begin{dfn}
 Let $b(P) = \{\text{ backwards arrows of } P\}$.
 \end{dfn}

Recall that $\Tsub{\phi_4}$ and $\Tsub{\phi_3}$ are both dissection
compatible Hopf algebras, and that $\phi_3$ and $\phi_4$ define the
same differential. The difference set between the dual tree generating
sets $\phi_3$ and $\phi_4$ is $\mathcal{S} = \bigcup_{P \;
  R-\textrm{deco}}b(P)$. The algebra $\Tsub{\phi_3}$ is almost
$\phi_4$ compatible.

\begin{thm}
For $P$ an $R$-deco polygon of weight $n$, and $d \in D(P)$, let
$\{P_d^1, \ldots, P_d^{|d|+1}\}$ be the set of polygons decorating the
generator $T_{\phi_4, d}(P)$. Then, for the map $q_n$ as defined in
definition \ref{qndef} \ba q_{n-1}\left(\Lambda_{\phi_2}(P)\right) =
q_{n-1}\left(\Lambda_{\phi_4} (P) + \sum_{ d \subseteq b(P)}
(-1)^{|d|} \sha_{j=1}^{|d|+1}\Lambda_{\phi_4}
(P_d^j)\right)\;.  \label{relate}\ea \label{relatethm}
\end{thm}

\begin{proof}
For all arrows $d \subset b(P)$, $\sgn_{\phi_4}(d) =+1$. By
theorem \ref{edgeswitch} \ba \Lambda_{\phi_3} (P) = \Lambda_{\phi_4}
(P) + \sum_{ d \subseteq b(P)} (-1)^{|d|}
\sha_{j=1}^{|d|+1}\Lambda_{\phi_4} (P_d^j) \;.\label{step1}\ea

For $d \in D(P)$, with $|b(P) \cap d| = j$, let the set $\{P_d^1,
\ldots P_d^{|d|+1}\}$ decorate the vertexes of $T_{\phi_3, d}(P)$ and
$T_{\phi_4, d}(P)$ enumerated $\{P_d^1, \ldots P_d^{|d|+1}\}$ so that
$\{P_d^1,\ldots , P_d^j\}$ decorate the terminal vertexes of the edges
associated to a backwards arrow in $T_{\phi_3, d}(P)$. Then the set
\bas \{\tau P_d^1, \ldots \tau P_d^j, P_d^{j+1}, \ldots
P_d^{|d|+1}\}\eas decorates the vertexes of $T_{\phi_2, d}(P)$. The
generators $T_{\phi_2, d}(P)$ and $T_{\phi_4, d}(P)$ have different
signs, but the same underlying trees, with $P_d^m$ replacing $\tau
P_d^m$ for $m \leq j$. Let $P_d^i \in V_{n_i}(R)$ with
$\sum_{i=1}^{|d|+1} n_i= n$.

Recall that \bas \sgn_{\phi_2}(d) = (-1)^{\sum_{i=1}^j n_i} \quad \textrm{ and }
\quad \sgn_{\phi_3}(d) = (-1)^j \;. \eas Define a series of signed trees
$\{T_{i,d}(P)\}$, $1 \leq i \leq j$  by replacing the polygons
$\{\tau P_d^1, \ldots , \tau P_d^i\}$ in $T_{\phi_2,d}(P)$ with the
set $\{(-1)^{n_1} P_1, \ldots , (-1)^{n_i} P_d^i\}$. In this series,
$T_{\phi_2, d}(P) = T_{0, d}(P)$,  $T_{j, d}(P) = (-1)^j
T_{\phi_3, d}(P)$, and \bas \Lambda(T_{i-1,d}(P) + T_{i,d}(P)) \eas  is  in
$\ker(q_{n_i})$. The alternating sum \bas
\sum_{i=1}^j(-1)^{i-1} \Lambda \left( T_{i-1, d}(P) + T_{i,d}(P)
\right) = \Lambda \left( T_{\phi_2, d}(P) - T_{\phi_3, d}(P) \right) \eas 
is in $\ker (q_{\sum_{i=1}^j n_i})$. Summing over all dissections $d \in
D(P)$ gives \bas \Lambda_{\phi_2}(P) - \Lambda_{\phi_3}(P) \in
\ker q_{n-1} \;. \eas Plugging this into equation \eqref{step1}
gives \bas q_{n-1}(\Lambda_{\phi_2} (P) )=
q_{n-1}\left(\Lambda_{\phi_4} (P))+ \sum_{ d \subseteq b(P)} (-1)^{|d|}
\sha_{j=1}^{|d|+1}\Lambda_{\phi_4} (P^j_d)\right) \;. \eas
\end{proof}

\begin{eg}
Let $P_2=abc$, $P_3=abcd$ be $R$-deco polygons of weight 2 and 3. The
following are the explicit calculations for $P_2$, abd $P_3$.
\bas \begin{array}{cc} & d= \threegonarrowa{a}{b}{c}
  \\ q_1(\Lambda_{\phi_2} (P_2) - \Lambda_{\phi_4} (P_2)) & = - q_1(
  \Lambda_{\phi_4} (ac) \sha \Lambda_{\phi_4} (ba)) \end{array}
\;.\eas Since $\ker q_1 = 0$, this is an exact relation: \bas
\Lambda_{\phi_2} (P_2)) - \Lambda_{\phi_4} (P_2)) = - \Lambda_{\phi_4}
(ac) \sha \Lambda_{\phi_4} (ba) \eas For $P_3$, $q_3\left(
\Lambda_{\phi_2} (P_3) - \Lambda_{\phi_4} (P_3))\right) = $
\bas \begin{array}{ccc} d = \xy \POS(10,4) \ar@{=} +(-10,0)_d \ar@{-}
  +(0,-10)^c \POS(10,-6) \ar@{-} +(-10,0)^b \POS(0,4) \ar@{-}
  +(0,-10)_a \POS(0,4) *+{\bullet} \POS(10,-6) \ar@{->} +(-9,4)
\endxy &  d = \xy
\POS(10,4) \ar@{=} +(-10,0)_d
\ar@{-} +(0,-10)^c
\POS(10,-6) \ar@{-} +(-10,0)^b
\POS(0,4) \ar@{-} +(0,-10)_a
\POS(0,4) *+{\bullet}
\POS(10,4) \ar@{->} +(-9,-4)
\endxy & d = \xy
\POS(10,4) \ar@{=} +(-10,0)_d
\ar@{-} +(0,-10)^c
\POS(10,-6) \ar@{-} +(-10,0)^b
\POS(0,4) \ar@{-} +(0,-10)_a
\POS(0,4) *+{\bullet}
\POS(10,4) \ar@{->} +(-4,-9)
\endxy \\  -  q_3( \Lambda_{\phi_4}
(acd) \sha \Lambda_{\phi_4} (ba) + & \Lambda_{\phi_4} (bca)
\sha \Lambda_{\phi_4} (ad)  +  &   \Lambda_{\phi_4} (abd) \sha
\Lambda_{\phi_4} (cb) )\end{array} \eas \bas \begin{array}{cc} d = \xy
\POS(10,4) \ar@{=} +(-10,0)_d
\ar@{-} +(0,-10)^c
\POS(10,-6) \ar@{-} +(-10,0)^b
\POS(0,4) \ar@{-} +(0,-10)_a
\POS(0,4) *+{\bullet}
\POS(10,-6) \ar@{->} +(-9,4)
\POS(10,4) \ar@{->} +(-9,-4)
\endxy & d = \xy
\POS(10,4) \ar@{=} +(-10,0)_d
\ar@{-} +(0,-10)^c
\POS(10,-6) \ar@{-} +(-10,0)^b
\POS(0,4) \ar@{-} +(0,-10)_a
\POS(0,4) *+{\bullet}
\POS(10,4) \ar@{->} +(-9,-4)
\POS(10,4) \ar@{->} +(-4,-9)
\endxy \\   +q_3 (  \Lambda_{\phi_4} (ad) \sha
\Lambda_{\phi_4} (ba)\sha \Lambda_{\phi_4} (ca)+  &
 \Lambda_{\phi_4} (ad) \sha \Lambda_{\phi_4} (ba) \sha
\Lambda_{\phi_4} (cb) ) \end{array} \;.\eas Since $\ker q_2 \neq
0$, computing the difference explicitly gives \bas
\Lambda_{\phi_2}(P_3) - \Lambda_{\phi_4} (P_3) = - \Lambda_{\phi_4}
(acd) \sha \Lambda_{\phi_4} (ba) - \Lambda_{\phi_4} (bca) \sha
\Lambda_{\phi_4} (ad) - \\ \Lambda_{\phi_4} (abd) \sha
\Lambda_{\phi_4} (cb) + \Lambda_{\phi_4} (ad) \sha \Lambda_{\phi_4}
(ba)\sha \Lambda_{\phi_4} (ca)+ \\ \Lambda_{\phi_4} (ad) \sha
\Lambda_{\phi_4} (ba) \sha \Lambda_{\phi_4} (cb) +[ad|bca + cba]\eas

\end{eg}

The Hopf algebra $B_{\phi_4}$ is contained in $
H^0(B_{\D_3}(\mcp))$ by theorem \ref{Bcocycle}.

\subsubsection{Introducing a new symmetry}

Instead of directly comparing $\Lambda_{\phi_2}(P)$ and
$\Lambda_{\phi_2} (\bP)$ in this subsection, I compare $\Lambda_{\phi_4}
(P)$ and $\Lambda_{\phi_4} (\bP)$. Theorem \ref{relatethm} then relates
these terms to $\Lambda_{\phi_2}(P)$ and $\Lambda_{\phi_2}(\sigma P)$
respectively, as desired.

\begin{dfn}
For $P$ the $R$-deco polygon $12\ldots n$, let $(\bP)$ be the $R$-deco
polygon $2 \ldots n1$ with labels rotated one place in a
clockwise direction. \end{dfn}

\begin{eg} For the weight $3$ polygon $P =1234$ one has

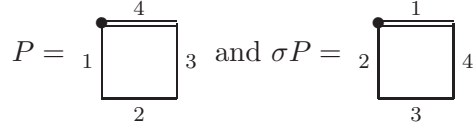
\bas P = \xy
\POS(0,4) *+{\bullet}
\POS(10,4) \ar@{=} +(-10,0)_4
\ar@{-} +(0,-10)^3
\POS(10,-6) \ar@{-} +(-10,0)^2
\POS(0,4) \ar@{-} +(0,-10)_1
\endxy
\textrm{ and }
\bP = 
\xy
\POS(0,4) *+{\bullet}
\POS(10,4) \ar@{=} +(-10,0)_1
\ar@{-} +(0,-10)^4
\POS(10,-6) \ar@{-} +(-10,0)^3
\POS(0,4) \ar@{-} +(0,-10)_2
\endxy \eas
\end{eg}

\begin{eg}
For a weight 1 polygon, $P= ab$, $\bP = ba$, applying the map $\Phi$ from equation \eqref{GGLmap} \begin{multline} \Phi(\Lambda_{\phi_4} (P)
-\Lambda_{\phi_4} (\sigma P) )= \Phi(\Lambda_{\phi_2} (P)
-\Lambda_{\phi_2}(\bP)) =  \\ \Li_1(\frac{a}{b}) - \Li_1(\frac{b}{a}) =
\ln(b) - \ln(a) \;.\label{Li1}\end{multline} The last equality holds up to a power of
$i\pi$.

Direct calculation shows that for $P= abc$, \begin{multline}
  \Lambda_{\phi_4} (P) -\Lambda_{\phi_4} (\sigma P) = \\ P -\bP + [ab
    \sha bc] - [ba \sha cb] + [(ac-ca+ba-ab) | bc] +
         [ba|(ac-ca+cb-bc)] \;. \label{3gonphi4}\end{multline}  Applying theorem \ref{relatethm} gives \bas
\Lambda_{\phi_2}(P) -\Lambda_{\phi_2} (\sigma P) = \\ P -\bP + [ab
  \sha bc] - [ac\sha ba] + [(ac-ca+ba-ab) | bc] +
         [ba|(ac-ca+cb-bc)] \;. \eas
Subsequent direct calculations get increasingly
complex. \label{2&3gons} \end{eg}

To calculate this relation for higher weight polygons, I examine the
action of $\sigma$ on the dissecting arrows of an $R$-deco polygon $P$
of weight $n$. The rotation map $\sigma$ acts on dissecting arrows,
rotating the starting vertex and ending edge one position backwards,
as defined by the orientation of the polygon. Therefore, $\sigma
({_i}\alpha_j) = {_{i-1}}\alpha_{j-1}$ if $i$ or $j \neq 1$, $\sigma
({_1}\alpha_j) = {_{n+1}}\alpha_{j-1}$ and $\sigma ({_i}\alpha_1) =
{_{i-1}}\alpha_{n+1}$.

\begin{eg}
For the $4$-gons $P$ and $\bP$, the dissecting arrows $\alpha$ and
$\sigma\alpha$ are as follows:

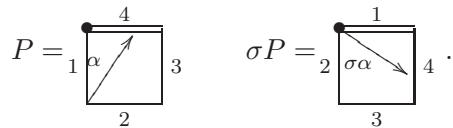
\xy
\hskip 80pt
$P=$
\hskip 5pt
\POS(0,4) *+{\bullet}
\POS(10,4) \ar@{=} +(-10,0)_4
\ar@{-} +(0,-10)^3
\POS(10,-6) \ar@{-} +(-10,0)^2
\POS(0,4) \ar@{-} +(0,-10)_1
\POS(0,-6) \ar^{\alpha}@{->} +(6,9)
\POS(0,-6)
\hskip 60pt
$\bP= $
\hskip 5pt
\POS(0,4) *+{\bullet}
\POS(10,4) \ar@{=} +(-10,0)_1
\ar@{-} +(0,-10)^4
\POS(10,-6) \ar@{-} +(-10,0)^3
\POS(0,4) \ar@{-} +(0,-10)_2
\POS(0,4) \ar_{\sigma\alpha}@{->} +(9,-6)
\POS(10,-6)
\hskip 40pt .
\endxy

\end{eg}

For a general $d \in D(P)$, write $d = \{\alpha_1, \ldots \alpha_k\}$
and $\sigma d = \{\sigma\alpha_1, \ldots, \sigma \alpha_k\}$. To compare
$\Lambda_{\phi_4} (P)$ and $\Lambda_{\phi_4} (\sigma P)$, I start with
dissections of $P$ with one arrow. There are two cases to consider.

\begin{enumerate}
\item The dissecting arrow $\alpha$ starts at the first vertex.  The
  first vertex is in both $P_\alpha^r$ and in $P_\alpha^l$.  The
  associated subpolygons $P$ are related to the subpolygons of $\bP$
  by \bas \sigma(P_\alpha^r) = (\bP)_{\sigma \alpha}^r\quad; \quad
  \sigma(P_\alpha^l) = (\bP)_{\sigma \alpha}^l \eas as illustrated in
  the following example:

\bas P=
\xy
\POS(0,4) *+{\bullet}
\POS(10,4) \ar@{=} +(-10,0)_4
\ar@{-} +(0,-10)^3
\POS(10,-6) \ar@{-} +(-10,0)^2
\POS(0,4) \ar@{-} +(0,-10)_1 \ar_{\alpha}@{->} +(9,-6)
\endxy
\rightarrow P_\alpha^l = 34; \; P_\alpha^r = 123 \eas 
\bas \bP= 
\xy
\POS(0,4) *+{\bullet}
\POS(10,4) \ar@{=} +(-10,0)_1
\ar@{-} +(0,-10)^4 \ar_{\sigma \alpha}@{->} +(-6,-9)
\POS(10,-6) \ar@{-} +(-10,0)^3
\POS(0,4) \ar@{-} +(0,-10)_2
\endxy
\rightarrow (\bP)_{\sigma\alpha}^l = 43; \; (\bP)_{\sigma \alpha}^r = 231 
\;.\eas

\item The dissecting arrow $\alpha$ does not start at the first
  vertex: There are three sub-cases.

\begin{enumerate}
\item The dissecting arrow $\alpha$ ends on the first edge in $P$
  (labeled 1). The first vertex is in $P_\alpha^r$. The dissected
  polygons of $P$ and $\bP$ are \bas \sigma(P_\alpha^r) =
  (\bP)_{\sigma\alpha}^r\quad; \quad P_\alpha^l =
  (\bP)_{\sigma{\alpha}}^l \eas as illustrated in the following example

\bas 
P=
\xy
\POS(0,4) *+{\bullet}
\POS(10,4) \ar@{=} +(-10,0)_4
\ar@{-} +(0,-10)^3
\POS(10,-6) \ar@{-} +(-10,0)^2 \ar_{\alpha}@{->} +(-9,6)
\POS(0,4) \ar@{-} +(0,-10)_1
\endxy
\rightarrow P_\alpha^l = 21; \; P_\alpha^r = 134 \eas \bas \bP= 
\xy
\POS(0,4) *+{\bullet}
\POS(10,4) \ar@{=} +(-10,0)_1
\ar@{-} +(0,-10)^4
\POS(10,-6) \ar@{-} +(-10,0)^3
\POS(0,4) \ar@{-} +(0,-10)_2
\POS(0,-6) \ar_{\sigma \alpha}@{->} +(6,9)
\endxy
\rightarrow (\bP)_{\sigma \alpha}^l = 21; \; (\bP)_{\sigma\alpha}^r = 341 
\quad  .\eas

\item The dissecting arrow $\alpha$ ends on the root edge in $P$
  (labeled n). The first vertex is in $P_\alpha^l$. The dissected
  polygons of $P$ and $\bP$ are \bas P_\alpha^r =
  (\bP)_{\sigma\alpha}^r\quad; \quad \sigma(P_\alpha^l) = \bP_{\sigma
    \alpha}^l \eas as illustrated in the following example
\bas P=
\xy
\POS(0,4) *+{\bullet}
\POS(10,4) \ar@{=} +(-10,0)_4
\ar@{-} +(0,-10)^3
\POS(10,-6) \ar@{-} +(-10,0)^2 \ar_{\alpha}@{->} +(-6,9)
\POS(0,4) \ar@{-} +(0,-10)_1
\endxy
\rightarrow P_\alpha^l = 124; \; P_\alpha^r = 34 \eas
\bas \bP =
\xy
\POS(0,4) *+{\bullet}
\POS(10,4) \ar@{=} +(-10,0)_1
\ar@{-} +(0,-10)^4
\POS(0,-6) \ar^{\sigma \alpha}@{->} +(9,6)
\POS(10,-6) \ar@{-} +(-10,0)^3
\POS(0,4) \ar@{-} +(0,-10)_2
\endxy
\rightarrow \bP_{\sigma \alpha}^l = 241; \; (\bP)_{\sigma \alpha}^r = 34 \quad . \eas

\item The dissecting arrow $\alpha$ ends on neither the first edge or
  root edge in $P$. The root vertex is in $P_\alpha^{l}$ if $\alpha$
  is forward ($P_\alpha^{r}$ if $\alpha$ backward). The dissected
  polygons of $P$ and $\bP$ are \bas \sigma(P_\alpha^{l}) =
  (\bP)_{\sigma\alpha}^{l}\quad; \quad P_\alpha^{r} =
  (\bP)_{\sigma\alpha}^{r} \eas if $\alpha$ is forwards, and \bas
  \sigma(P_\alpha^{r}) = (\bP)_{\sigma\alpha}^{r}\quad; \quad
  P_\alpha^{l} = (\bP)_{\sigma\alpha}^{l} \eas if $\alpha$ is
  backwards. This is illustrated in the following example
\bas P = 
\xy
\POS(0,4) *+{\bullet}
\POS(10,4) \ar@{=} +(-10,0)_4
\ar@{-} +(0,-10)^3
\POS(10,-6) \ar@{-} +(-10,0)^2
\POS(0,4) \ar@{-} +(0,-10)_1
\POS(0,-6) \ar_{\alpha}@{->} +(9,6)
\endxy
\rightarrow P_\alpha^l = 134; \; P_\alpha^r = 23 \eas 
\bas \bP= \xy 
\POS(0,4) *+{\bullet}
\POS(0,4) \ar^{\sigma \alpha}@{->} +(6,-9)
\POS(10,4) \ar@{=} +(-10,0)_1
\ar@{-} +(0,-10)^4
\POS(10,-6) \ar@{-} +(-10,0)^3
\POS(0,4) \ar@{-} +(0,-10)_2
\endxy 
\rightarrow (\bP)_{\sigma \alpha}^l = 341; \; (\bP)_{\sigma \alpha}^r = 23 
\quad  . \eas
\end{enumerate}

\end{enumerate}

This exhaustively categorizes all possible dissecting arrows. I
summarize the results as follows.

\begin{lem}
Let $P$ be an $R$-deco polygon. For arrows of the form
${_1}\alpha_j$, starting at the first vertex, both subpolygons
$P_\alpha^r$ and $P_\alpha^l$ contain the first
vertex. The subpolygons of $P$ associated to a
single dissecting arrow can be classified in the following way:

\bas (\bP)_{\sigma \alpha}^r =
\begin{cases} \sigma(P_\alpha^r) &
\text{if $P_\alpha^r$ contains the first vertex of $P$} \\ P_\alpha^r &
\text{otherwise }
\end{cases} \eas
The same is true if $r$ is replaced with $l$.  \label{sigmasyms}
\end{lem}

To calculate the action of the operator $\sigma$ on the Hopf algebra
$B_{\phi_4}$, I compare terms in the Hopf algebra
$B_{\phi_4}$ to two new algebras $B_{\phi_{fv}}$ and
$B_{\phi_{\ofv}}$, defined by new generating sets $\phi_{fv}$ and
$\phi_{\ofv}$ that exploit the symmetries defined in lemma
\ref{sigmasyms}.

\begin{dfn}
Let $fv(P)$ be set of arrows that start at the \emph{f}irst
\emph{v}ertex of a polygon $P$.  If $P \in V_n(R)$, write $fv(P)=
\{\alpha_2, \dots, \alpha_{n+1}\}$ where $\alpha_i$ ends at the
$i^{th}$ side. Define $\ofv(\sigma P) = \{\sigma \alpha_2, \dots, \sigma
\alpha_{n+1}\}$ to be the set of arrows that start at the $(n+1)^{th}$ (last)
vertex of $\bP$. 
 \end{dfn}

Define the generating sets $\phi_{fv}$ and $\phi_{\ofv}$ to define the
same differential as $\phi_4$. Their difference set from $\phi_4$ is
$\cup_P fv(P)$ and $\cup_P \ofv(P)$ respectively. Let
$\Tsub{\phi_{fv}}$ and $\Tsub{\phi_{\ofv}}$ be the two almost $\phi_4$
compatible algebras defined by the dual tree generating sets
$\phi_{fv}$ and $\phi_{\ofv}$.

Let $P$ be an $R$-deco polygon of weight $n$. Consider non-trivial
dissections $d \in D(P)$, with $d \cap fv(P) = \{\alpha_{i_1} \ldots
\alpha_{i_j}\}$, with $i_1 < i_2 \ldots < i_j$. Let $\{P_d^0, \ldots
P_d^{|d|}\}$ be the set of polygons decorating $T_{\phi_{fv}, d}(P)$
and $T_{\phi_4, d}(P)$, with each $P_d^i \in V_{n_i}(R)$ with
$\sum_{i=0}^{|d|} n_i = n$. For $m\leq j$, the polygons $P_d^{m-1}$
and $P_d^m$ are adjacent, connected by $\alpha_{i_m}$, and $P_d^{m-1}
\prec P_d^{m}$ in $T_{\phi_{fv},d}(P)$ and $P_d^{m} \prec P_d^{m-1}$
in $T_{\phi_4,d}(P)$. If $d \cap fv(P) = \emptyset$, $P_d^0$ is the
subpolygon containing the first vertex of $P$. The set $\{\sigma
P_d^0, \ldots \sigma P_d^j, P_d^{j+1}, \ldots P_d^{|d|}\}$ decorates
the generator $T_{\phi_4, \sigma d}(\bP)$ with $\bP_d^{m} \prec
\bP_d^{m-1}$, and $T_{\phi_{\ofv}, \sigma d}(\bP)$ with $\bP_d^{m-1}
\prec \bP_d^{m}$. The subpolygons $\bP_d^{m-1}$ and $\bP_d^{m}$ are
adjacent, connected by $\sigma \alpha_{i_m}$. If $d \cap fv(P) =
\emptyset$, then $\sigma P_0$ is the subpolygon containing the last
vertex of $\bP$. Recall that \bas \sgn (T_{fv, d}( P)) = \sgn
(T_{\ofv, \sigma d}(\bP)) = (-1)^{|d \cap fv(P)|} \eas while $\sgn
(T_{\phi_4, d}( P)) = 1$.  If $\beta \not \in fv(P)$, write $d_\beta =
d \cap D(P_\beta^\rst)$. The polygon $P_\beta^\rst$ does not contain
the first vertex of $P$, therefore $P_\beta^\rst = (\sigma P)_{\sigma
  \beta}^\rst$. By lemma \ref{sigmasyms}, the pruned subtrees
corresponding to the admissible cut $c = \beta$ in $T_{\phi_4, d}(P)$,
and $T_{\phi_{fv}, d}(P)$ are the same: $T_{\phi_4,
  d_\beta}(P_\beta^\rst)$. Similarly for $\sigma c = \sigma \beta$ in
$T_{\phi_{\ofv}, \sigma d}(\sigma P)$ and $T_{\phi_{4}, \sigma
  d}(\sigma P)$ the pruned trees are both $T_{\phi_4, \sigma
  d_\beta}(\sigma P_\beta^\rst)$. I summarize this in the following
diagrams. Notice that $P_d^0$ and $P_d^j$ need not be a root vertex of
these generators. These trees are drawn without a root vertex
specified.  \bas T_{\phi_{4}, d}(P) = 
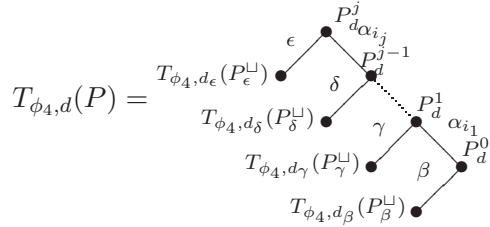
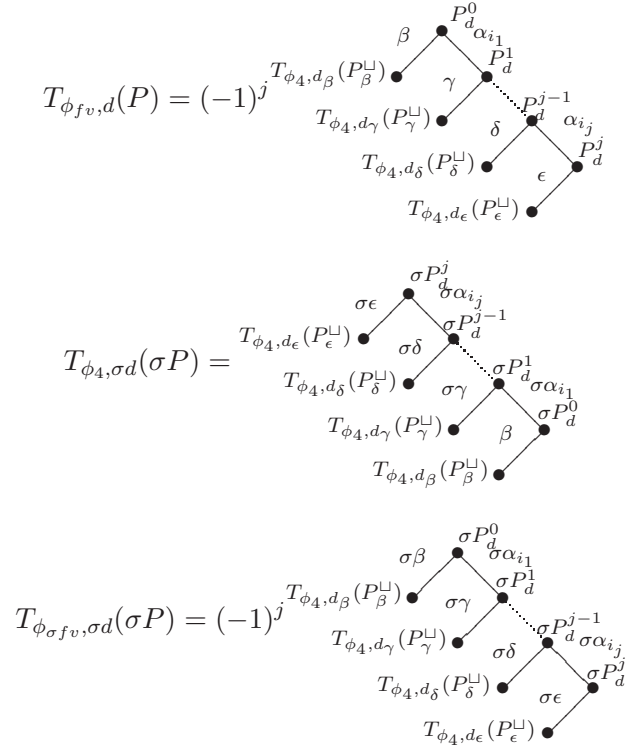
\xy\POS(0,9) *+{\bullet}
\POS(3,11) *{^{P_d^j}} \POS(0,9) \ar@{-}_{\epsilon} +(-6,-6)
\ar@{-}^{\alpha_{i_j}} +(6,-6) \POS(6,3) *+{\bullet} \POS(-6,3)
*+{\bullet} \POS(8,5) *{^{P_d^{j-1}}} \POS(-15,3) *{^{T_{\phi_4,
      d_\epsilon}(P_\epsilon^\rst)}} \POS(6,3) \ar@{-}_{\delta}
+(-6,-6) \ar@{.} +(6,-6) \POS(0,-3) *+{\bullet} \POS(-9,-3)
*{^{T_{\phi_4, d_\delta}(P_\delta^\rst)}} \POS(12, -3) *+{\bullet}
\POS(14,-1) *{^{P_d^1}} \POS(12, -3) \ar@{-}_{\gamma} +(-6,-6)
\ar@{-}^{\alpha_{i_1}} +(6,-6) \POS(6,-9) *+{\bullet} \POS(18,-9)
*+{\bullet} \POS(-3,-9) *{^{T_{\phi_4, d_\gamma}(P_\gamma^\rst)}}
\POS(20,-7) *{^{P_d^0 }} \POS(3,-15) *{^{T_{\phi_4,
      d_\beta}(P_\beta^\rst)}} \POS(18,-9) \ar@{-}_{\beta} +(-6,-6)
\POS(12,-15) *+{\bullet} \endxy \eas
\vspace{.5cm} 
\bas T_{\phi_{fv}, d}(P) = (-1)^j \xy\POS(0,9)
*+{\bullet} \POS(3,11) *{^{P_d^0}} \POS(0,9) \ar@{-}_{\beta}
+(-6,-6) \ar@{-}^{\alpha_{i_1}} +(6,-6) \POS(6,3) *+{\bullet}
\POS(-6,3) *+{\bullet} \POS(8,5) *{^{P_d^1}} \POS(-15,3) *{^{T_{\phi_4,
      d_\beta}(P_\beta^\rst)}} \POS(6,3) \ar@{-}_{\gamma} +(-6,-6)
\ar@{.} +(6,-6) \POS(0,-3) *+{\bullet} \POS(-9,-3) *{^{T_{\phi_4,
      d_\gamma}(P_\gamma^\rst)}} \POS(12, -3) *+{\bullet} \POS(14,-1)
*{^{P_d^{j-1}}} \POS(12, -3) \ar@{-}_{\delta} +(-6,-6)
\ar@{-}^{\alpha_{i_j}} +(6,-6) \POS(6,-9) *+{\bullet} \POS(18,-9)
*+{\bullet} \POS(-3,-9) *{^{T_{\phi_4, d_\delta}(P_\delta^\rst)}}
\POS(20,-7) *{^{P_d^j}} \POS(3,-15) *{^{T_{\phi_4,
      d_\epsilon}(P_\epsilon^\rst)}} \POS(18,-9) \ar@{-}_{\epsilon}
+(-6,-6) \POS(12,-15) *+{\bullet} \endxy \eas
\bas T_{\phi_{4}, \sigma d}(\bP) = 
\xy\POS(0,9) *+{\bullet} \POS(3,11.5)
*{^{\bP_d^j}} \POS(0,9) \ar@{-}_{\sigma \epsilon} +(-6,-6)
\ar@{-}^{\sigma \alpha_{i_j}} +(6,-6) \POS(6,3) *+{\bullet} \POS(-6,3)
*+{\bullet} \POS(9,5) *{^{\bP_d^{j-1}}} \POS(-15,3) *{^{T_{\phi_4,
      d_\epsilon}(P_\epsilon^\rst)}} \POS(6,3) \ar@{-}_{\sigma \delta}
+(-6,-6) \ar@{.} +(6,-6) \POS(0,-3) *+{\bullet} \POS(-9,-3)
*{^{T_{\phi_4, d_\delta}(P_\delta^\rst)}} \POS(12, -3) *+{\bullet}
\POS(14,-1) *{^{\bP_d^1}} \POS(12, -3) \ar@{-}_{\sigma \gamma} +(-6,-6)
\ar@{-}^{\sigma \alpha_{i_1}} +(6,-6) \POS(6,-9) *+{\bullet} \POS(18,-9)
*+{\bullet} \POS(-3,-9) *{^{T_{\phi_4, d_\gamma}(P_\gamma^\rst)}}
\POS(20,-7) *{^{\bP_d^0 }} \POS(3,-15) *{^{T_{\phi_4,
      d_\beta}(P_\beta^\rst)}} \POS(18,-9) \ar@{-}_{\beta} +(-6,-6)
\POS(12,-15) *+{\bullet} \endxy \eas
\bas T_{\phi_{\ofv}, \sigma d}(\bP) = (-1)^j\xy\POS(0,9)
*+{\bullet} \POS(3,11) *{^{\bP_d^0}} \POS(0,9) \ar@{-}_{\sigma \beta}
+(-6,-6) \ar@{-}^{\sigma \alpha_{i_1}} +(6,-6) \POS(6,3) *+{\bullet}
\POS(-6,3) *+{\bullet} \POS(8,5) *{^{\bP_d^1}} \POS(-15,3) *{^{T_{\phi_4,
      d_\beta}(P_\beta^\rst)}} \POS(6,3) \ar@{-}_{\sigma \gamma} +(-6,-6)
\ar@{.} +(6,-6) \POS(0,-3) *+{\bullet} \POS(-9,-3) *{^{T_{\phi_4,
      d_\gamma}(P_\gamma^\rst)}} \POS(12, -3) *+{\bullet} \POS(15,-1)
*{^{\bP_d^{j-1}}} \POS(12, -3) \ar@{-}_{\sigma \delta} +(-6,-6)
\ar@{-}^{\sigma \alpha_{i_j}} +(6,-6) \POS(6,-9) *+{\bullet} \POS(18,-9)
*+{\bullet} \POS(-3,-9) *{^{T_{\phi_4, d_\delta}(P_\delta^\rst)}}
\POS(20,-7) *{^{\bP_d^j}} \POS(3,-15) *{^{T_{\phi_4,
      d_\epsilon}(P_\epsilon^\rst)}} \POS(18,-9) \ar@{-}_{\sigma \epsilon}
+(-6,-6) \POS(12,-15) *+{\bullet} \endxy \eas
\vspace{.5cm} 

These diagrams also illustrate the following lemma. 

\begin{lem} 
Consider $c \in D(P)$ such that $T_{\phi_{fv}, c}(P)$ (and $T_{\phi_{\ofv}, \sigma
  c}(\bP)$) only have root and leaf vertexes. Let $\{P_c^0, \ldots
P_c^{|c|}\}$ be the set of labels of $T_{\phi_{fv}, c}(P)$, with $\{P_c^0,
\ldots P_c^m\}$ labeling the root vertexes. Then \bas \Delta_c
\Lambda_{\phi_{fv}}(P) = \sha_{j=0}^m \Lambda_{*}(P_c^j) \otimes
\sha_{i=m+1}^{|c|} \Lambda_{*}P_c^i \eas with \bas * = \begin{cases}
  \phi_{fv} & \textrm{ if } (P_c^i) \textrm{ contains first vertex of } P
  \\ \phi_4 & \textrm{ else. } \end{cases} \;. \eas
Similarly, \bas \Delta_{\sigma c}
\Lambda_{\ofv}(\bP) = \sha_{j=0}^m \Lambda_{*}((\bP)_{\sigma c}^j) \otimes
\sha_{i=m+1}^{|c|} \Lambda_{*}((\bP)_{\sigma c}^i) \eas with \bas * = \begin{cases}
  \phi_{\ofv} & \textrm{ if } P_c^i \textrm{ contains first vertex of } P
  \\ \phi_4 & \textrm{ else. } \end{cases} \eas
\label{fv-ofvcuts}\end{lem}

\begin{proof}
For a fixed $c$ as above and any dissection $d \in D(P)$, such that $c
\subset d$, write $d = c \cup_{i=1}^{|c|} d_i$ with $d_i \in
D(P_c^i)$. By definition of $\phi_{fv}$, $ d_i \cap fv(P) = \emptyset$
if and only if $P_c^i$ does not inherit the first vertex from $P$. That is, if
and only if \bas T_{\phi_{fv}, d_i}(P_c^i) = T_{\phi_4, d_i}(P_c^i) \;.\eas
Varying over all dissections $d$ containing $c$ shows that \bas
\Lambda_{\phi_{fv}}(P_c^i) = \Lambda_{\phi_4}(P_c^i) \; \eas if and only
if $P_c^i$ does not inherit the first vertex of $P$.

\end{proof}

By abuse of notation, call the dissections $c \in D(P)$ such that
$T_{\phi_{fv}, c}(P)$ only has leaf and root vertexes the admissible
dissections of $P$ in $\phi_{fv}$. 

\begin{cor}
Using the definitions above, write \bas \Delta \Lambda_{\phi_{fv}}(P) =
\sum_{c \textrm{ admis. dis.}} \Delta_c \Lambda_{\phi_{fv}}(P)\; \eas and \bas
\Delta \Lambda_{\phi_{\ofv}}(\bP) = \sum_{c \textrm{ admis. dis.}} \Delta_c
\Lambda_{\phi_{\ofv}}(\bP)\;. \eas
\label{fv-ofvcoprod} \end{cor}

Instead of calculating $\Lambda_{\phi_4} (P) - \Lambda_{\phi_4}
(\sigma P)$, I calculate the expression \begin{multline} \left(\Lambda_{\phi_4}
(P))-\Lambda_{\phi_{fv}} (P))\right) - \left(\Lambda_{\phi_4} (\sigma
P)) - \Lambda_{\phi_{\ofv}} (\sigma P))\right) +  \\
\left(\Lambda_{\phi_{fv}} (P)) - \Lambda_{\phi_{\ofv}} (\sigma
P)\right) \label{newrel} \;. \end{multline} This is done in steps. The first two
terms of \eqref{newrel} are derived from theorem \ref{edgeswitch};
the third term can be calculated up to a primitive coideal.

\begin{lem}
Let $P= 1\ldots n+1$ be an $R$-deco polygon. For $d \subset fv(P)$,
let $\{P_d^0, \ldots P_d^j\}$ be the polygons decorating $T_{\phi_4, d}(P)$. There are two expressions for \bas
\left(\Lambda_{\phi_4} (P)-\Lambda_{\phi_{fv}} (P)\right) -
\left(\Lambda_{\phi_4} (\sigma P) - \Lambda_{\phi_{\ofv}} (\sigma
P)\right) \eas
\begin{enumerate} 
\item $ = \sum_{i=1}^{n}\Lambda_{\phi_{fv}} (P_{\alpha_i}^l)\sha
  \Lambda_{\phi_4} (P_{\alpha_i}^r) -
  \sum_{i=2}^{n}\Lambda_{\phi_4} ((\sigma P)_{\sigma \alpha_i}^l)\sha
  \Lambda_{\phi_{\ofv}} ((\sigma P)_{\sigma\alpha_i}^r )$
   \label{mixed} 
\item $ =-\sum_{d \in fv(P)} (-1)^{|d|} \sha_{j=0}^{|d|}
  \Lambda_{\phi_4}(P_d^j) - \sha_{j=0}^{|d|}
  \Lambda_{\phi_4}(\sigma (P_d^j))\;. $
     \label{onlyphi4}
\end{enumerate}
\label{phi4-fv}
\end{lem}

\begin{proof}
To see expression \eqref{mixed}, notice that for $\alpha_i, \alpha_j
\in fv(P)$, with $i <j$, $\alpha_i$ dissects the subpolygon
$P_{\alpha_j}^r$, and $\alpha_j$ dissects the subpolygon
$P_{\alpha_i}^l$. For $\sigma\alpha_i, \sigma \alpha_j \in \ofv(\bP)$,
with $i <j$, $\sigma \alpha_i$ dissects the subpolygon $(\bP)_{\sigma
  \alpha_j}^r$, and $\sigma \alpha_j$ dissects the subpolygon
$(\bP)_{\sigma \alpha_i}^l$. Therefore the generators $T_{\phi_{fv},
  fv(P)}(P)$ and $T_{\phi_{\ofv}, \ofv(P)}(\bP)$ are linear, and the
result follows form corollary \ref{lineartrees} and
\ref{lineartreesleaf}.

To see expression \eqref{onlyphi4}, theorem \ref{edgeswitch} gives
\bas \left(\Lambda_{\phi_4} (P)-\Lambda_{\phi_{fv}} (P)\right) =
\sum_{\begin{subarray}{c} d \subset fv(P) \\ d \neq
    \emptyset\end{subarray}} (-1)^{|d|+1} \sha_{j=0}^{|d|}
\Lambda_{\phi_4}(P_d^j) \;,  \eas and \bas \left(\Lambda_{\phi_4} (\bP)-\Lambda_{\phi_{\ofv}} (\bP)\right) =
\sum_{\begin{subarray}{c} d \subset fv(P) \\ d \neq
    \emptyset\end{subarray}} (-1)^{|d|+1} \sha_{j=0}^{|d|}
\Lambda_{\phi_4}(\sigma (P_d^j)) \;.  \eas
\end{proof}

In the following example, I compute $\left(\Lambda_{\phi_4}
(P)-\Lambda_{\phi_{fv}} (P)\right) $ and $ \left(\Lambda_{\phi_4}
(\sigma P) - \Lambda_{\phi_{\ofv}} (\sigma P)\right)$ for a polygon of
weight 3.

\begin{eg} 
Consider $P = 1234$. By expression \eqref{onlyphi4} of lemma
\ref{phi4-fv}, one has \bas \Lambda_{\phi_4}
(P)-\Lambda_{\phi_{fv}}(P) = \Lambda_{\phi_4} (12) \sha
\Lambda_{\phi_4} (234) + \\ \Lambda_{\phi_4} (123) \sha
\Lambda_{\phi_4} (34) - \Lambda_{\phi_4} (12) \sha \Lambda_{\phi_4}
(23) \sha \Lambda_{\phi_4} (34) \;.\eas Applying theorem \ref{relatethm}
gives \bas q_2(\Lambda_{\phi_4} (P)-\Lambda_{\phi_{fv}}(P)) =
q_2(\Lambda_{\phi_2} (12) \sha (\Lambda_{\phi_2} (234) +
\Lambda_{\phi_2}(24)\sha \Lambda_{\phi_2}(32)) +\\ (\Lambda_{\phi_2} (123) +
\Lambda_{\phi_2}(13) \sha \Lambda_{\phi_2}(21))\sha \Lambda_{\phi_2}
(34) - \\ \Lambda_{\phi_2} (12) \sha \Lambda_{\phi_2} (23) \sha
\Lambda_{\phi_2} (34))\;. \eas The expression involving $\bP = 2341$
is \bas \Lambda_{\phi_4} (\sigma P)-\Lambda_{\phi_{\ofv}}(\bP) =
\Lambda_{\phi_4} (21) \sha \Lambda_{\phi_4} (342) +
\\ \Lambda_{\phi_4} (231) \sha \Lambda_{\phi_4} (43) -
\Lambda_{\phi_4} (43)\sha \Lambda_{\phi_4} (32) \sha \Lambda_{\phi_4}
(21) \; .\eas By expression \ref{onlyphi4} of lemma \ref{phi4-fv}, one
has \bas q_2(\Lambda_{\phi_4} (\sigma P)-\Lambda_{\phi_{fv}}(\bP)) =
q_2(\Lambda_{\phi_2} (21) \sha (\Lambda_{\phi_2} (342) +
\Lambda_{\phi_2} (32) \sha \Lambda_{\phi_2} (43)) +
\\ (\Lambda_{\phi_2} (231) + \Lambda_{\phi_2} (21) \sha
\Lambda_{\phi_2}(32)) \sha \Lambda_{\phi_2} (43) \\ - \Lambda_{\phi_2}
(43)\sha \Lambda_{\phi_2} (32) \sha \Lambda_{\phi_2} (21)) \; .\eas

Adding the two expressions gives \bas q_2(\Lambda_{\phi_4}
(P)-\Lambda_{\phi_{fv}}(P) -\Lambda_{\phi_4} (\sigma
P)+\Lambda_{\phi_{\ofv}}(\bP)) = \\ q_2( \Lambda_{\phi_2} (12) \sha
\Lambda_{\phi_2} (234) + \Lambda_{\phi_2} (123)\sha \Lambda_{\phi_2}
(34) \\ - \Lambda_{\phi_2} (231) \sha \Lambda_{\phi_2} (43) -
\Lambda_{\phi_2} (21) \sha \Lambda_{\phi_2} (342) \\ +
\Lambda_{\phi_2} (12) \sha \Lambda_{\phi_2}(24)\sha
\Lambda_{\phi_2}(32) + \Lambda_{\phi_2}(13) \sha
\Lambda_{\phi_2}(21))\sha \Lambda_{\phi_2} (34) \\- \Lambda_{\phi_2}
(12) \sha \Lambda_{\phi_2} (23) \sha \Lambda_{\phi_2} (34)) -
\Lambda_{\phi_2} (21)\sha \Lambda_{\phi_2} (32) \sha \Lambda_{\phi_2}
(43) \;. \eas

\end{eg}

The rest of this section calculates the third term in \eqref{newrel}.

\begin{dfn}
Let $J_n\subset \mathcal{P}_\bullet^{(1)}(R)$ be the linear subspace
generated by $\{P - \sigma P | P \textrm{ polygon of weight }
n\}$. These are primitive coideals in $B_{\phi_4}$.
\end{dfn}

\begin{thm} 
Define the quotient map \bas r_n: \bar{T}(V(R)) \rightarrow
\bar{T}(V(R))/ (\sum_{k= 1}^n J_k) \;. \eas Then for $P$ a
polygon of weight $n$, \bas r_n(\left(\Lambda_{\phi_{fv}} (P)) -
\Lambda_{\phi_{\ofv}} (\sigma P)\right)) = 0
\;.\eas \label{wholeideal}\end{thm}

The coideals $J_n$ are problematic in this context since $J_1$ is not
in the kernel of $\Phi$ (equation \eqref{GGLmap}), \ba \Phi([ab] -
[ba]) = \Li(\frac{a}{b}) - \Li(\frac{b}{a}) = \ln(b) -
\ln(a)\;,\label{phiimage}\ea as shown in example
\ref{2&3gons}. Instead, I work with a modified quotient map.

\begin{dfn}
Define an ideal $\tilde J_1 = \{ac-ca + cb-bc - ab+ba\}$ to reflect
the image of $\Phi$ restricted to 2-gons. Define the quotient map
\bas \tilde{r}_n: \bar{T}(V(R)) \rightarrow \bar{T}(V(R))/ (\tilde J_1
+ \sum_{k= 2}^n J_k) \;. \eas
\end{dfn}

I need the following definition to calculate $\tilde{r}_n(\Lambda_{\phi_{fv}}(P) - \Lambda_{\phi_{\ofv}}(P))$.

\begin{dfn} 
For $P= 1\ldots {n+1}$ be polygon of weight $n$, define two
polygons $A_P= 2\ldots {n+1}$ and $B_P=2\ldots n1$ of weight $n-1$. The
polygon $A_P$ can be drawn as a subpolygon of $P$, $A_P =
P_{\alpha_2}^l$. Define \bas A_{P,\alpha_i}^l = P_{\alpha_{i+1}}^l
\text{, and } A_{P, \alpha_i}^r = P_{_2\alpha_{i+1}}^r\;.\eas Similarly,
$B_P$ can be drawn as a subpolygon of $\sigma P$, $B_P = (\bP)_{\sigma
  \alpha_n}^r$. Define \bas B_{P,\sigma \alpha_i}^l = \sigma
(P_{{_{n+1}}\alpha_i}^l) = (\sigma P)_{\sigma ({_{n+1}}\alpha_i)}^l
\text{, and } B_{P, \sigma \alpha_i}^r = \sigma (P_{\alpha_i}^r) =
(\sigma P)_{\sigma \alpha_i}^r \; .\eas 
\label{ABdef}\end{dfn}

Recall that in the above notation, when I write $P_\beta^r$, I mean $\beta
\in D(P)$. For example, when I write $A_{P, \alpha_i}^l$,
$\alpha_i \in D(A_P)$ while in $P_{\alpha_{i+1}}^l$, $\alpha_{i+1} \in
D(P)$. 

\begin{eg} 
The following diagram shows $A_{P,\alpha_4}^l$, $A_{P,\alpha_4}^r$,
$B_{\bP,\sigma \alpha_4}^l$ and $B_{\bP, \sigma \alpha_4}^r$ \bas 
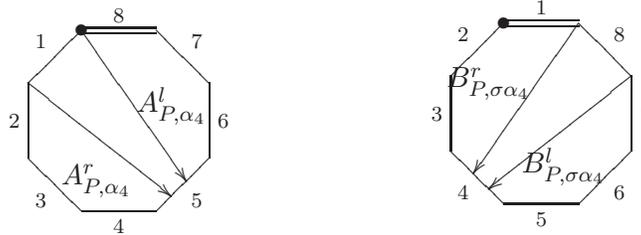
\xy (0,0)
\POS(0,0) \ar@{-} +(10,0)_4 \POS(10,0) \ar@{-} +(7,7)_5 \POS(17,7)
\ar@{-} +(0,10)_6 \POS(17,17) \ar@{-} +(-7,7)_7 \POS(10,24) \ar@{=}
+(-10,0)_8 \POS(0,24) \ar@{-} +(-7,-7)_1 \POS(0,24)*+{\bullet}
\POS(-7,17) \ar@{-} +(0,-10)_2 \POS(-7,7) \ar@{-} +(7,-7)_3
\POS(0,24) \ar@{->} +(14,-20)
\POS(-7,17)  \ar@{->} +(19,-15)
\POS(12,14) *{A^l_{P,\alpha_4}}
\POS(2,4) *{A^r_{P,\alpha_4}}
\endxy
\hskip 70pt
\xy
\POS(0,0)  \ar@{-} +(10,0)_5
\POS(10,0) \ar@{-} +(7,7)_6
\POS(17,7) \ar@{-} +(0,10)_7
\POS(17,17) \ar@{-} +(-7,7)_8
\POS(10,24) \ar@{=} +(-10,0)_1
\POS(0,24) \ar@{-} +(-7,-7)_2 \POS(0,24)*+{\bullet}
\POS(-7,17) \ar@{-} +(0,-10)_3
\POS(-7,7) \ar@{-} +(7,-7)_4
\POS(10,24) \ar@{->} +(-14,-20)
\POS(17,17) \ar@{->} +(-19,-15) 
\POS(8,5) *{B^l_{P,\sigma \alpha_4}}
\POS(-2,16) *{B^r_{P,\sigma \alpha_4}} \endxy \eas  \label{ABpict}
\end{eg}
 
\begin{dfn}
 Write $\{\alpha, d, P\} = \{v_{r,\alpha}, v_{l,\alpha}\}$ to indicate
 the set of subpolygons to the right
 and left of the dissecting arrow $\alpha \in d \in D(P)$. \end{dfn}

\begin{eg} 
For example, $\{\beta,d, P\}$, for the polygons $P=123456$ and
dissection $d= \{\alpha, \beta, \gamma\}$ as drawn below, \bas 
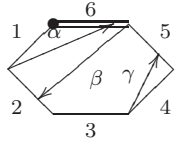
\xy
\POS(0,4) *+{\bullet} \POS(10,4) \ar@{=} +(-10,0)_6 \ar@{-} +(6,-6)^5
\POS(16,-2)\ar@{-} +(-6,-6)^4 \POS(0,-8)\ar@{-} +(10,0)_3 \ar@{-}
+(-6,6)^2 \POS(0,4)\ar@{-} +(-6,-6)_1 \POS(-6,-2) \ar^{\alpha}@{->}
+(14,6) \POS(10,4) \ar^{\beta}@{->} +(-12,-10) \POS(10,-8)
\ar^{\gamma}@{->} +(4,8) \endxy \eas is given by $\{\beta, P\} =
\{v_{r,\beta}=26, v_{l,\beta}=352\}$
\end{eg}

I now construct a generalization of the insertion operators
$\star_{\prec v}$ and $\star_{\succ v}$, defined in definition
\ref{grafting} to apply to words of the form $\Lambda_\phi(P)$.

\begin{dfn}
Consider $u \in  V(R)$, and $P$ and
$R$-deco polygon. Define \bas u \star_{\prec \{\alpha, P\}}
\Lambda_\phi(P) & = \sum_{\begin{subarray}{c}d \in D(P)\\ d = d^r \cup
    d^l \cup \alpha \end{subarray}} u \star_{ \prec \{\alpha,d, P\}}
\Lambda_\phi(T_{\phi,d^r}(P_\alpha^r)) \sha
\Lambda_\phi(T_{\phi,d^l}(P_\alpha^l)) \eas where $d^r \in
D(P_\alpha^r)$ and $d^l\in D(P_\alpha^l)$. Similarly \bas u \star_{
  \succ \{\alpha, P\}} \Lambda_\phi(P) & = \sum_{\begin{subarray}{c}d
    \in D(P)\\ d = d^r \cup d^l \cup \alpha \end{subarray}} u \star_{
  \succ \{\alpha, d,P\}} \Lambda_\phi(T_{\phi,d^r}(P_\alpha^r)) \sha
\Lambda_\phi(T_{\phi,d^l}(P_\alpha^l)) \eas
\label{LambdaB}\end{dfn}

The next theorem gives an expression for $\Lambda_{\phi_{fv}}(P) -
\Lambda_{\phi_{\ofv}}(\bP)$.

\begin{thm}
Let $P$, $A_P$, $B_P$ and associated subpolygons be as above. Define the
quotient map \bas \tilde r_n: \Lambda(\Tsub{\phi_4}) \rightarrow
\Lambda(\Tsub{\phi_4})/ (\tilde J_1 + \sum_{k= 2}^n J_k) \;, \eas For $P$ an
$R$-deco polygon of weight $n \geq 2$, \begin{multline} \tilde
  r_n\left(\Lambda_{\phi_{fv}} (P) - \Lambda_{\phi_{\ofv}} (\sigma
  P))\right) = \\ \sum_{i=2}^{n} \tilde r_n \left(
  (i(n+1)-(n+1)i)\star_{\prec \{\alpha_{i-1}, A_P\} } \Lambda_{\phi_4}
  (A_{P,\alpha_{i-1}}^r) \sha \Lambda_{\phi_{fv}} (A_{P,\alpha_{i-1}}^l)
    \right. \\
- \left. (i1-1i) \star_{\succ \{\sigma
    \alpha_i, B_P\}} \Lambda_{\phi_4} (B_{P,\sigma \alpha_i}^l)\sha
  \Lambda_{\phi_{\ofv}} (B_{P, \sigma \alpha_i}^r) \right) \\ 
  \;.  \label{fv-ofv} \end{multline}
\label{graftedpolygons}\end{thm}

\begin{rem}
Notice that \begin{multline} i(n+1)-(n+1)i\star_{\prec
    \{\alpha_{i-1}, A_P\} } (\Lambda_{\phi_4} (A_{P,\alpha_{i-1}}^r)
  \sha \Lambda_{\phi_{fv}} (A_{P,\alpha_{i-1}}^l)) =
  \\ i(n+1)\star_{\prec \{\alpha_{i-1}, A_P\} }
  (\Lambda_{\phi_4} (A_{P,\alpha_{i-1}}^r) \sha \Lambda_{\phi_{fv}}
  (A_{P,\alpha_{i-1}}^l)) \\ - (n+1)i\star_{\prec
    \{\alpha_{i-1}, A_P\} } (\Lambda_{\phi_4} (A_{P,\alpha_{i-1}}^r)
  \sha \Lambda_{\phi_{fv}} (A_{P,\alpha_{i-1}}^l)) \;.
\end{multline}
 \end{rem}

Before proving this theorem, I use it to prove theorem
\ref{wholeideal}.

\begin{proof}[Proof of theorem \ref{wholeideal}]
Since every term in theorem \ref{fv-ofv} calculation involves a term
in $J_1$, specifically, the newly inserted difference, \bas r_1\circ \tilde{r}_n(\Lambda_{fv}(P) -
\Lambda_{\ofv}(\bP)) = 0 \; \eas for any $R$-deco polygon $P$.
\end{proof}

I have now computed all the necessary terms for $\Lambda_{\phi_4}(P) -
\Lambda_{\phi_4} (\bP)$.  Let $P$ be a polygon of weight
$n$. Combining line \ref{onlyphi4} of lemma \ref{phi4-fv} with theorem
\ref{graftedpolygons} gives \bas \tilde r_n(\Lambda_{\phi_4}(P) -
\Lambda_{\phi_4}(\bP) )=-\sum_{d \in fv(P)} (-1)^{|d|}
\sha_{j=0}^{|d|}( \Lambda_{\phi_4}(P_d^j) - \sha_{j=0}^{|d|}
\Lambda_{\phi_4}(\sigma (P_d^j))) \\ + \tilde r_n \sum_{i=2}^{n}  \left(
i(n+1)-(n+1)i\star_{\prec
  \{\alpha_{i-1}, A_P\} } \Lambda_{\phi_4} (A_{P,\alpha_{i-1}}^r) \sha
\Lambda_{\phi_{fv}} (A_{P,\alpha_{i-1}}^l))  \right.  \\- \left. (i1-1i) \star_{\succ \{\sigma \alpha_i, B_P\}} \Lambda_{\phi_4}
(B_{P,\sigma \alpha_i}^l)\sha \Lambda_{\phi_{\ofv}} (B_{P, \sigma
  \alpha_i}^r)  \right) \;, \eas where
$\{P_d^0, \ldots P_d^{|d|}\}$ decorate the vertexes of $T_{\phi_4, d}(P)$,
for $d \subset fv(P)$.

\begin{proof}[Proof of theorem \ref{graftedpolygons}]
Let $c_r$ be an admissible dissection of $A_{P, \alpha_i-1}^r$ in
$\Lambda_{\phi_4}$ and $c_l$ of $A_{P, \alpha_i-1}^l$ in
$\Lambda_{\phi_{fv}}$. By definition \ref{LambdaB} and lemma
\ref{graftingcoprod}, \begin{multline*} \Delta \left( i(n+1)-(n+1)i
\star_{\prec \{\alpha_{i-1}, A_P\}} (\Lambda_{\phi_4}
(A_{P, \alpha_{i-1}}^l)\sha \Lambda_{\phi_{fv}} (A_{P,
  \alpha_{i-1}}^r) )\right) = \end{multline*} \begin{multline} \sum_{c_r, c_l}
\Lambda_{\phi_4}(R_{c_r}(A_{P, \alpha_{i-1}}^r)) \sha
\Lambda_{\phi_{fv}}(R_{c_l}(A_{P, \alpha_{i-1}}^l)) \\ \otimes i (n+1)-(n+1)i
\star_{\prec \{\alpha_{i-1}, A_P\}}
\left(\Lambda_{\phi_4}(L_{c_r}(A_{P, \alpha_{i-1}}^r)) \sha
\Lambda_{\phi_{fv}}(L_{c_l}(A_{P, \alpha_{i-1}}^l)) \right) \label{Tterms}
\end{multline}
\begin{multline} + \sum_{c_r, c_l}
i(n+1)-(n+1)i \star_{\prec \{\alpha_{i-1}, A_P\}}
  \left(\Lambda_{\phi_4}(R_{c_r}(A_{P, \alpha_{i-1}}^r)) \sha
  \Lambda_{\phi_{fv}}(R_{c_l}(A_{P, \alpha_{i-1}}^l)) \right)\\ \otimes
  \Lambda_{\phi_4}(L_{c_r}(A_{P, \alpha_{i-1}}^r)) \sha
  \Lambda_{\phi_{fv}}(L_{c_l}(A_{P, \alpha_{i-1}}^l))
  \label{cterms}\end{multline}
\begin{multline}  + \sum_{c_r, c_l}\left(i(n+1)-(n+1)i
\star_{\prec v_{r,\alpha_{i-1}}}
\Lambda_{\phi_4}(R_{c_r}(A_{P, \alpha_{i-1}}^r))\right) \sha
\Lambda_{\phi_{fv}}(R_{c_l}(A_{P, \alpha_{i-1}}^l)) \\ \otimes 
\Lambda_{\phi_4}(L_{c_r}(A_{P, \alpha_{i-1}}^r)) \sha
\Lambda_{\phi_{fv}}(L_{c_l}(A_{P, \alpha_{i-1}}^l))  \label{cmterms1}\end{multline}
\begin{multline} + \sum_{c_r, c_l}
\Lambda_{\phi_4}(R_{c_r}(A_{P, \alpha_{i-1}}^r)) \sha
\left(i(n+1)-(n+1)i \star_{\prec v_{l,\alpha_{i-1}}}
\Lambda_{\phi_{fv}}(R_{c_l}(A_{P, \alpha_{i-1}}^l)) \right)\\ \otimes
\Lambda_{\phi_4}(L_{c_r}(A_{P, \alpha_{i-1}}^r)) \sha
\Lambda_{\phi_{fv}}(L_{c_l}(A_{P, \alpha_{i-1}}^l))
\;. \label{cmterms2}\end{multline}

There is a similar expression for \bas \Delta \left( (i1-1i)
\star_{\succ \{\sigma\alpha_{i}, B_P\}} \Lambda_{\phi_{\ofv}}
(B_{P, \alpha_{i}}^l)\sha \Lambda_{\phi_4} (B_{P, \alpha_{i}}^r)
\right) \;.\eas This proof proceeds by comparing the coproduct of both
sides of \eqref{fv-ofv}. In fact, I only consider the coproduct of
terms involving the inserted sum $i(n+1) -(n+1)i$. The arguments for
$i1-1i$ are similar, and not done here.

This proof proceeds by induction. Notice from example \ref{2&3gons}
that this theorem holds for $n=2$, with $P= 123$. By expression
\eqref{onlyphi4} of lemma \ref{phi4-fv} \bas (\Lambda_{\phi_4} (P) -
\Lambda_{\phi_{fv}} (P)) - (\Lambda_{\phi_{4}} (\sigma P) -
\Lambda_{\phi_{\ofv}} (\sigma P)) = 23 \sha 12 - 21 \sha 32 \;. \eas
Comparing this to the expression \eqref{3gonphi4} in example
\ref{2&3gons} gives \begin{multline*} \Lambda_{\phi_{fv}} (P) -
\Lambda_{\phi_{\ofv}} (\sigma P) = \\ [123] -[231] + [13-31 - 12+21| 23]
+ [21 | 13-31 - 23 + 32]\;.\end{multline*} In this case, the polygons $A = 23$,
and $B =21$. Recall that modulo $\tilde J_1$,  $13-31 - 12+21 =23 -32$ and  $13-31 - 23 + 32 = 12-21$. Therefore, under the quotient map $\tilde r_n$, \bas
\Lambda_{\phi_{fv}} (P) - \Lambda_{\phi_{\ofv}} (\sigma P) =  [23-32|
  23] - [21 | 21-12]\;,\eas as desired.  Suppose equation
\eqref{fv-ofv} holds for all polygons of weight $m$ for $m < n$.

Consider a general polygon $P$ of weight $n$. Let $c$ be an admissible
dissection of $P$ in $\phi_{fv}$ as in corollary
\ref{fv-ofvcoprod}. Since the generator $T_{\phi_{fv}, fv(P)}(P)$ is
linear, $c$ contains at most one arrow in $fv(P)$. Therefore, there
are two cases to consider: $c \cap fv(P) = \emptyset$, and $|c \cap
fv(P)| = 1$.

If $c \cap fv(P) = \emptyset$, there is only one label of
$T_{\phi_{fv}, c}(P)$ which inherits its first vertex from $P$,
$P_c^\rt$. By lemma \ref{fv-ofvcuts}, the generators
$T_{\phi_{fv},c}(P)$ and $T_{\phi_{\ofv},\sigma c}(\bP)$ are identical
after replacing $P_c^\rt$ with $\sigma P_c^\rt$. The dissection $c$
contains either an arrow ending on the root side, $\gamma_c(n+1)$, or
one ending on the first side $\gamma_c(1)$, but not both.  If
$P_c^\rt$ is a root (resp. leaf), label of $T_{fv, c}(P)$, $c$
contains $\gamma_c(n+1)$ (resp. $\gamma_c(1)$). Either arrow may be
trivial. Write $c' = c \setminus \{\gamma_c(n+1) , \gamma_c(1)\}$ and
\bas P_{c'}^\rt = 1 a_2 \ldots a_l n+1\eas the polygon associated to
the dissection $c'$ that inherits the first vertex of $P$. Write $a_1
= 1 $, $c_1 = c' \cup \gamma_c(a_1)$, $a_{l+1} = n+1$, $c_{l+1} = c'
\cup \gamma_c(a_{l+1})$. Define polygons $Q_i$ and $T_i$ such that
$Q_i = (P_{c'}^\rt)_{\gamma_c(a_i)}^l$, and $T_i =
(P_{c'}^\rt)_{\gamma_c(a_i)}^r$ for $i \in \{ 1 , l+1\}$. Write the
polygon \bas \Q_{l+1} = 1 a_2 \ldots a_{q} (n+1)\;.\eas for later use,
define a family of arrows \ba \gamma_c(a_m) = {_{a_{q}+1}\alpha_{a_m}}
\; . \label{gammadef} \ea By construction $P_c^\rt \in \{Q_i,
T_i\}$. If $\gamma(a_{l+1})$ is trivial, $a_{q} = n$ and  $P_c^\rt= Q_{l+1}$
decorates a root vertex of $T_{\phi_{fv},c}(P)$, and $T_{l+1}$ a
trivial polygon. If $\gamma(a_1)$ is trivial, $a_1 = 1$ and $P_c^\rt = T_1$ is
a leaf polygon of $T_{\phi_{fv},c}(P)$, and $Q_1$ a trivial polygon.

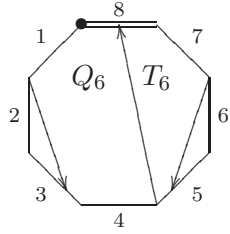
\bas
\xy 
\POS(0,0)  \ar@{-} +(10,0)_4
\POS(10,0) \ar@{-} +(7,7)_5
\POS(17,7) \ar@{-} +(0,10)_6
\POS(17,17) \ar@{-} +(-7,7)_7
\POS(10,24) \ar@{=} +(-10,0)_8
\POS(0,24) \ar@{-} +(-7,-7)_1 \POS(0,24)*+{\bullet}
\POS(-7,17) \ar@{-} +(0,-10)_2
\POS(-7,7) \ar@{-} +(7,-7)_3
\POS(-7,17) \ar@{->} +(5,-15)
\POS(10,0)  \ar@{->} +(-5,24)
\POS(17,17)  \ar@{->} +(-5,-15)
\POS(1,17) *{Q_6}
\POS(10,17) *{T_6}
\endxy  
\quad  Q_6=1348, \;T_6 = 578 \; \;, P_{c'}^\rt = 134578
\eas

Let $\{R_1, \ldots R_l, Q_i, T_i, L_1, \ldots L_n\}$ be the labels of
$T_{fv, c_i}(P)$ for $i \in \{1, l+1\}$, with the $R_k$ and $L_j$
corresponding to root and leaf labels respectively.  By lemma
\ref{fv-ofvcuts}, I write \begin{multline} \Delta_{c_1} +
  \Delta_{c_{l+1}}\left(\Lambda_{\phi_{fv}} (P) -
  \Lambda_{\phi_{\ofv}} (\bP)\right)= \\ \sha_k \Lambda_{\phi_4}
  (R_k) \sha \Lambda_{\phi_{fv}}(Q_1) \otimes(\Lambda_{\phi_{fv}} (T_1)
  -\Lambda_{\phi_{\ofv}} (\sigma T_1)) \sha_j \Lambda_{\phi_4} (L_j )
  \\ +  (\Lambda_{\phi_{fv}} (Q_{l+1}) -\Lambda_{\phi_{\ofv}} (\sigma
  Q_{l+1})) \sha_k \Lambda_{\phi_4} (R_k ) \otimes \sha_j
  \Lambda_{\phi_4} (L_j ) \sha
  \Lambda_{\phi_{\sigma fv}}(\sigma T_{l+1}) \label{ccoprod} \;.\end{multline}

Consider the case of admissible cuts such that $c \cap fv(P) = 1$.
Specifically, consider admissible dissections of $P$ in
$\Lambda_{fv}(P)$ of the form $c_m = c' \cup \alpha_{a_m} ,
\gamma_c(a_m)$ for $m \in \{2 \ldots l\}$ (where $\gamma_c(a_m)$
defined as in equation \eqref{gammadef}). For $d \supseteq
\{\alpha_{a_m} , \gamma_c(a_m)\}$, let the polygons $Q_m$, $T_m$ and
$S_m$ decorate the vertexes of $T_{\phi_{fv}, d}(P_{c'}^\rt)$ such
that $Q_m\prec T_m$ in $T_{\phi_{fv},d}(P)$ are adjascent to the arrow
$\alpha_{a_m}$ and $Q_m\prec S_m$, with both adjascent to the arrow
$\gamma_c(\alpha_m)$. The polygon $S_m$ decorates the remaining
vertex. If $m = q+1$ (or $q$), then $S_m$ is trivial, and $T_m =
T_{l+1}$ (or $Q_m = Q_1$).

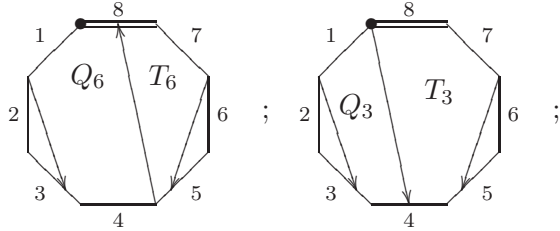
\xy 
\hskip 80pt
\POS(0,0)  \ar@{-} +(10,0)_4
\POS(10,0) \ar@{-} +(7,7)_5
\POS(17,7) \ar@{-} +(0,10)_6
\POS(17,17) \ar@{-} +(-7,7)_7
\POS(10,24) \ar@{=} +(-10,0)_8
\POS(0,24) \ar@{-} +(-7,-7)_1 \POS(0,24)*+{\bullet}
\POS(-7,17) \ar@{-} +(0,-10)_2
\POS(-7,7) \ar@{-} +(7,-7)_3
\POS(-7,17) \ar@{->} +(5,-15)
\POS(10,0)  \ar@{->} +(-5,24)
\POS(17,17)  \ar@{->} +(-5,-15)

\POS(11,17) *{T_6}
\POS(1,17) *{Q_6}

\hskip 70pt  \POS(0,12) *{;}  \hskip 40pt

\POS(0,0)  \ar@{-} +(10,0)_4
\POS(10,0) \ar@{-} +(7,7)_5
\POS(17,7) \ar@{-} +(0,10)_6
\POS(17,17) \ar@{-} +(-7,7)_7
\POS(10,24) \ar@{=} +(-10,0)_8
\POS(0,24) \ar@{-} +(-7,-7)_1 \POS(0,24)*+{\bullet}
\POS(-7,17) \ar@{-} +(0,-10)_2
\POS(-7,7) \ar@{-} +(7,-7)_3
\POS(-7,17) \ar@{->} +(5,-15)
\POS(17,17)  \ar@{->} +(-5,-15)
\POS(0,24)  \ar@{->} +(5,-24)

\POS(9,15) *{T_3}
\POS(-2,13) *{Q_3}

\hskip 70pt  \POS(0,12) *{;}  \hskip 40pt

\endxy
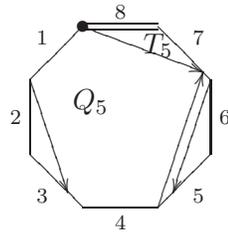
\xy
\hskip 80pt
\POS(0,0)  \ar@{-} +(10,0)_4
\POS(10,0) \ar@{-} +(7,7)_5
\POS(17,7) \ar@{-} +(0,10)_6
\POS(17,17) \ar@{-} +(-7,7)_7
\POS(10,24) \ar@{=} +(-10,0)_8
\POS(0,24) \ar@{-} +(-7,-7)_1 \POS(0,24)*+{\bullet}
\POS(-7,17) \ar@{-} +(0,-10)_2
\POS(-7,7) \ar@{-} +(7,-7)_3
\POS(-7,17) \ar@{->} +(5,-15)
\POS(10,0)  \ar@{->} +(6,18)
\POS(17,17)  \ar@{->} +(-5,-15)
\POS(0,24)  \ar@{->} +(16,-6)

\POS(10,21.5) *{T_5}
\POS(1,14) *{Q_5}

\hskip 90pt

\POS(11,18) *{Q_3=134 \quad T_3=4578 = T} \POS(15,14) *{Q_5=1347 \quad
  T_5=78 \quad S_5 = 57} \endxy
Write \begin{multline} \sum_{m =
    2}^{l} \Delta_{c_m} \left(\Lambda_{\phi_{fv}} (P)) -
  \Lambda_{\phi_{\ofv}} (\sigma P)\right)= - \sum_{m = 2}^{l}
  \\ \sha_k \Lambda_{\phi_4} (R_k ) \sha
  \Lambda_{\phi_{fv}}(Q_m) \otimes (\Lambda_{\phi_{fv}} (T_m)
  -\Lambda_{\phi_{\ofv}} (\sigma T_m)) \sha_j \Lambda_{\phi_4} (L_j )
  \\ + \sha_k \Lambda_{\phi_4} (R_k ) \sha ( \Lambda_{\phi_{fv}}(Q_m)
  - \Lambda_{\phi_{\ofv}}(\sigma Q_m)) \otimes \Lambda_{\phi_{\ofv}}
  (\sigma T_m) \sha_j \Lambda_{\phi_4} (L_j ) \label{cmcoprod}\end{multline}
where the negative sign comes from $\sgn_{fv}(c_m)$ and the fact that
$c_m \cap fv(P) = \alpha_m$. In this expression, I have included the
term $\Lambda_{\phi_4}(S_m)$ in the set $\{\Lambda_{\phi_4}(R_k)\}$ if
$m \leq q$ and in the set $\{\Lambda_{\phi_4}(L_k)\}$ if $m > q$.

The set of non-trivial admissible dissections of $P$ in
$\phi_{fv}$ can be partitioned into sets of the form
$\{c_i\}_{i=1}^{l+1}$. For the remainder of the proof, I calculate the
contribution to the coproduct from $\sum_{m =
  1}^{l+1}\Delta_{c_i}$. The result can be derived by summing over all
such subsets.

From expression \eqref{mixed} of
lemma \ref{phi4-fv}, \bas \Lambda_{\phi_4}(T_{l+1}) -
\Lambda_{\phi_{\ofv}}(T_{l+1}) =
\sum_{m=q+2}^{l}\Lambda_{\phi_{\ofv}}(T_m)\sha \Lambda_{\phi_4}(S_m)
\;.\eas Since $T_{q+1} = T_{l+1}$, and $S_{q+1}$ is trivial, \ba
\Lambda_{\phi_4}(T_{l+1}) = \sum_{m=q+1}^{l}\Lambda_{\phi_{\ofv}}(T_m)\sha
\Lambda_{\phi_4}(S_m)\;.\label{subT}\ea Similarly, \ba
\Lambda_{\phi_4}(Q_1) = \sum_{m=2}^{q}\Lambda_{\phi_{fv}}(Q_m)\sha
\Lambda_{\phi_4}(S_m)\;. \label{subQ}\ea Inserting equations \eqref{subT} and \eqref{subQ} into
\eqref{ccoprod} gives \begin{multline} (\Delta_{c_1} + \Delta_{c_{l
      +1}})(\Lambda_{\phi_{fv}}(P) - \Lambda_{\phi_{\ofv}}(\bP)) =
  \\ \sha_k \Lambda_{\phi_4} (R_k )\sha \left(\Lambda_{\phi_{fv}} (Q_{l+1})
  - \Lambda_{\phi_{\ofv}} (\sigma Q_{l+1}) \right) \otimes \\ \sum_{m=q+1}^{l}
  [\Lambda_{\phi_{\ofv}} (\sigma T_m) \sha \Lambda_{\phi_4} (S_m)]\sha_j
  \Lambda_{\phi_4} (L_j )  \\+ \sha_k \Lambda_{\phi_4} (R_k )\sha
    \sum_{m=2}^{q} [\Lambda_{\phi_{fv}} (Q_m) \sha \Lambda_{\phi_4}
    (S_m)] \otimes  \\ \left(\Lambda_{\phi_{fv}} (T_1) -
  \Lambda_{\phi_{\ofv}} (\sigma T_1) \right)  \sha_j \Lambda_{\phi_4}
  (L_j ) \label{shufflesubbed}\end{multline}

Combining \eqref{shufflesubbed} with \eqref{cmcoprod} gives an expression
for $\sum_{m=1}^{l+1}\Delta_{c_m} (\Lambda_{\phi_{fv}}(P) -
\Lambda_{\phi_{\ofv}}(\bP)) = $ \begin{multline} \sum_{m=q+1}^{l}
  \left(\Lambda_{\phi_{fv}} (Q_{l+1}) - \Lambda_{\phi_{\ofv}} (\sigma
  Q_{l+1}) - \Lambda_{\phi_{fv}}(Q_m) + \Lambda_{\phi_{\ofv}}(\sigma
  Q_m) \right) \\ \sha_k \Lambda_{\phi_4} (R_k ) \otimes
  [\Lambda_{\phi_{\ofv}} (T_m) \sha \Lambda_{\phi_4} (S_m)]\sha_j
  \Lambda_{\phi_4} (L_j ) \\ - \sum_{m=2}^{q} \sha_i \Lambda_{\phi_4}
  (R_i ) \sha ( \Lambda_{\phi_{fv}}(Q_m) -
  \Lambda_{\phi_{\ofv}}(\sigma Q_m)) \sha \Lambda_{\phi_4}(S_m)
  \\ \otimes \Lambda_{\phi_{\ofv}} (T_m) \sha_j \Lambda_{\phi_4} (L_j
  )) \label{rootterms}\end{multline}
\begin{multline} 
+ \sum_{m=2}^{q} \sha_k \Lambda_{\phi_4} (R_k )\sha
    [\Lambda_{\phi_{fv}} (Q_m) \sha
    \Lambda_{\phi_4} (S_m)] \otimes \\ \left(\Lambda_{\phi_{fv}} (T_1) - \Lambda_{\phi_{\ofv}} (\sigma T_1) - \Lambda_{\phi_{fv}} (T_m) + \Lambda_{\phi_{\ofv}} (\sigma T_m)
  \right) \sha_j \Lambda_{\phi_4} (L_j
  )  \\- \sum_{m=q+1}^{l}  \sha_i \Lambda_{\phi_4} (R_i )
  \sha \Lambda_{\phi_{fv}}(Q_m)  \\ \otimes  \Lambda_{\phi_4}(S_m) \sha
  (\Lambda_{\phi_{fv}} (T_m) -\Lambda_{\phi_{\ofv}} (\sigma T_m))
  \sha_j \Lambda_{\phi_4} (L_j )
\label{leafterms}
\end{multline}

Consider the sum in \eqref{rootterms}. The following arguments are similar for
the terms in \ref{leafterms}.  By induction,
\begin{multline} \tilde r_n\sum_{m=q+1}^{l} \left(\Lambda_{\phi_{fv}} (Q_{l+1}) -
\Lambda_{\phi_{\ofv}} (\sigma Q_{l+1}) - \Lambda_{\phi_{fv}} (Q_m) +
\Lambda_{\phi_{\ofv}} (\sigma Q_m)\right) = \\ \tilde
r_n\sum_{m=q+1}^{l}  \sum_{i=2}^{q} a_i(n+1)-(n+1)a_i \star_{\prec \{\alpha_{i-1}, A_Q\}}(\Lambda_{\phi_4} (A_{Q_m,\alpha_{i-1}}^r)
\sha \Lambda_{\phi_{fv}} (A_{Q_m,\alpha_{i-1}}^l)) -
\\  a_ia_m-a_ma_i \star_{\prec \{\alpha_{i-1},
  A_{Q_m}\}}(\Lambda_{\phi_4}
(A_{Q_m,\alpha_{i-1}}^r)\sha\Lambda_{\phi_{fv}}
(A_{Q_m,\alpha_{i-1}}^l)) \;. \label{Q-Qm}\end{multline} Since
\bas B_Q = B_{Q_m}= a_2 \ldots a_q 1\eas for all $j$, all terms
involving $a_i1-a_i1$ cancel. Similarly \begin{multline} \tilde
  r_n \left(\sum_{m=2}^{q} \left(\Lambda_{\phi_{fv}} (Q_m) -
  \Lambda_{\phi_{\ofv}} (\sigma Q_m))\right) \right)= \\ \tilde r_n
  \left( \sum_{m=2}^{q} \sum_{i=2}^{m-1} a_ia_m-a_ma_i \star_{\prec
    \{\alpha_{i-1}, A_{Q_m}\}}(\Lambda_{\phi_4}
  (A_{Q_m,\alpha_{i-1}}^r)\sha \Lambda_{\phi_{fv}}
  (A_{Q_m,\alpha_{i-1}}^l)) - \right. \\  \left.  a_i1-1a_i \star_{
    \prec \{\sigma \alpha_{i}, B_{Q_m}\}}(\Lambda_{\phi_4}
  (B_{Q_m,\sigma \alpha_i}^l) \sha \Lambda_{\phi_{\ofv}} (B_{Q_m,
    \sigma \alpha_i}^r))  \;. \right) \label{Qm}\end{multline}

Use the coideal $\tilde J_1$ to rewrite \bas a_ia_m-a_ma_i =
a_i(n+1)-(n+1)a_i - a_m(n+1)+(n+1)a_m \;.\eas Inserting this into the
expressions \eqref{Q-Qm} and \eqref{Qm} and substituting into
\eqref{rootterms} gives an expression for $\sum_{m=2}^{l+1} \Delta
_{c_m} \Lambda_{\phi_{fv}}(P) - \Lambda_{\ofv}(\bP)$ \begin{multline}
 \sum_{ i =2}^{q} a_i(n+1)-(n+1)a_i \star_{\prec
    \{\alpha_{i-1}, A_{Q_{l+1}}\}}    (\Lambda_{\phi_4}
  (A_{Q_{l+1},\alpha_{i-1}}^r)\sha\Lambda_{\phi_{fv}}
  (A_{Q_{l+1},\alpha_{i-1}}^l)) \\  \sha_k(\Lambda_{\phi_4}R_k)  \otimes
  \Lambda_{\phi_4}(T_{l+1}) \sha_j (\Lambda_{\phi_4}L_j)
   \label{part1} \end{multline}
\begin{multline} + \sum_{m=2}^{q} \sum_{i=2}^{m-1} 
a_i1-1a_i \star_{\succ 
  \{\alpha_i,B_{Q_{m}}\}}(\Lambda_{\phi_4} (B_{Q_{m}, \alpha_i}^l) \sha \Lambda_{\phi_{\ofv}} (B_{Q_{m}, \alpha_i}^r))
\\ \sha_k(\Lambda_{\phi_4}R_k) \otimes  \Lambda_{\phi_{\ofv}} (T_m)
 \sha_j \Lambda_{\phi_4} (L_j ))
 \label{part2} \end{multline} \begin {multline} - \sum_{m= 2}^{l}
  \sum_{i = 2}^{\textrm{min}(m-1,q)} a_i(n+1)-(n+1)a_i
  \star_{\prec \{\alpha_{i-1},
    A_{Q_m}\}}(\Lambda_{\phi_4}
  (A_{Q_m,\alpha_{i-1}}^r)\sha\Lambda_{\phi_{fv}}
  (A_{Q_m,\alpha_{i-1}}^l)) \\\sha_i(\Lambda_{\phi_4}R_i) \otimes
  \Lambda_{\phi_{\ofv}} (T_m) \sha_j
  \Lambda_{\phi_4} (L_j )
  \label{part3} \end{multline} \begin{multline} + \sum_{m= 2}^{l}
  a_m(n+1)-(n+1)a_m
  \star_{\prec A_{Q_m}^\rt} (\Lambda_{\phi_4}
  (A_{Q_m})) \\\sha_i(\Lambda_{\phi_4}R_i) \otimes
  \Lambda_{\phi_{\ofv}} (T_m)  \sha_j
  \Lambda_{\phi_4} (L_j )) \;. \label{part4} \end{multline}

The expression for \eqref{part1} is derived with the aid of equation
\eqref{subT}. Line \eqref{part2} is derived from line \eqref{Qm}. In
lines \eqref{part2}, \eqref{part3} and \eqref{part4} the $S_m$ terms
  are absorbed in to the set of $R_i$ (if $m \leq q$) or $L_i$ (if $m >
  q$). Finally, line \eqref{part4} is derived from statement
  \eqref{mixed} of lemma \ref{phi4-fv}, namely that \bas
  \Lambda_{\phi_4}(A_{Q_m}) = \sum_{i=2}^{\textrm{weight }
    Q_m}(\Lambda_{\phi_4}
  (A_{Q_m,\alpha_{i-1}}^r)\sha\Lambda_{\phi_{fv}}
  (A_{Q_m,\alpha_{i-1}}^l))\;. \eas

I use this lines \eqref{part1}, \eqref{part2}, \eqref{part3} and
\eqref{part4} to compare \eqref{rootterms} to the coproduct of the
terms involving $i(n+1)-(n+1)i$ in equation \eqref{fv-ofv}. Line
\eqref{part2} I ignore as it contributes to the coproduct of terms
involving $a_i1-1a_i$. Instead, I consider \begin{multline} -
  \sum_{m=q+1}^{l} \sum_{i=m+1}^{l} \sha_i(\Lambda_{\phi_4}R_i) \sha
  \Lambda_{\phi_{\ofv}} (Q_m) \otimes \sha_j \Lambda_{\phi_4} (L_j ))
  \sha \\ a_i(n+1)-a_i(n+1)\star_{\prec
    \{\alpha_{i-1},A_{T_m}\}}(\Lambda_{\phi_4} (A_{T_m,
    \alpha_{i-1}}^l)\sha \Lambda_{\phi_{fv}} (A_{T_m,
    \alpha_{i-1}}^r))
   \label{part2'} \end{multline} which comes from expression
\eqref{leafterms}.

Notice that $\{c_m|1 \leq m \leq l +1\}$ are admissible dissections
of $P$ in $\phi_{fv}$. They need not be admissible dissections of $A_P$
or $B_P$ in either $\phi_4$ of $\phi_{fv}$. However, each can be
partitioned into admissible dissections of the subpolygons $A_{P,
  \alpha_{i-1}}^l$ and $A_{P, \alpha_{i-1}}^r$. Line \eqref{part1}
corresponds to the admissible cut $c_{l+1}$. This can
be partitioned $c_{l+1} = c_{A_{P, \alpha_{a_i-1}}^r} \cup c_{A_{P,
    \alpha_{a_i-1}}^l}$ for $2\leq i \leq  q$, where $c_{A_{P,
    \alpha_{a_i-1}}^r}$ is an admissible dissection of $A_{P,
  \alpha_{a_i-1}}^l$ in $\Lambda_{\phi_{fv}}$ and $c_{A_{P,
    \alpha_{a_i-1}}^r}$ is an admissible dissection of $A_{P,
  \alpha_{i-1}}^r$ in $\Lambda_{\phi_4}$. It cannot be partitioned
admissible dissections of $B_{P, \alpha_{a_i}}^r$ and $B_{P, \alpha_{a_i}}^l$
for any $i$, as the arrow ${_{\alpha_q+1}}\alpha \in c$ is not in
$D(B_P)$.  For instance, for $i=3$, 
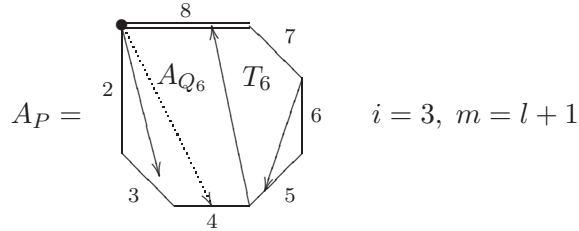
\bas \xy \POS(-17,12) *{A_P=} 
\POS(0,0)  \ar@{-} +(10,0)_4
\POS(10,0) \ar@{-} +(7,7)_5
\POS(17,7) \ar@{-} +(0,10)_6
\POS(17,17) \ar@{-} +(-7,7)_7
\POS(10,24) \ar@{=} +(-17,0)_8
\POS(-7,24)*+{\bullet}
\POS(-7,24) \ar@{-} +(0,-17)_2
\POS(-7,7) \ar@{-} +(7,-7)_3
\POS(-7,24) \ar@{->} +(5,-20)
\POS(-7,24) \ar@{.>} +(12,-24)
\POS(10,0)  \ar@{->} +(-5,24)
\POS(17,17)  \ar@{->} +(-5,-15)
\POS(11,17) *{T_6}
\POS(1,17) *{A_{Q_6}}  
\POS(40 , 12) *{i= 3, \; m= l+1}\endxy 
\;.\eas
The terms in \eqref{part1} correspond to terms in \eqref{cterms} with
$c_r = c_{A_{P, \alpha_{a_i-1}}^r}$ and $c_l = c_{A_{P,
    \alpha_{a_i-1}}^l}$.

Consider the admissible dissections contributing to
\eqref{part2'}. Here $i > m> q$, therefore $\alpha_{a_m} \in D(B_{P,
  \alpha_{a_i}}^r)$.  \bas 
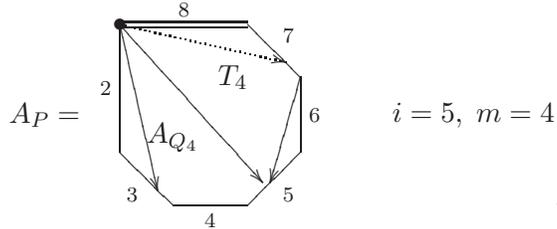
\xy \POS(-17,12) *{A_P=} \POS(0,0) \ar@{-}
+(10,0)_4 \POS(10,0) \ar@{-} +(7,7)_5 \POS(17,7) \ar@{-} +(0,10)_6
\POS(17,17) \ar@{-} +(-7,7)_7 \POS(10,24) \ar@{=} +(-17,0)_8
\POS(-7,24)*+{\bullet} \POS(-7,24) \ar@{-} +(0,-17)_2 \POS(-7,7)
\ar@{-} +(7,-7)_3 \POS(-7,24) \ar@{->} +(5,-22) \POS(-7,24) \ar@{.>}
+(22,-5) \POS(-7,24) \ar@{->} +(19,-21) \POS(17,17) \ar@{->} +(-4,-14)
\POS(8,17) *{T_4} \POS(0,9) *{A_{Q_4}} \POS(40 , 12) *{i= 5, \; m=
  4}\endxy .  \eas The terms of \eqref{part2'} correspond to terms in
\eqref{Tterms}, with $c_r = c_{A_{P, \alpha_{a_i}}^r}$ and $c_l =
c_{A_{P, \alpha_i}^l}$; $ c_m = c_{A_{P, \alpha_{a_i}}^r} \cup c_{A_{P,
    \alpha_i}^l}$.

For $2 \leq m \leq l$, and $2 < i < \textrm{wt }Q_m$,
\bas 
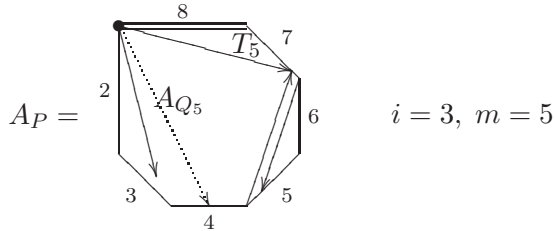
\xy \POS(-17,12) *{A_P=}
\POS(0,0)  \ar@{-} +(10,0)_4
\POS(10,0) \ar@{-} +(7,7)_5
\POS(17,7) \ar@{-} +(0,10)_6
\POS(17,17) \ar@{-} +(-7,7)_7
\POS(10,24) \ar@{=} +(-17,0)_8
\POS(-7,24)*+{\bullet}
\POS(-7,24) \ar@{-} +(0,-17)_2
\POS(-7,7) \ar@{-} +(7,-7)_3
\POS(-7,24) \ar@{->} +(5,-20)
\POS(-7,24) \ar@{.>} +(12,-24)
\POS(10,0)  \ar@{->} +(6,18)
\POS(17,17)  \ar@{->} +(-5,-15)
\POS(-7,24)  \ar@{->} +(23,-6)
\POS(10,21.5) *{T_5}
\POS(1,14) *{A_{Q_5}}
\POS(40 , 12) *{i= 3, \; m= 5}\endxy 
\eas
the terms of \eqref{part3} corresponds to terms of \eqref{cterms},
with $c_r = c_{A_{P, \alpha_{i-1}}^r}$ and $c_l = c_{A_{P,
    \alpha_{i-1}}^l} \cup \alpha_{a_m}$. For $i = 2$ the subpolygon
$(A_{Q_m})_{\alpha_{a_2-1}}^r$ is trivial. The admissible dissection
can be written $c_m = c_{A_{P, \alpha_{i-1}}^r} \cup c_{A_{P,
    \alpha_{i-1}}^l} \cup \alpha_{a_2}$. In this case, the terms of
\eqref{part3} corresponds to terms of \eqref{cmterms1}, with $c_r =
c_{A_{P, \alpha_{a_i-1}}^r}$ and $c_l = c_{A_{P, \alpha_{a_i-1}}^l}$.

Finally, the terms of \eqref{part4} corresponds to terms of
\eqref{cmterms1}, with $c_r = c_{A_{P, \alpha_{a_m-1}}^r}$ and $c_l =
c_{A_{P, \alpha_{a_m-1}}^l}$. 
\bas 
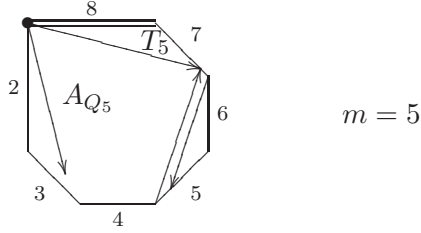
\xy
\POS(0,0)  \ar@{-} +(10,0)_4
\POS(10,0) \ar@{-} +(7,7)_5
\POS(17,7) \ar@{-} +(0,10)_6
\POS(17,17) \ar@{-} +(-7,7)_7
\POS(10,24) \ar@{=} +(-17,0)_8
\POS(-7,24)*+{\bullet}
\POS(-7,24) \ar@{-} +(0,-17)_2
\POS(-7,7) \ar@{-} +(7,-7)_3
\POS(-7,24) \ar@{->} +(5,-20)
\POS(10,0)  \ar@{->} +(6,18)
\POS(17,17)  \ar@{->} +(-5,-15)
\POS(-7,24)  \ar@{->} +(23,-6)
\POS(10,21.5) *{T_5}
\POS(1,14) *{A_{Q_5}}
\POS(40 , 12) *{ m= 5}\endxy 
 \eas

Varying the dissections $c$ and associated $c_m$ account for all terms in the
expressions \eqref{cterms}, \eqref{Tterms}, \eqref{cmterms1}, and
\eqref{cmterms2}, showing that the coproduct of the two sides of
\eqref{fv-ofv} are equal.
\end{proof}
\bibliography{/home/mithu/bibliography/Bibliography}{}

\bibliographystyle{amsplain}

\end{document}